\documentclass[11pt]{article}
\usepackage{amssymb}
\usepackage{amsmath}
\usepackage{mathrsfs}
\usepackage{graphics}
\usepackage{graphicx}
\usepackage{xcolor}
\usepackage{subfigure}
\usepackage[T1]{fontenc}
\usepackage{latexsym,amssymb,amsmath,amsfonts,amsthm}\usepackage{txfonts}
\numberwithin{equation}{section}
\newcommand{\be}{\begin{eqnarray}}
\newcommand{\ben}{\begin{eqnarray*}}
\newcommand{\en}{\end{eqnarray}}
\newcommand{\enn}{\end{eqnarray*}}

\newtheorem{theorem}{Theorem}[section]
\newtheorem{lemma}{Lemma}[section]
\newtheorem{prp}[theorem]{Proposition}
\newtheorem{thm}[theorem]{Theorem}
\newtheorem{cor}[theorem]{Corollary}
\newtheorem{dfn}{Definition}[section]

\newtheorem{remark}{Remark}

\definecolor{rr}{rgb}{0.8,0.2,0}
\definecolor{ts}{rgb}{1,0,0}
\definecolor{dz}{rgb}{0,0,1}
\begin{document}
\renewcommand{\theequation}{\arabic{section}.\arabic{equation}}
\begin{titlepage}
\title{\bf Ergodicity for stochastic conservation laws with multiplicative noise}
\author{ Zhao Dong$^{1}$, Rangrang Zhang$^{2,}\footnote{Corresponding author.}$, Tusheng Zhang$^{3}$\\
{\small $^1$ RCSDS, Academy of Mathematics and Systems Science, Chinese Academy of Sciences, Beijing 100190, China}\\
{\small $^2$  School of Mathematics and Statistics,
Beijing Institute of Technology, Beijing 100081, China}\\
{\small $^3$ School of Mathematics, University of Manchester, Oxford Road, Manchester M13 9PL, England, UK}\\
({\small{\sf dzhao@amt.ac.cn},\ {\sf rrzhang@amss.ac.cn}, \ {\sf tusheng.zhang@manchester.ac.uk}})}
\date{}
\end{titlepage}
\maketitle

\noindent\textbf{Abstract}:
 We proved that there exists a unique invariant measure for solutions of stochastic conservation laws with Dirichlet boundary condition driven by multiplicative noise. Moreover, a polynomial mixing property is established. This is done in the setting of kinetic solutions taking values in an
  $L^1$-weighted space.



\tableofcontents

\section{Introduction}
In this paper, we investigate the long time behaviour of stochastic scalar conservation laws with multiplicative noise. The (deterministic) conservation laws are fundamental to our understanding of the space-time evolution laws of interesting physical quantities. Mathematically or statistically, such physical laws should incorporate with noise influences, due to the lack of knowledge of certain physical parameters as well as bias or incomplete measurements arising in {experiments} or modeling. More precisely, fix any
$T>0$ and let $(\Omega,\mathcal{F},\mathbb{P},\{\mathcal{F}_t\}_{t\in
[0,T]},(\{\beta_k(t)\}_{t\in[0,T]})_{k\in\mathbb{N}})$ be a stochastic basis. Without loss of generality, here the filtration $\{\mathcal{F}_t\}_{t\in [0,T]}$ is assumed to be complete and $\{\beta_k(t)\}_{t\in[0,T]},k\in\mathbb{N}$, are independent (one-dimensional)  $\{\mathcal{F}_t\}_{t\in [0,T]}-$Wiener processes. We use $\mathbb{E}$ to denote the expectation with respect to the probability measure $\mathbb{P}$.
Let $D\subset\mathbb{R}^d$ denote a bounded domain with Lipschitz boundary  $\partial D$.
We are concerned with the following initial-Dirichlet boundary valued problem of the scalar conservation law with stochastic forcing, denoted by $\mathcal{E}(A,\Phi,\vartheta)$:
\begin{eqnarray}\label{P-19}
du+div(A(u))dt=\Phi(u) dW(t) \quad {\text{in}} \ \ D\times(0,T),
\end{eqnarray}
with the initial condition
\begin{eqnarray}\label{P-19-1}
  u(0,\cdot)=\vartheta\quad {\text{in}}\ \ D,
\end{eqnarray}
and the boundary condition
\begin{eqnarray}\label{P-19-2}
  u=0, \quad {\text{on}}\ \  \Sigma.
\end{eqnarray}
Here, $\Sigma=(0,T)\times \partial D$,
$u:(\omega,t, x)\in\Omega\times [0,T]\times D\mapsto u(\omega,t, x):=u(t, x)\in\mathbb{R}$ is a random field, the flux function $A:\mathbb{R}\to\mathbb{R}^d$ and the coefficient $\Phi:\mathbb{R}\to\mathbb{R}$ are measurable and fulfill certain conditions specified later,
and $W$ is a cylindrical Wiener process on a given (separable) Hilbert space $U$ with
the form $W(t)=\sum_{k\geq 1}\beta_k(t) e_k,t\in[0,T]$, where $(e_k)_{k\geq 1}$ is a complete orthonormal basis of the Hilbert space $U$. Set $Q=(0,T)\times D$.

The deterministic conservation laws (i.e., $\Phi\equiv0$ in  (\ref{P-19})) is well studied in the PDEs literature, see e.g. the monograph \cite{Dafermos} and the most recent reference Ammar,
Willbold and Carrillo \cite{K-P-J} (and references therein). As well known, the Cauchy problem
for the deterministic first-order PDE (\ref{P-19}) does not admit any (global) smooth solutions, but there exist infinitely many weak solutions and an additional entropy condition has to be added to get the uniqueness and further to identify the physical weak solution. The notion of entropy solutions for the deterministic problem in the $L^{\infty}$ framework was initiated by Otto in \cite{O}. Porretta and Vovelle \cite{P-V} studied the problem in the $L^1$  setting. To deal with unbounded solutions, the authors of \cite{P-V} defined a notion of renormalized entropy solutions which generalized Otto's original definition of entropy solutions. The kinetic formulation of weak entropy solution of the Cauchy problem for a general multidimensional scalar conservation laws was derived by Lions, Perthame and Tadmor in \cite{L-P-T}. Concerning the initial-boundary problem for deterministic conservation laws, it is crucial to give an interpretation of the boundary condition (\ref{P-19-2}).
In the setting of functions of bounded variation, Bardos, Le Roux and N\'{e}d\'{e}lec \cite{BLN} considered the boundary condition (\ref{P-19-2}) as an ``entropy'' inequality on the boundary $\Sigma$ and obtained the global well-posedness of entropy solutions to (\ref{P-19})-(\ref{P-19-2}). Later, Otto \cite{O} extended it to the $L^{\infty}$ setting by introducing the notion of boundary entropy-flux pairs. Imbert and Vovelle \cite{IV04} derived a kinetic formulation of weak entropy solutions of the initial-boundary value problem and proved the uniqueness of such a kinetic solution.


In recent years, there has been a growing interest in the study of conservation laws driven by stochastic
forcing. Having a stochastic forcing term in (\ref{P-19}) is very natural and important for various modeling problems arising in a wide variety of fields, e.g., physics, engineering, biology and so on.
The Cauchy problem (\ref{P-19}) driven by additive noise has been studied by Kim in \cite{K} wherein the author proposed a method of compensated compactness to prove the existence of a stochastic weak entropy solution via vanishing viscosity approximation. Concerning the case with multiplicative noise, Feng and Nualart \cite{F-N} introduced a notion of strong entropy solutions and established the existence and uniqueness in the one-dimensional case. Using a kinetic formulation, Debussche and Vovelle \cite{D-V-1} solved the stochastic Cauchy problem (\ref{P-19}) with periodic boundary condition in any dimension. Based on  \cite{D-V-1}, Dong et. al. \cite{DWZZ} established small noise large deviations for kinetic solutions of periodic stochastic conservation laws with multiplicative noise. On the other hand, Vallet and Wittbold in \cite{V-W} studied the multi-dimensional Dirichlet boundary value problem for stochastic conservation laws driven by additive noise. For the initial-Dirichlet boundary value problem with multiplicative noise, Bauzet, Vallet and Wittbold \cite{BVW14} established the existence and uniqueness of stochastic entropy solutions when the flux function is assumed to be globally Lipschitz. Recently, Kobayasi and Noboriguchi \cite{KN16} relaxed the condition on the flux function to be of polynomial growth by using kinetic formulation for stochastic conservation laws with nonhomogeneous Dirichlet boundary conditions.

 We remark that there are not many works on the long time behavior/ergodicity of stochastic scalar conservation laws. In the space dimension one, E et. al. \cite{E00} proved the existence and uniqueness of invariant measures for the periodic stochastic inviscid Burgers equation with additive forcing.
 Debussche and Vovelle \cite{D-V-2}  studied scalar conservation laws with additive stochastic forcing on toruses of any dimension and proved the existence and uniqueness of an invariant measure for sub-cubic fluxes and sub-quadratic fluxes, respectively.
Later, Chen and Pang \cite{CP19} extend the result of \cite{D-V-2} to degenerate second-order parabolic-hyperbolic conservation laws driven by additive noise.
 We want to stress that in the above papers, only additive noise was considered and no convergence rate to the invariant measure was obtained.

The purpose of this paper is to obtain the ergodicity and further to establish  the  polynomial
 mixing property for stochastic conservation laws (\ref{P-19})-(\ref{P-19-2}) driven by multiplicative noise. As far as we know, this is the first result for the case of multiplicative noise. Our method is inspired by the work \cite{DGT20} where the authors proved the ergodicity for entropy solutions of
  stochastic porous media equations on smooth bounded domain with Dirichlet boundary conditions.

However, we will work on the setting of kinetic formulation of the solutions. As in \cite{DGT20}, in order to obtain a polynomial rate of convergence to the invariant measure,
we choose to work on a weighted $L^1_{w;x}$ space for a suitable weight function $w$. As an important part of the proof, we apply the doubling variables method in $L^1_{w;x}$ to obtain a ``super $L^1_{w;x}-$contration principle'' for the solutions, that is, there exists an extra strictly negative term on the right hand side of the $L^1_x-$contration principle (see (\ref{e-16-1})), which is the key to obtain the polynomial decay rate. As the invariant measure is living in the  $L^1_x$-space,  we need to
show that the kinetic solution to (\ref{P-19})-(\ref{P-19-2}) has a continuous extension (with respect to the time) in the space $L^1_{\omega}L^1_x$. To do so, we use the vanishing viscosity method to introduce  approximating equations and to overcome difficulties caused by the unboundedness of the flux function. This is quite different from the work \cite{DGT20} where the authors used smooth approximation of the coefficients.
 The Markov semigroup associated with the kinetic solution is defined in the $L^1_x$-space, which is further proved to be Feller. The final step is to show that the solutions of the stochastic conservation laws  converges to a unique stationary solution with a polynomial convergence rate.
The ``super $L^1_{w;x}-$contration principle'' plays a key role.

The rest of the paper is organized as follows.  In Section 2, the mathematical formulation of stochastic scalar conservation laws and some known results are presented. In Section 3, we state our main results. Section 4 is devoted to proving a ``super $L^1_{w;x}-$contration principle'' for the kinetic solutions. The existence of a continuity extension of kinetic solutions is proved in Section 5. In Section 6, we prove that the kinetic solution of the initial-boundary value problem admits a unique invariant measure and satisfies the polynomial mixing property.
\vskip 0.3cm
In the sequel, we use the letter $C$ to denote a generic constant whose values may change from one line to another. Sometimes, we precise its dependence on  parameters.

\section{Preliminaries}
Let $\mathcal{L}(K_1,K_2)$ (resp. $\mathcal{L}_2(K_1,K_2)$) be the space of bounded (resp. Hilbert-Schmidt) linear operators from a Hilbert space $K_1$ to another Hilbert space $K_2$, whose norm is denoted by $\|\cdot\|_{\mathcal{L}(K_1, K_2)}$(resp. $\|\cdot\|_{\mathcal{L}_2(K_1, K_2)})$. Further, $C_b$ represents the space of bounded, continuous functions.
 Let $\|\cdot\|_{L^p_x}$ denote the norm of the $L^p( D)$-space for $p\in (0,\infty]$, where $x$ indicates the name of the variable. In particular, when $p=2$, we set $H=L^2(D)$. For all $a\in \mathbb{R}$ and $p\in (0,\infty]$, denote by $W^{a,p}(D)$ the usual Sobolev space, whose norm is denoted by $\|\cdot\|_{W^{a,p}_x}$. When $p=2$, set
$H^a(D)=W^{a,2}(D)$. Moreover, we use the brackets $\langle\cdot,\cdot\rangle$ to denote the duality between $C^{\infty}_c( D\times \mathbb{R})$ and the space of distributions over $ D\times \mathbb{R}$.
 With a slight abuse of the notation $\langle\cdot,\cdot\rangle$, we set
\[
\langle F, G \rangle:=\int_{ D}\int_{\mathbb{R}}F(x,\xi)G(x,\xi)dxd\xi, \quad F\in L^p( D\times \mathbb{R}), G\in L^q( D\times \mathbb{R}).
\]
for $1\leq p< \infty$ and $q:=\frac{p}{p-1}$, the conjugate exponent of $p$. In particular, when $p=1$, we set $q=\infty$ by convention.

For a measure $m$ on the Borel measurable space $D\times[0,T]\times \mathbb{R}$, the shorthand $m(\phi)$ is defined by
\[
m(\phi):=\langle m, \phi \rangle([0,T]):=\int_{D\times[0,T]\times \mathbb{R}}\phi(x,t,\xi)dm(x,t,\xi), \quad  \phi\in C_b(D\times[0,T]\times \mathbb{R}).
\]
Define
\begin{eqnarray}\label{eee-12}
 w(x)=-(x_1+x_2+\dots+x_d)+C_0,
\end{eqnarray}
where $C_0$ is a constant bigger than $\max_{x\in D}(x_1+x_2+\dots+x_d)$ so that
 $w(x)>0$ in $D$.
Let $L^1_{w;x}$ be the space of all measurable functions $f:  D\rightarrow \mathbb{R}$ such that
\[
\|f\|_{L^1_{w;x}}:=\int_{  D}|f(x)|w(x)dx<\infty.
\]
Clearly, $L^1_{w;x}$ is equivalent to $L^1_x$.

To end this subsection, we mention some notations related to the predictability. For a stochastic basis $(\Omega, \mathcal{F}, \{\mathcal{F}_t\}_{t\in
[0,T]}, \mathbb{P})$ and $p\in [1,\infty)$, we denote by $L^p_{\omega}$ the space of $p-$integrable random variables in $\omega\in \Omega$.
For $T>0$, let $\mathcal{B}([0,T])$ be the Borel $\sigma-$algebra on $[0,T]$ and denote by $\mathcal{P}_T\subset \mathcal{B}([0,T])\otimes \mathcal{F}$ the predictable $\sigma-$algebra.
Let $L^{p}_{\omega}L^q_x$ stand for the space of $p-$integrable random variables taking values in $L^q_x$ and $L^p_{\omega;t}L^q_x$ represent the set of functions $v\in L^p(\Omega\times [0,T];L^q_x)$ which are equal $dt\times \mathbb{P}-$almost everywhere to a predictable process $u$, where $dt$ is the Lebesuge measure on $[0,T]$.


\subsection{Hypotheses}
For the initial value $\vartheta$, the flux function $A$, and the coefficient $\Phi$ of (\ref{P-19})-(\ref{P-19-2}), we introduce the following hypotheses.
\begin{description}
  \item[(\textbf{H1})]
  The flux function $A\in C^2(\mathbb{R};\mathbb{R}^d)$. Each component $A_j$ is differentiable, strictly increasing and odd. The derivative $a_j=A'_j\geq0$ has at most polynomial growth. That is, there exist constants $C>0$ and $q_0\geq 1$ such that
     \begin{eqnarray}\label{qeq-22}
     \sum^d_{j=1} |a_j(\xi)-a_j(\zeta)|\leq C(1+|\xi|^{q_0-1}+|\zeta|^{q_0-1})|\xi-\zeta|.
      \end{eqnarray}
     Moreover, assume
      \begin{eqnarray}\label{eqq-36}
       \sum^d_{j=1} |A_j(u)-A_j(v)|\geq C_{q_0}|u-v|^{q_0+1}, \quad {\text{for}}\ \ u,v\in \mathbb{R}.
      \end{eqnarray}

\item[(\textbf{H2})]
The initial value $\vartheta\in L^{q}_{\omega}L^{q}_{x}$ for some $q> 2(q_0+1)$, which is an $\mathcal{F}_0\otimes \mathcal{B}(D)-$measurable random variable.
  \item[(\textbf{H3})] For each $u\in \mathbb{R}$, the map $\Phi(u): U\rightarrow H$ is defined by $\Phi(u) e_k=g_k(\cdot, u)$, where $g_k(\cdot,u)$ is a regular function on $ D$.
      More precisely, we assume that $g_k\in C(D\times \mathbb{R})$ satisfying
\begin{eqnarray}\label{equ-28}
G^2(x,u)=\sum_{k\geq 1}|g_k(x,u)|^2&\leq& C_0(1+|u|^2),\\
\label{equ-29}
\sum_{k\geq 1}|g_k(x,u)-g_k(y,v)|^2&\leq& C_0\Big(|x-y|^2+{|u-v|^2}\Big),
\end{eqnarray}
for some constant $C_0>0$ and $x, y\in  D, u,v\in \mathbb{R}$.
Since $\|g_k\|_{H}\leq C\|g_k\|_{C( D)}$, we deduce that $\Phi(u)\in \mathcal{L}_2(U,H)$, for each $u\in \mathbb{R}$.

\end{description}


To deduce the $L^1-$theory of (\ref{P-19})-(\ref{P-19-2}), we need a stronger condition than (H3) on $\Phi$:
\begin{description}
  \item[(\textbf{H4})] There exist constants $C_k$ such that
  \begin{eqnarray}
  |g_k(x,u)|\leq C_k(1+|u|), \quad \sum_{k\geq 1}C^2_k<\infty,
  \end{eqnarray}
   and (\ref{equ-29}) remain unchanged.
\end{description}
\begin{remark}
  The set of $A_i$ satisfying Hypothesis (H1) is not empty, e.g. taking $A_i(u)=\frac{1}{d}|u|^{q_0}u$ with an even integer $q_0$.
Moreover, the condition (\ref{eqq-36}) shows that there exists at least one non-zero component of $A(u)$, which satisfies the non-degeneracy condition required by Theorem 1 in \cite{D-V-2}.
\end{remark}

\begin{remark}
 The condition $q> 2(q_0+1)$ in (H2) is required to apply the generalized It\^{o} formula from Proposition A.1 in \cite{DHV}, see the proof of Theorem \ref{thm-10}.
\end{remark}

\subsection{Kinetic solution}
 We  follow closely the framework of \cite{KN16,KN18}.
Firstly, the domain $D$ can be localized by the following method:
choosing a finite open cover $\{U_{i}\}_{i=0,\dots, M}$ of $\bar{D}$ and a partition of unity $\{\lambda_i\}_{i=0,\dots, M}$ on $\bar{D}$ subject to $\{U_{i}\}_{i=0,\dots, M}$ such that $U_{0}\cap \partial D=\emptyset$, for $i=1,\dots, M$,
\begin{eqnarray*}
  D^{i}:=D\cap U_{i}=\Big\{x\in U_{i}; (\mathcal{A}_i x)_d>h_{i}(\overline{\mathcal{A}_i x})\Big\},\quad
  \partial D^{i}:=\partial D\cap U_{i}=\Big\{x\in U_{i}; (\mathcal{A}_i x)_d=h_{i}(\overline{\mathcal{A}_i x})\Big\},
\end{eqnarray*}
with a Lipschitz function $h_{i}: \mathbb{R}^{d-1}\rightarrow \mathbb{R}$, where $\mathcal{A}_i$ is an orthogonal matrix corresponding to a change of coordinates of $\mathbb{R}^d$ and $\bar{y}=(y_1, \dots, y_{d-1})$ for $y\in \mathbb{R}^d$. In order to emphasize the correspondence between $U_{i}$ and $\lambda_i$, we denote by $U_{\lambda_i}=U_{i}, D^{\lambda_i}=D^{i}, h_{\lambda_i}=h_{i}$.

For the sake of simplicity, we will drop the index $i $ of $\lambda_i$ and suppose that the matrix $\mathcal{A}_i=I$.
Moreover, we set
\[Q^{\lambda}=(0,T)\times D^{\lambda},\quad \Sigma^{\lambda}=(0,T)\times \partial D^{\lambda}, \quad \Pi^{\lambda}=\{\bar{x}; x\in U_{\lambda}\}.
\]
Denote by $L_{\lambda}$ the Lipschitz constant of $h_{\lambda}$ on $\Pi^{\lambda}$ and set $L:=\sum^{M}_{i=0}L_{\lambda_i}$.

To regularize functions that are defined on $D^{\lambda}$ and $\mathbb{R}$, let us consider a standard modifier $\psi$ on $\mathbb{R}$, that is, $\psi$ is a nonnegative and even function in $C^{\infty}_c((-1,1))$ with $\int_{\mathbb{R}}\psi=1$. We set
\begin{eqnarray*}
 \rho^{\lambda}(x)=\Pi^{d-1}_{i=1} \psi(x_i)\psi(x_d-(L_{\lambda}+1)),
\end{eqnarray*}
for $x=(x_1,\dots, x_d)$. For $\gamma, \delta>0$, we set
$\rho^{\lambda}_{\gamma}(x)=\frac{1}{\gamma^d}\rho^{\lambda}\Big(\frac{x}{\gamma}\Big)$ and $\psi_{\delta}(\xi)=\frac{1}{\delta}\psi\Big(\frac{\xi}{\delta}\Big)$.

 Recall that we are working on the stochastic basis $(\Omega,\mathcal{F},\mathbb{P},\{\mathcal{F}_t\}_{t\in [0,T]},(\beta_k(t))_{k\in\mathbb{N}})$.
\begin{dfn}(Kinetic measure)\label{dfn-3}
 A map $m$ from $\Omega$ to $\mathcal{M}^+_0( D\times [0,T)\times \mathbb{R})$, the set of non-negative finite measures over $ D\times [0,T)\times\mathbb{R}$, is said to be a kinetic measure if
\begin{description}
  \item[1.] $ m $ is weakly measurable, that is, for each $\phi\in C_b( D\times [0,T)\times \mathbb{R}), \langle m, \phi \rangle: \Omega\rightarrow \mathbb{R}$ is measurable,
  \item[2.] $m$ vanishes at infinity, i.e.,
\begin{eqnarray}\label{equ-37}
\lim_{R\rightarrow +\infty}\mathbb{E}[m( D\times [0,T)\times B^c_R)]=0,\quad B^c_R:=\{\xi\in \mathbb{R}; |\xi|\geq R\},
\end{eqnarray}
  \item[3.] for every $\phi\in C_b( D\times \mathbb{R})$, the process
\[
(\omega,t)\in\Omega\times[0,T)\rightarrow \int_{ D\times [0,t]\times \mathbb{R}}\phi(x,\xi)dm(x,s,\xi)\in\mathbb{R}\quad {\text{is\ predictable}}.
\]
\end{description}
\end{dfn}

\begin{dfn}(Kinetic solution)\label{dfn-1}
Let $\vartheta\in L^{q}_{\omega}L^{q}_x.$ A measurable function $u: \Omega\times  [0,T]\times D\rightarrow \mathbb{R}$ is called a kinetic solution to (\ref{P-19})-(\ref{P-19-2}) with datum $\vartheta$ if
\begin{description}
  \item[1.] $u\in L^q_{\omega;t}L^q_{x}$ and for any $1\leq p\leq q$, there exists $C_p\geq0$ such that
\begin{eqnarray}\label{eqq-31}
\mathbb{E}\ \underset{0\leq t\leq T}{{\rm{ess\sup}}}\ \|u(t)\|^p_{L^p_x}\leq C_p,
\end{eqnarray}
\item[2.] there exists a kinetic measure $m$ and for any $N>0$, there exist nonnegative functions $\bar{m}^{\pm}_N\in L^1(\Omega \times \Sigma\times (-N,N))$ such that $\{\bar{m}^{\pm}_N(t)\}$ are predictable,
\begin{eqnarray*}
    \underset{\xi\uparrow N}{{\rm{\lim}}}\ \bar{m}^{+}_N(t,x,\xi)=\underset{\xi\downarrow -N}{{\rm{\lim}}}\ \bar{m}^{-}_N(t,x,\xi)=0,
   \end{eqnarray*}
for all $\varphi\in C^{\infty}_c([0,T)\times \bar{D}\times (-N,N))$, $f:=I_{u>\xi}$ satisfies
\begin{eqnarray}\notag
&&\int^T_0\langle f(t), \partial_t \varphi(t)\rangle dt+\langle f_0, \varphi(0)\rangle
 +\int^T_0\langle f(t), a(\xi)\cdot \nabla \varphi (t)\rangle dt+M_N\int_{\Sigma\times \mathbb{R}}f_b\varphi d\xi d\sigma(x)dt\\ \notag
&=& -\sum_{k\geq 1}\int^T_0\int_{ D} g_k(x,u(t,x))\varphi (t,x,u(t,x))dxd\beta_k(t)\\ \notag
&& -\frac{1}{2}\int^T_0\int_{ D}\partial_{\xi}\varphi (t,x,u(t,x))G^2(x,u(t,x))dxdt
 \\ \label{P-21}
&&+m(\partial_{\xi} \varphi)+\int_{\Sigma\times \mathbb{R}}\partial_{\xi}\varphi\bar{m}^{+}_Nd\xi d\sigma(x)dt, \quad a.s. ,
\end{eqnarray}
and $\bar{f}:=1-f=I_{u\leq\xi}$ satisfies
\begin{eqnarray}\notag
&&\int^T_0\langle \bar{f}(t), \partial_t \varphi(t)\rangle dt+\langle \bar{f}_0, \varphi(0)\rangle +\int^T_0\langle \bar{f}(t), a(\xi)\cdot \nabla \varphi (t)\rangle dt+M_N\int_{\Sigma\times \mathbb{R}}\bar{f}_b\varphi d\xi d\sigma(x)dt\\ \notag
&=& \sum_{k\geq 1}\int^T_0\int_{ D} g_k(x,u(t,x))\varphi (t,x,u(t,x))dxd\beta_k(t) \\ \notag
&& +\frac{1}{2}\int^T_0\int_{D}\partial_{\xi}\varphi (t,x,u(t,x))G^2(x,u(t,x))dxdt\\ \label{P-21-1}
&&-m(\partial_{\xi} \varphi)-\int_{\Sigma\times \mathbb{R}}\partial_{\xi}\varphi\bar{m}^{-}_Nd\xi d\sigma(x)dt, \quad a.s. ,
\end{eqnarray}
 where $a(\xi):=A'(\xi)$, $G^2=\sum^{\infty}_{k=1}|g_k|^2$, $M_N=\underset{-N\leq \xi\leq N}{\max}\ |a(\xi)|$, $f_0=I_{\vartheta>\xi}$ and $f_b=I_{0>\xi}$.
\end{description}
\end{dfn}
\begin{remark}
 The boundary function $\bar{m}^{\pm}$ does not appear in the case of the periodic boundary condition. In this case, it is enough to consider the equality (\ref{P-21}) for $f$ (the equality satisfied by $\bar{f}$ can be derived from (\ref{P-21})). However, in the case of the Dirichlet boundary conditions, the boundary functions $\bar{m}^{+}$ and $\bar{m}^{-}$ are different from each other, thus, we need to consider both (\ref{P-21}) and (\ref{P-21-1}).
\end{remark}
We need the following definition.
\begin{dfn}(Young measure)
 Let $(X,\lambda)$ be a finite measure space and $\mathcal{M}_1(\mathbb{R})$ be the set of all (Borel) probability measures on $\mathbb{R}$. A map $\nu:X\to\mathcal{M}_1(\mathbb{R})$ is
said to be a Young measure on $X$, if for each $\phi\in C_b(\mathbb{R})$, the map $z\in X\mapsto \nu_z(\phi)\in\mathbb{R}$ is measurable. We say that a Young measure $\nu$ vanishes at infinity if, for each $1\leq p\leq q$,
\begin{eqnarray}\label{equ-26}
\int_X\int_{\mathbb{R}}|\xi|^pd\nu_z(\xi)d\lambda(z)<+\infty.
\end{eqnarray}

\end{dfn}

Let $(X,\lambda)$ be a finite measure space. For some measurable function $u: X\rightarrow \mathbb{R}$, define $f: X\times \mathbb{R}\rightarrow [0,1]$ by $f(z,\xi)=I_{u(z)>\xi}$ a.e. and
we use $\bar{f}:=1-f$ to denote its conjugate function. Define $\Lambda_f(z,\xi):=f(z,\xi)-I_{0>\xi}$, which can be viewed as a correction to $f$. Note that $\Lambda_f$ is integrable on $X\times \mathbb{R}$ if $u$ is.


Define two non-increasing functions $\mu_m(\xi)$ and $\mu_{\nu}(\xi)$ on $\mathbb{R}$ by
\begin{eqnarray}\label{a-2}
  \mu_m(\xi)&=&\mathbb{E}m([0,T)\times D\times (\xi, \infty) ),\\
\label{eqq-20}
  \mu_{\nu}(\xi)&=&\mathbb{E}\int_{(0,T)\times D\times (\xi, \infty)}d \nu_{t,x}(\zeta)dxdt.
\end{eqnarray}
where $m$ is a kinetic measure and $\nu$ is a Young measure satisfying (\ref{equ-26}). Let $\mathbb{D}$ be the set of $\xi\in (0,\infty)$ such that both of $\mu_m$ and $\mu_{\nu}$ are differentiable at $-\xi$ and $\xi$. Clearly, $\mathbb{D}$ is a full set in $(0,+\infty)$. Denote by $\mu'_{m}$ and $\mu'_{\nu}$ the derivatives of $\mu_m$ and $\mu_{\nu}$, respectively. It was shown in Lemma 2 of \cite{KN16} that
\begin{lemma}\label{lem-4}
\begin{description}
  \item[(i)] For any $0\leq p\leq q$,
  \begin{eqnarray*}
    \limsup_{\xi\rightarrow \infty,\xi\in \mathbb{D}}\mu'_m(\pm \xi)=0, \quad \limsup_{\xi\rightarrow \infty,\xi\in \mathbb{D}}\xi^p\mu'_{\nu}(\pm \xi)=0.
  \end{eqnarray*}

  \item[(ii)] If $N\in \mathbb{D}$, then as $\delta\downarrow 0$,
  \begin{eqnarray*}
    \int_{\mathbb{R}}\psi_{\delta}(N\pm \zeta)d\mu_m(\zeta)\rightarrow \mu'_m(\mp N), \quad \int_{\mathbb{R}}\psi_{\delta}(N\pm \zeta)(1+|\zeta|^2)d\mu_{\nu}(\zeta)\rightarrow (1+N^2)\mu'_{\nu}(\mp N).
  \end{eqnarray*}
\end{description}

\end{lemma}

It is shown in \cite{D-V-1} that for each kinetic solution $u$, almost surely the function $f=I_{u(x,t)>\xi}$ admits left and right weak limits at any point $t\in[0,T]$. More precisely, the following results are obtained.
\begin{prp}(Left and right weak limits)\label{prp-3} Let $u$ be a kinetic solution to (\ref{P-19})-(\ref{P-19-2}).
Then $f=I_{u>\xi}$ admits, almost surely, left and right limits respectively at every point $t\in [0,T]$. More precisely, for any  $t\in [0,T]$, there exist functions $f^{t\pm}$ on $\Omega\times  D\times \mathbb{R}$ such that $\mathbb{P}-$a.s.
\begin{eqnarray*}
\langle f(t-\varepsilon),\varphi\rangle\rightarrow \langle f^{t-},\varphi\rangle
\end{eqnarray*}
and
\begin{eqnarray*}
\langle f(t+\varepsilon),\varphi\rangle\rightarrow \langle f^{t+},\varphi\rangle
\end{eqnarray*}
as $\varepsilon\rightarrow 0$ for all $\varphi\in C^1_c( D\times \mathbb{R})$. Moreover, almost surely,
\[
\langle f^{t+}-f^{t-}, \varphi\rangle=-\int_{ D\times[0,T]\times \mathbb{R}}\partial_{\xi}\varphi(x,\xi)I_{\{t\}}(s)dm(x,s,\xi).
\]
In particular, almost surely, the set of $t\in [0,T]$ fulfilling that $f^{t+}\neq f^{t-}$ is countable.
\end{prp}
For the function $f=I_{u>\xi}$, we set $f^{\pm}(t)=f^{t \pm}$, $t\in [0,T]$. Since we are dealing with the filtration associated to Brownian motion, both $f^{\pm}$ are  clearly predictable as well. Also $f=f^+=f^-$ almost everywhere in time and we can take any of them in an integral with respect to the Lebesgue measure or in a stochastic integral.

Due to Proposition \ref{prp-3}, the weak form (\ref{P-21})-(\ref{P-21-1}) satisfied by a kinetic solution can be strengthened to be weak only respect to $x$ and $\xi$.
In order to state it, we need to introduce the cutoff function, for any $0<\eta<N$,
\begin{eqnarray*}
  \Psi_{\eta}(\xi)=\int^{\xi}_{-\infty}(\psi_{\eta}(\zeta+N-\eta)-\psi_{\eta}(\zeta-N+\eta))d\zeta.
\end{eqnarray*}
With the help of the cutoff function, the test functions in (\ref{P-21}) and (\ref{P-21-1}) can be extended  to the class of functions in $C^{\infty}_{c}([0,T)\times \mathbb{R}^d\times \mathbb{R} )$. The following result was proved in \cite{KN16}.
\begin{prp}\label{prp-7}
  Assume $u$ is a kinetic solution of (\ref{P-19})-(\ref{P-19-2}). Set $f:=I_{u>\xi}$ and define
   \begin{eqnarray*}
  f^{\lambda,\gamma}(t,x,\xi)=\int_{D^{\lambda}}f(t,y,\xi)\rho^{\lambda}_{\gamma}(y-x)dy,
  \end{eqnarray*}
 for any element $\lambda$ of the partition of unity $\{\lambda_i\}$ on $\bar{D}$.
 Let $\tilde{f}^{(\lambda)}$ be any weak star limit of $\{f^{\lambda,\gamma}\}$ as $\gamma\rightarrow 0$ in $L^{\infty}(\Omega\times\Sigma^{\lambda}\times \mathbb{R})$
  and define $\tilde{f}:=\sum^M_{i=0}\lambda_i\tilde{f}^{(\lambda_i)}$. Then
  \begin{description}
    \item[(i)] for any $\varphi\in C^{\infty}_c(\mathbb{R}^d\times \mathbb{R})$, $t\in [0,T), \eta>0$ and $N>0$, the function $f=I_{u>\xi}$ satisfies
        \begin{eqnarray}\notag
          &&-\int_D\int^N_{-N}\Psi_{\eta}f^+(t)\varphi d\xi dx+\int^t_0\int_D\int^N_{-N}\Psi_{\eta}f a(\xi)\cdot \nabla \varphi d\xi dxds\\ \notag
          && +\int_D\int^N_{-N}\Psi_{\eta}f_0\varphi d\xi dx+\int^t_0\int_{\partial D}\int^N_{-N}\Psi_{\eta}(-a(\xi)\cdot \mathbf{n}) \tilde{f}\varphi d\xi d\sigma ds\\ \notag
          &=& -\sum_{k\geq 1}\int^t_0\int_D\int^N_{-N}\Psi_{\eta}g_k(x,\xi)\varphi d\nu_{x,s}(\xi)dxd\beta_k(s)\\ \notag
          &&-\frac{1}{2}\int^t_0\int_D\int^N_{-N}\Psi_{\eta}\partial_{\xi}\varphi G^2(x,\xi)d\nu_{x,s}(\xi)dxds+\int_{[0,t]\times D\times (-N,N)}\Psi_{\eta}\partial_{\xi}\varphi dm\\ \notag
          &&
          +\frac{1}{2}\int^t_0\int_D\int^N_{-N}\Big(\psi_{\eta}(N-\xi-\eta)-\psi_{\eta}(N+\xi-\eta)\Big)G^2(x,\xi)\varphi d\nu_{s,x}(\xi)dxds\\ \label{eq-1}
          && -\int_{[0,t]\times D\times (-N,N)}\Big(\psi_{\eta}(N-\xi-\eta)-\psi_{\eta}(N+\xi-\eta)\Big)\varphi dm \quad a.s.
        \end{eqnarray}
and $\bar{f}=I_{u\leq \xi}$ satisfies
\begin{eqnarray}\notag
          &&-\int_D\int^N_{-N}\Psi_{\eta}\bar{f}^+(t)\varphi d\xi dx+\int^t_0\int_D\int^N_{-N}\Psi_{\eta}\bar{f} a(\xi)\cdot \nabla \varphi d\xi dxds\\ \notag
          && +\int_D\int^N_{-N}\Psi_{\eta}\bar{f}_0\varphi d\xi dx+\int^t_0\int_{\partial D}\int^N_{-N}\Psi_{\eta}(-a(\xi)\cdot \mathbf{n}) \tilde{\bar{f}}\varphi d\xi d\sigma ds\\ \notag
          &=& \sum_{k\geq 1}\int^t_0\int_D\int^N_{-N}\Psi_{\eta}g_k(x,\xi)\varphi d\nu_{x,s}(\xi)dxd\beta_k(s)\\ \notag
          &&+\frac{1}{2}\int^t_0\int_D\int^N_{-N}\Psi_{\eta}\partial_{\xi}\varphi G^2(x,\xi)d\nu_{x,s}(\xi)dxds-\int_{[0,t]\times D\times (-N,N)}\Psi_{\eta}\partial_{\xi}\varphi dm\\ \notag
          &&
          -\frac{1}{2}\int^t_0\int_D\int^N_{-N}\Big(\psi_{\eta}(N-\xi-\eta)-\psi_{\eta}(N+\xi-\eta)\Big)G^2(x,\xi)\varphi d\nu_{s,x}(\xi)dxds\\ \label{eq-1-1}
          && +\int_{[0,t]\times D\times (-N,N)}\Big(\psi_{\eta}(N-\xi-\eta)-\psi_{\eta}(N+\xi-\eta)\Big)\varphi dm \quad a.s.
        \end{eqnarray}
where $\nu=-\partial_{\xi}f=\partial_{\xi}\bar{f}=\delta_{u=\xi}$.

    \item[(ii)] $\mathbb{P}-$a.s., for a.e. $(t,x)\in \Sigma$, we have
        \begin{eqnarray}\label{eqq-38}
         -a(\xi)\cdot \mathbf{n}(\bar{x}) \tilde{f}(t,x,\xi)&=&M_Nf_b(t,x,\xi)+\partial_{\xi}\bar{m}^+_N(t,x,\xi), \\
         \label{eqq-39}
         -a(\xi)\cdot \mathbf{n}(\bar{x}) \tilde{\bar{f}}(t,x,\xi)&=&M_N\bar{f}_b(t,x,\xi)-\partial_{\xi}\bar{m}^{-}_N(t,x,\xi)
        \end{eqnarray}
         for a.e. $\xi\in (-N,N)$.

  \end{description}
\end{prp}
We remark that the weak star limit $\tilde{f}^{(\lambda)}$ may depend on the chosen subsequence $\{\gamma_n\}_{n\geq 1}\subset \{\gamma\}_{\gamma>0}$. From now on, when considering $\gamma\rightarrow 0$, we always refer to a subsequence of $\{\gamma_n\}_{n\geq 1}$ converging to $0$.

For any $s\in (0,T)$, replacing the starting point $0$ by $s$, we obtain
\begin{lemma}\label{lem-5}
 For all $0< s<T$ and $\varphi\in C^{\infty}_c(\mathbb{R}^d\times \mathbb{R})$, the function $f=I_{u>\xi}$ associated to the kinetic solution $u$ of (\ref{P-19})-(\ref{P-19-2}) satisfies that
 \begin{eqnarray}\notag
    &&-\int_D\int^N_{-N}\Psi_{\eta}f^+(t)\varphi d\xi dx+\int^t_s\int_D\int^N_{-N}\Psi_{\eta}f a(\xi)\cdot \nabla \varphi d\xi dxdr\\ \notag
          && +\int_D\int^N_{-N}\Psi_{\eta}f^+_s\varphi d\xi dx+\int^t_s\int_{\partial D}\int^N_{-N}\Psi_{\eta}(-a(\xi)\cdot \mathbf{n}) \tilde{f}\varphi d\xi d\sigma dr\\ \notag
          &=& -\sum_{k\geq 1}\int^t_s\int_D\int^N_{-N}\Psi_{\eta}g_k(x,\xi)\varphi d\nu_{x,r}(\xi)dxd\beta_k(r)\\ \notag
          &&-\frac{1}{2}\int^t_s\int_D\int^N_{-N}\Psi_{\eta}\partial_{\xi}\varphi G^2(x,\xi)d\nu_{x,r}(\xi)dxdr+\int_{(s,t]\times D\times (-N,N)}\Psi_{\eta}\partial_{\xi}\varphi dm\\ \notag
          &&
          +\frac{1}{2}\int^t_s\int_D\int^N_{-N}\Big(\psi_{\eta}(N-\xi-\eta)-\psi_{\eta}(N+\xi-\eta)\Big)G^2(x,\xi)\varphi d\nu_{r,x}(\xi)dxdr\\ \label{eee-1}
          && -\int_{(s,t]\times D\times (-N,N)}\Big(\psi_{\eta}(N-\xi-\eta)-\psi_{\eta}(N+\xi-\eta)\Big)\varphi dm \quad a.s.
 \end{eqnarray}
on $[s,T)$.
\end{lemma}
Existence and uniqueness of a solution to (\ref{P-19})-(\ref{P-19-2}) with initial datum $\vartheta\in L^{\infty}(\Omega\times D)$ has been proved by \cite{KN16}. This result can easily be extended to initial data in $L^q_{\omega}L^{q}_x$ for $q$ bigger than the degree of polynomial growth of the flux function $A$ (cf. \cite{KN18}). Precisely the following result can be proved.
\begin{thm}\label{thm-9}
 Under Hypotheses (H1)-(H3), there exists a unique kinetic solution to  (\ref{P-19})-(\ref{P-19-2}) with initial datum $\vartheta\in L^q_{\omega}L^{q}_x$, which has almost surely continuous trajectories in $L^{q}_x$.
\end{thm}

\subsection{Renormalized kinetic solution}
As
 the invariant measures are living in $L^1_x$, in this part, we need to extend the initial data $u_0$ from $L^{q}_x$ to $L^1_x$. This generalization, the so called $L^1$-theory,
  has been done in several papers. The $L^1_x-$theory for
  the periodic scalar conservation laws driven by stochastic forcing was developed in
   \cite{D-V-2}, which generalized the deterministic results established
   by Chen and Perthame \cite{CP03}.
Later, Noboriguchi \cite{Dai-1} developed the $L^1_x-$theory for periodic stochastic scalar conservation laws driven by multiplicative noise.


To state the $L^1_x-$theory of (\ref{P-19})-(\ref{P-19-2}), we shall extend the notion of solutions from kinetic solutions to renormalized kinetic solutions.
Firstly, we need a weak version of kinetic measures.
 \begin{dfn}(Weak kinetic measure)
 A map $m$ from $\Omega$ to $\mathcal{M}^+_0( D\times [0,T)\times \mathbb{R})$ is said to be a weak kinetic measure if
\begin{description}
  \item[1.] $ m $ is weakly measurable,
  \item[2.] $m$ vanishes at infinity in average:
\begin{eqnarray*}
\lim_{R\rightarrow +\infty}\frac{1}{R}\mathbb{E}[m(D\times [0,T)\times \{\xi\in \mathbb{R}; R\leq |\xi|\leq 2R\})]=0,
\end{eqnarray*}
  \item[3.] for every $\phi\in C_b( D\times \mathbb{R})$, the process
\[
(\omega, t)\rightarrow \int_{ D\times [0,t]\times \mathbb{R}}\phi(x,\xi)dm(x,s,\xi)\in\mathbb{R} \quad {\text{is\ predictable}}.
\]
\end{description}
\end{dfn}
The following is a weak version of kinetic solution called renormalized kinetic solution.
\begin{dfn}(Renormalized kinetic solution)\label{dfn-2}
Let $\vartheta \in L^1_x$. A measurable function $u: \Omega\times [0,T]\times D\rightarrow \mathbb{R}$ is called a renormalized kinetic solution to (\ref{P-19})-(\ref{P-19-2}) with datum $\vartheta$, if
\begin{description}
  \item[1.] $u\in L^1_{\omega;t}L^1_x$ and
\begin{eqnarray}
\mathbb{E}\ \underset{0\leq t\leq T}{{\rm{ess\sup}}}\ \|u(t)\|_{L^1_x}<+\infty,
\end{eqnarray}
\item[2.] there exists a weak kinetic measure $m$ and if, for any $N>0$, there exist non-negative functions
 $\bar{m}^{\pm}_N \in L^1(\Omega\times \Sigma\times (-N,N))$ such that $\{\bar{m}^{\pm}_N(t)\}$ are predictable,
 \begin{eqnarray*}
    \underset{\xi\uparrow N}{{\rm{\lim}}}\ \bar{m}^{+}_N(t,x,\xi)=\underset{\xi\downarrow -N}{{\rm{\lim}}}\ \bar{m}^{-}_N(t,x,\xi)=0,
   \end{eqnarray*}
for all $\varphi\in C^{\infty}_c([0,T)\times \bar{D}\times (-N,N))$, $f:=I_{u>\xi}$ satisfies
    \begin{eqnarray}\notag
&&\int^T_0\langle f(t), \partial_t \varphi(t)\rangle dt+\langle f_0, \varphi(0)\rangle +\int^T_0\langle f(t), a(\xi)\cdot \nabla \varphi (t)\rangle dt+M_N\int_{\Sigma\times \mathbb{R}}f_b\varphi d\xi d\sigma(x)dt\\ \notag
&=& -\sum_{k\geq 1}\int^T_0\int_{ D} g_k(x,u(t,x))\varphi (t,x,u(t,x))dxd\beta_k(t) \\ \notag
&& -\frac{1}{2}\int^T_0\int_{ D}\partial_{\xi}\varphi (t,x,u(t,x))G^2(x,u(t,x))dxdt\\ \label{P-22}
&&+\int_{[0,T)\times D\times \mathbb{R}} \partial_{\xi} \varphi dm+\int_{\Sigma\times \mathbb{R}} \partial_{\xi} \varphi \bar{m}^+_Nd\xi d\sigma(x)dt, \quad a.s. ,
\end{eqnarray}
and $\bar{f}:=1-f=I_{u\leq\xi}$ satisfies
\begin{eqnarray}\notag
&&\int^T_0\langle \bar{f}(t), \partial_t \varphi(t)\rangle dt+\langle \bar{f}_0, \varphi(0)\rangle +\int^T_0\langle \bar{f}(t), a(\xi)\cdot \nabla \varphi (t)\rangle dt+M_N\int_{\Sigma\times \mathbb{R}}\bar{f}_b\varphi d\xi d\sigma(x)dt\\ \notag
&=& \sum_{k\geq 1}\int^T_0\int_{ D} g_k(x,u(t,x))\varphi (t,x,u(t,x))dxd\beta_k(t) \\ \notag
&& +\frac{1}{2}\int^T_0\int_{D}\partial_{\xi}\varphi (t,x,u(t,x))G^2(x,u(t,x))dxdt\\ \label{P-23}
&&-\int_{[0,T)\times D\times \mathbb{R}} \partial_{\xi} \varphi dm-\int_{\Sigma\times \mathbb{R}} \partial_{\xi} \varphi \bar{m}^-_Nd\xi d\sigma(x)dt, \quad a.s. .
\end{eqnarray}
\end{description}
\end{dfn}


\section{Statement of main results}
In this section, we state the main results whose proofs are given in Sections 4, 5, and 6.
\subsection{The contraction inequality in the weighted space}
From Theorem \ref{thm-9}, we know that for any initial datum $\vartheta\in L^{q}_{\omega}L^{q}_{x}$, (\ref{P-19})-(\ref{P-19-2}) admits a unique kinetic solution $u(t; \vartheta)\in L^{q}_{\omega}L^{q}_{x}$ for almost all $t\in [0,T)$. Our first result reads as follows.
\begin{thm}\label{thm-8}
Let $u(t;\vartheta)$ and $u(t;\tilde{\vartheta})$ are kinetic solutions of $\mathcal{E}(A,\Phi,\vartheta)$ and $\mathcal{E}(A,\Phi,\tilde{\vartheta})$, respectively. Under Hypotheses (H1)-(H3),
\begin{eqnarray}\label{ee-34}
 \underset{0\leq t\leq T}{{\rm{ess\sup}}}\ \mathbb{E}\|u(t;\vartheta)-u(t;\tilde{\vartheta})\|_{L^1_{w;x}}\leq \mathbb{E}\|\vartheta-\tilde{\vartheta}\|_{L^1_{w;x}}.
\end{eqnarray}
\end{thm}

\subsection{The continuous extension in the weighted space}\label{sec-4}
In this part, we state the various extensions of the  kinetic solutions of (\ref{P-19})-(\ref{P-19-2}).
\begin{thm}\label{thm-10}
Assume Hypotheses (H1)-(H3) hold. If $u$ is a kinetic solution to (\ref{P-19})-(\ref{P-19-2}), then  $u\in C([0,T];L^1_{\omega}L^1_{w;x})$.
\end{thm}
Then,
with the help of (\ref{ee-34}) and Theorem \ref{thm-10}, we derive the following extension of $u$ with respect to the initial value. That is,
\begin{prp}\label{prpo-1}
  Under Hypotheses (H1)-(H3), the mapping
  \[
  L^{q}_{\omega}L^{q}_x\ni \vartheta\mapsto u(\cdot; \vartheta)\in C([0,T];L^{1}_{\omega}L^{1}_{w;x})
  \]
  extends uniquely to a continuous map $v$ from $L^{1}_{w;x}$  to $C([0,T];L^{1}_{\omega}L^{1}_{w;x})$. Furthermore, for all $\vartheta, \tilde{\vartheta}\in L^{1}_{w;x}$,
  \begin{eqnarray}\label{e-43}
   \sup_{t\in [0,T]}\mathbb{E}\|v(t;\vartheta)-v(t;\tilde{\vartheta})\|_{L^1_{w;x}}\leq \mathbb{E}\|\vartheta-\tilde{\vartheta}\|_{L^1_{w;x}}.
  \end{eqnarray}
\end{prp}

  Using a similar method as in the proof of Proposition 3.2 in \cite{Dai-1}, the following can be proved.
\begin{prp}\label{prp-16}
Assume Hypotheses (H1) and (H4) hold. Then the extension $v(t;\vartheta)$ established by Proposition \ref{prpo-1} is the unique renormalized kinetic solution to (\ref{P-19})-(\ref{P-19-2}) on $[0,T]$ in the sense of Definition \ref{dfn-2}.
\end{prp}

\subsection{Ergodicity for renormalized kinetic solutions}
Let $B_b(L^1_{w;x})$ be the space of bounded measurable functions from $L^1_{w;x}$ to $\mathbb{R}$ and $C_b(L^1_{w;x})$
 the space of continuous bounded measurable functions from $L^1_{w;x}$ to $\mathbb{R}$. From Proposition \ref{prp-16}, we know that for any $\vartheta\in L^1_{\omega}L^1_{w;x}$, the extension $v(t;\vartheta)$ defined by Proposition \ref{prpo-1} is the unique renormalized kinetic solution to (\ref{P-19})-(\ref{P-19-2}) on $(0,T]$.
Now, we can define the Markovian semigroup associated with $v(t;\vartheta)$ as follows
\begin{dfn}
For any $t\geq 0$, define $P_t: B_b(L^1_{w;x})\rightarrow B_b(L^1_{w;x})$ by
  \[
  P_tF(\vartheta):=\mathbb{E}F(v(t;\vartheta)), \quad F\in B_b(L^1_{w;x}),\ \vartheta\in L^1_{w;x}.
  \]
\end{dfn}

\begin{prp}\label{prp-4}(Feller)
Assume Hypotheses (H1), (H2) and (H4) are in force. Then the family $(P_t)_{t\geq 0}$ is a Feller semigroup, that is, $P_t$ maps $C_b(L^{1}_{w;x})$ into $C_b(L^{1}_{w;x})$.
\end{prp}

The following result not only reveals the existence and uniqueness of the invariant measures but also provides a mixing rate uniformly with respect to the initial condition.
\begin{thm}\label{thm-3}
Under Hypotheses (H1), (H2) and (H4), there exists a unique invariant measure $\mu\in \mathcal{M}_1(L^1_{w;x})$ for the semigroup $P_t$. Furthermore, there exists $C>0$, depending only on $q_0$, such that for all $t>0$,
\begin{eqnarray}
\sup_{\vartheta\in L^1_{w;x}}\sup_{\|F\|_{Lip(L^1_{w;x})}\leq 1} \left|P_tF(\vartheta)-\int_{L^1_{w;x}}F(\xi)\mu(d\xi)\right|\leq C\|w\|^{q^*}_{L^{q^*}_x}t^{-\frac{1}{q_0}},
\end{eqnarray}
where $q^*=\frac{q_0+1}{q_0}$ and $Lip(L^1_{w;x})$ is the space of Lipschitz continuous functions from $L^1_{w;x}$ to $\mathbb{R}$.
\end{thm}
Let $\mathcal{D}_{\mathbb{P}}(v(t; \vartheta))=\mathbb{P}\circ v(t; \vartheta)^{-1}$ be the law of $v(t; \vartheta)$ under $\mathbb{P}$. In view of the equivalence between $\|\cdot\|_{L^1_{w;x}}$ and $\|\cdot\|_{L^1_{x}}$, using Kantorovich-Rubinstein formula (see Theorem 5.10 in \cite{V08}), it follows immediately that
\begin{cor}
  Under Hypotheses (H1), (H2) and (H4), there exists a unique invariant measure $\mu\in \mathcal{M}_1(L^1_{x})$ for the semigroup $P_t$. Furthermore, there exists $C>0$, depending only on $q_0$, such that for all $t>0$,
\begin{eqnarray}
\sup_{\vartheta\in L^1_{x}} \mathcal{W}_1\Big(\mathcal{D}_{\mathbb{P}}(v(t; \vartheta)),\mu\Big)\leq C\|w\|^{q^*}_{L^{q^*}_x}t^{-\frac{1}{q_0}},
\end{eqnarray}
where $\mathcal{W}_1$ is the $L^1-$Wasserstein distance.
\end{cor}

\section{Proof of the contraction inequality in the weighted space}\label{well-posedness}
In this section, we will prove the contraction inequality (\ref{ee-34}) for kinetic solutions of $\mathcal{E}(A,\Phi,\vartheta)$ and $\mathcal{E}(A,\Phi,\tilde{\vartheta})$. Firstly, we prove a technical proposition using the doubling variables method applied in several papers, e.g. \cite{D-V-1,DWZZ}.
\begin{prp}\label{prp-1} Assume Hypotheses (H1)-(H3) are in force.
Let $u(t;\vartheta)$ and $u(t;\tilde{\vartheta})$ be kinetic solutions of $\mathcal{E}(A,\Phi,\vartheta), \mathcal{E}(A,\Phi,\tilde{\vartheta})$, respectively. Then, for any $0\leq t< T$, $\gamma, \delta>0$, $N\in \mathbb{D}$, and for any element $\lambda$ of the partition of unity $\{\lambda_i\}_{i=0,1,\dots,M}$ on $\bar{D}$, the functions $f_1(t):=f_1(t,x,\xi)=I_{u(t,x;\vartheta)>\xi}$ and $f_2(t):=f_2(t,y,\zeta)=I_{u(t,y;\tilde{\vartheta})>\zeta}$ with data $(f_{i,0}, f_{i,b})$, $i=1,2,$ satisfy
 \begin{eqnarray}\notag
 &&\mathbb{E} \int_{D^{\lambda}_x}\int_{D_y}\int^N_{-N}\int^N_{-N}(f^{\pm}_1(t)\bar{f}^{\pm}_2(t)+\bar{f}^{\pm}_1(t)f^{\pm}_2(t))\rho^{\lambda}_{\gamma}(y-x)\psi_{\delta}(\xi-\zeta)\lambda(x)w(x)d\xi d\zeta dy dx\\ \notag
 &\leq& \mathbb{E} \int_{D^{\lambda}_x}\int_{D_y}\int^N_{-N}\int^N_{-N}(f_{1,0}\bar{f}_{2,0}+\bar{f}_{1,0}f_{2,0})\rho^{\lambda}_{\gamma}(y-x)\psi_{\delta}(\xi-\zeta)\lambda(x)w(x)d\xi d\zeta dy dx\\
\label{e-4}
&& +\sum_{i=1}^5J_i+I_N,
  \end{eqnarray}
  where
\begin{eqnarray*}
J_1&=& \mathbb{E}\int^t_0\int_{\partial D^{\lambda}_x}\int_{D_y}\int^N_{-N}\int^N_{-N}(\tilde{f}^{(\lambda)}_1(s)\bar{f}_2(s)+\tilde{\bar{f}}^{(\lambda)}_1(s)f_2(s))(-a(\xi)\cdot\mathbf{ n})\\
&& \times\rho^{\lambda}_{\gamma}(y-x)\psi_{\delta}(\xi-\zeta)\lambda(x)w(x)d\xi d\zeta dy d\sigma(x) ds,\\
 J_2&=&\mathbb{E}\int^t_0\int_{D^{\lambda}_x}\int_{D_y}\int^N_{-N}\int^N_{-N}( f_1(s)\bar{f}_2(s)+\bar{f}_1(s)f_2(s)) (a(\xi)-a(\zeta))\cdot \nabla_x\rho^{\lambda}_{\gamma}(y-x)\lambda(x)w(x)\psi_{\delta}(\xi-\zeta)d\xi d\zeta dy dxds,\\
J_3&=&\mathbb{E}\int^t_0\int_{D^{\lambda}_x}\int_{D_y}\int^N_{-N}\int^N_{-N} (f_1(s)\bar{f}_2(s)+\bar{f}_1(s)f_2(s)) a(\xi)\cdot \nabla_x\lambda(x) \rho^{\lambda}_{\gamma}(y-x)w(x)\psi_{\delta}(\xi-\zeta)d\xi d\zeta dy dx ds,\\
J_4&=&-\sum^d_{j=1}\mathbb{E}\int^t_0\int_{D^{\lambda}_x}\int_{D_y}\int^N_{-N}\int^N_{-N} (f_1(s)\bar{f}_2(s)+\bar{f}_1(s)f_2(s)) a_j(\xi) \rho^{\lambda}_{\gamma}(y-x)\lambda(x)\psi_{\delta}(\xi-\zeta)d\xi d\zeta dy dx ds,\\
J_5&=&\mathbb{E}\int^t_0\int_{D^{\lambda}_x}\int_{D_y}\int^N_{-N}\int^N_{-N}\lambda(x)w(x)\rho^{\lambda}_{\gamma}(y-x)\psi_{\delta}(\xi-\zeta)
\sum_{k\geq 1}|g_k(x,\xi)-g_k(y,\zeta)|^2d\nu^{1}_{s,x}\otimes \nu^{2}_{s,y}(\xi,\zeta)dy dxds,
\end{eqnarray*}
\begin{eqnarray}\label{eqq-32}
\limsup_{N\rightarrow \infty}I_N=0,
\end{eqnarray}
with $ I_N$ being defined by (\ref{eqq-16}), $f_{1,0}=I_{\vartheta>\xi}$,
$f_{2,0}=I_{\tilde{\vartheta}>\zeta}$, $f_{i,b}=I_{0>\xi}$ for $i=1,2$, $\nu^1_{x,s}(\xi)=\delta_{u(s,x;\vartheta)=\xi}$ and $\nu^2_{y,s}(\zeta)=\delta_{u(s,y;\tilde{\vartheta})=\zeta}$.
\end{prp}

\begin{proof}
 Denote by $m_1$ and $m_2$ the two kinetic measures associated to $\mathcal{E}(A,\Phi,\vartheta)$ and $ \mathcal{E}(A,\Phi,\tilde{\vartheta})$, respectively.
Let $\varphi_1\in C^{\infty}_c( \mathbb{R}^d_x\times \mathbb{R}_{\xi})$ and
$\varphi_2\in C^{\infty}_c( \mathbb{R}^d_y\times \mathbb{R}_{\zeta})$.
Set $\alpha(x,\xi,y,\zeta)=\varphi_1(x,\xi)\varphi_2(y,\zeta)$.

Employing the same method as in \cite{D-V-1} and \cite{KN16},  using (\ref{eq-1})-(\ref{eq-1-1}), we obtain
\begin{eqnarray}\notag
 &&\mathbb{E} \int_{D^{\lambda}_x}\int_{D_y}\int^N_{-N}\int^N_{-N}\Psi_{\eta}(\xi,\zeta)f^+_1(t)\bar{f}^+_2(t)\alpha^{\lambda}w(x)d\xi d\zeta dy dx \\ \notag
 &=&\mathbb{E} \int_{D^{\lambda}_x}\int_{D_y}\int^N_{-N}\int^N_{-N}\Psi_{\eta}(\xi,\zeta)f_{1,0}\bar{f}_{2,0}\alpha^{\lambda}w(x)d\xi d\zeta dy dx\\ \notag
 &&+\mathbb{E}\int^t_0\int_{D^{\lambda}_x}\int_{D_y}\int^N_{-N}\int^N_{-N}\Psi_{\eta}(\xi,\zeta) f_1(s)\bar{f}_2(s) (a(\xi)\cdot \nabla_x+a(\zeta)\cdot \nabla_y) (\alpha^{\lambda}w(x))d\xi d\zeta dy dx ds\\ \notag
&&+\frac{1}{2}\mathbb{E}\int^t_0\int_{D^{\lambda}_x}\int_{D_y}\int^N_{-N}\int^N_{-N}\Psi_{\eta}(\xi,\zeta)\bar{f}_2(s)\partial_{\xi}\alpha^{\lambda} w(x)G^2(x,\xi)d\nu^1_{x,s}(\xi)d\zeta dy dx ds\\ \notag
&&+\mathbb{E}\int^t_0\int_{\partial D^{\lambda}_x}\int_{D_y}\int^N_{-N}\int^N_{-N}\Psi_{\eta}(\xi,\zeta)\tilde{f}^{(\lambda)}_1(s,x,\xi)\bar{f}_2(s)(-a(\xi)\cdot\mathbf{ n})\alpha^{\lambda}w(x)d\xi d\zeta dy d\sigma(x)ds\\ \notag
&&-\mathbb{E}\int_{(0,t]\times D^{\lambda}_x\times (-N,N)}\int_{D_y}\int^N_{-N}\Psi_{\eta}(\xi,\zeta)\bar{f}^+_2(s)\partial_{\xi}\alpha^{\lambda}w(x)d\zeta dydm_1(s,x,\xi)\\ \notag
&&-\frac{1}{2}\mathbb{E}\int^t_0\int_{D^{\lambda}_x}\int_{D_y}\int^N_{-N}\int^N_{-N}\Psi_{\eta}(\zeta)\bar{f}_2(s)\Big(\psi_{\eta}(-\eta+N-\xi)-\psi_{\eta}(-\eta+N+\xi)\Big)\alpha^{\lambda} w(x)G^2(x,\xi)d\nu^1_{s,x}(\xi)d\zeta dy dxds\\ \notag
&&+\mathbb{E}\int_{(0,t]\times D^{\lambda}_x\times (-N,N)}\int_{D_y}\int^N_{-N}\Psi_{\eta}(\zeta)\bar{f}^+_2(s)\Big(\psi_{\eta}(-\eta+N-\xi)-\psi_{\eta}(-\eta+N+\xi)\Big)\alpha^{\lambda} w(x)d\zeta dydm_1(s,x,\xi)\\ \notag
&&-\frac{1}{2}\mathbb{E}\int^t_0\int_{D^{\lambda}_x}\int_{ D_y}\int^N_{-N}\int^N_{-N}\Psi_{\eta}(\xi,\zeta)f_1(s)\partial_{\zeta}\alpha^{\lambda} w(x)G^2(y,\zeta)d\xi d\nu^2_{y,s}(\zeta)dydxds\\ \notag
&&+\mathbb{E}\int^t_0\int_{D^{\lambda}_x}\int_{\partial D_y}\int^N_{-N}\int^N_{-N}\Psi_{\eta}(\xi,\zeta)f_1(s)\tilde{\bar{f}}_2(s,y,\zeta)(-a(\zeta)\cdot \mathbf{n})\alpha^{\lambda} w(x)d\xi d\zeta d \sigma(y) dx ds\\ \notag
&&+\mathbb{E}\int_{(0,t]\times D_y\times (-N,N)}\int_{D^{\lambda}_x}\int^N_{-N}\Psi_{\eta}(\xi,\zeta)f^{-}_1(s)\partial_{\zeta}\alpha^{\lambda} w(x)d\xi dx dm_2(s,y,\zeta)\\ \notag
&&+\frac{1}{2}\mathbb{E}\int^t_0\int_{D^{\lambda}_x}\int_{ D_y}\int^N_{-N}\int^N_{-N}\Psi_{\eta}(\xi)f_1(s)\Big(\psi_{\eta}(-\eta+N-\zeta)-\psi_{\eta}(-\eta+N+\zeta)\Big)\alpha^{\lambda}w(x)G^2(y,\zeta)d\xi d\nu^2_{s,y}(\zeta)dy dx ds\\ \notag
&&-\mathbb{E}\int_{(0,t]\times D_y\times (-N,N)}\int_{D^{\lambda}_x}\int^N_{-N}\Psi_{\eta}(\xi)f^{-}_1(s)\Big(\psi_{\eta}(-\eta+N-\zeta)-\psi_{\eta}(-\eta+N+\zeta)\Big)\alpha^{\lambda}w(x)d\xi dxdm_2(s,y,\zeta)\\ \notag
&&-\sum_{k\geq 1}\mathbb{E}\int^t_0\int_{D^{\lambda}_x}\int_{ D_y}\int^N_{-N}\int^N_{-N}\Psi_{\eta}(\xi,\zeta) g_{k,1}(x,\xi)g_{k,2}(y,\zeta)\alpha^{\lambda}w(x)d\nu^1_{x,s}\otimes \nu^2_{y,s}(\xi,\zeta)dy dxds\\ \label{eqq-14}
&=:&\mathbb{E} \int_{D^{\lambda}_x}\int_{D_y}\int^N_{-N}\int^N_{-N}\Psi_{\eta}(\xi,\zeta)f_{1,0}\bar{f}_{2,0}\alpha^{\lambda}w(x)d\xi d\zeta dy dx+\sum^{12}_{i=1}I_i,
\end{eqnarray}
where $\Psi_{\eta}(\xi,\zeta)=\Psi_{\eta}(\xi)\Psi_{\eta}(\zeta)$ and $\alpha^{\lambda}=\alpha(x,\xi,y,\zeta)\lambda(x)$.
By a density argument, (\ref{eqq-14}) remains true for any test function $\alpha\in C^{\infty}_c( D_x\times \mathbb{R}_\xi\times  D_y\times \mathbb{R}_\zeta)$. Thanks to (\ref{equ-37}) and (\ref{eqq-31}), the assumption that $\alpha$ is compactly supported can be relaxed.
By a truncation argument,
we now will  take $\alpha(x,\xi,y,\zeta)=\rho^{\lambda}_{\gamma}(y-x)\psi_{\delta}(\xi-\zeta)\in C^{\infty}_b( D_x\times \mathbb{R}_\xi\times  D_y\times \mathbb{R}_\zeta)$. In this case, $\alpha^{\lambda}=\lambda(x)\rho^{\lambda}_{\gamma}(y-x)\psi_{\delta}(\xi-\zeta)$. Note that $\rho^{\lambda}_{\gamma}(y-x)$ are equal to 0 on $D^{\lambda}_x\times \partial D_y$, which yields that $I_8=0$. Also we have
\begin{eqnarray}\label{P-11}
(\partial_{\xi}+\partial_{\zeta})\alpha^{\lambda}=0, \quad \nabla_x\rho^{\lambda}_{\gamma}(y-x)=-\nabla_y\rho^{\lambda}_{\gamma}(y-x).
\end{eqnarray}
Since $\partial_{x_i}w(x)=-1$, by (\ref{P-11}), we have
\begin{eqnarray*}
I_1&=&\mathbb{E}\int^t_0\int_{D^{\lambda}_x}\int_{D_y}\int^N_{-N}\int^N_{-N}\Psi_{\eta}(\xi,\zeta) f_1(s)\bar{f}_2(s) (a(\xi)-a(\zeta))\cdot \nabla_x\rho^{\lambda}_{\gamma}(y-x)\lambda(x)w(x)\psi_{\delta}(\xi-\zeta)d\xi d\zeta dy dxds\\
&& +\mathbb{E}\int^t_0\int_{D^{\lambda}_x}\int_{D_y}\int^N_{-N}\int^N_{-N}\Psi_{\eta}(\xi,\zeta) f_1(s)\bar{f}_2(s) a(\xi)\cdot \nabla_x\lambda(x) \rho^{\lambda}_{\gamma}(y-x)w(x)\psi_{\delta}(\xi-\zeta)d\xi d\zeta dy dx ds\\
&& -\sum^d_{j=1}\mathbb{E}\int^t_0\int_{D^{\lambda}_x}\int_{D_y}\int^N_{-N}\int^N_{-N}\Psi_{\eta}(\xi,\zeta) f_1(s)\bar{f}_2(s) a_j(\xi) \rho^{\lambda}_{\gamma}(y-x)\lambda(x)\psi_{\delta}(\xi-\zeta)d\xi d\zeta dy dx ds.
\end{eqnarray*}
Utilizing (\ref{P-11}) again, we deduce that
\begin{eqnarray*}
I_2&=&-\frac{1}{2}\mathbb{E}\int^t_0\int_{D^{\lambda}_x}\int_{D_y}\int^N_{-N}\int^N_{-N}\Psi_{\eta}(\xi,\zeta)\bar{f}_2(s)\partial_{\zeta}\alpha^{\lambda} w(x)G^2(x,\xi)d\nu^1_{x,s}(\xi)d\zeta dy dx ds\\
&=&\frac{1}{2}\mathbb{E}\int^t_0\int_{D^{\lambda}_x}\int_{D_y}\int^N_{-N}\int^N_{-N}\partial_{\zeta}\Big[\Psi_{\eta}(\xi,\zeta)\bar{f}_2(s)\Big]\alpha^{\lambda} w(x)G^2(x,\xi)d\nu^1_{x,s}(\xi)d\zeta dy dx ds\\
&=&\frac{1}{2}\mathbb{E}\int^t_0\int_{D^{\lambda}_x}\int_{D_y}\int^N_{-N}\int^N_{-N}\Psi_{\eta}(\xi,\zeta)\alpha^{\lambda} w(x)G^2(x,\xi)d\nu^1_{x,s}\otimes \nu^{2}_{y,s}(\xi,\zeta) dy dx ds\\
&& -\frac{1}{2}\mathbb{E}\int^t_0\int_{D^{\lambda}_x}\int_{D_y}\int^N_{-N}\int^N_{-N}\Psi_{\eta}(\xi)(\psi_{\eta}(N-\eta-\zeta)-\psi_{\eta}(N-\eta+\zeta))\bar{f}_2(s)\alpha^{\lambda} w(x)G^2(x,\xi)d\nu^1_{x,s}(\xi)d\zeta dy dx ds\\
&=:& I_{2,1}+ I_{2,2}.
\end{eqnarray*}
Similarly, it follows that
\begin{eqnarray*}
I_7&=& -\frac{1}{2}\mathbb{E}\int^t_0\int_{D^{\lambda}_x}\int_{ D_y}\int^N_{-N}\int^N_{-N}\partial_{\xi}[\Psi_{\eta}(\xi,\zeta)f_1(s)]\alpha^{\lambda} w(x)G^2(y,\zeta)d\xi d\nu^2_{y,s}(\zeta)dy dxds\\
&=&\frac{1}{2}\mathbb{E}\int^t_0\int_{D^{\lambda}_x}\int_{ D_y}\int^N_{-N}\int^N_{-N}\Psi_{\eta}(\xi,\zeta)\alpha^{\lambda} w(x)G^2(y,\zeta)d\nu^1_{x,s}\otimes \nu^2_{y,s}(\xi,\zeta)dy dxds\\
&& +\frac{1}{2}\mathbb{E}\int^t_0\int_{D^{\lambda}_x}\int_{ D_y}\int^N_{-N}\int^N_{-N}\Psi_{\eta}(\zeta)(\psi_{\eta}(N-\xi-\eta)-\psi_{\eta}(N+\xi-\eta))f_1(s)\alpha^{\lambda} w(x)G^2(y,\zeta)d\xi d\nu^2_{y,s}(\zeta)dy dxds\\
&=:& I_{7,1}+ I_{7,2}.
\end{eqnarray*}
By integration by parts formula, we have
\begin{eqnarray*}
I_4&=&-\mathbb{E}\int_{[0,t]\times D^{\lambda}_x\times (-N,N)}\int_{D_y}\int^N_{-N}\partial_{\zeta}\Big[\Psi_{\eta}(\xi,\zeta)\bar{f}^+_2(s)\Big]\alpha^{\lambda}w(x)d\zeta dy dm_1(s,x,\xi) \\
&=& \mathbb{E}\int_{[0,t]\times D^{\lambda}_x\times (-N,N)}\int_{D_y}\int^N_{-N}\Psi_{\eta}(\xi)(\psi_{\eta}(N-\zeta-\eta)-\psi_{\eta}(N+\zeta-\eta))\bar{f}^+_2(s)\alpha^{\lambda}w(x)d\zeta dydm_1(s,x,\xi)\\
&& -\mathbb{E}\int_{[0,t]\times D^{\lambda}_x\times (-N,N)}\int_{D_y}\int^N_{-N}\Psi_{\eta}(\xi,\zeta)\alpha^{\lambda}w(x)d\nu^{2,+}_{x,s}(\zeta)dydm_1(s,x,\xi) \\
&\leq& \mathbb{E}\int_{[0,t]\times D^{\lambda}_x\times (-N,N)}\int_{D_y}\int^N_{-N}\Psi_{\eta}(\xi)(\psi_{\eta}(N-\zeta-\eta)-\psi_{\eta}(N+\zeta-\eta))\bar{f}^+_2(s)\alpha^{\lambda}w(x)d\zeta dydm_1(s,x,\xi)\\
&=:& I_{4,1}.
\end{eqnarray*}
Similarly, we can bound $I_9$ as follows
\begin{eqnarray*}
 I_9
 &\leq& -\mathbb{E}\int_{[0,t]\times D_y\times (-N,N)}\int_{D^{\lambda}_x}\int^N_{-N}\Psi_{\eta}(\zeta)(\psi_{\eta}(N-\xi-\eta)-\psi_{\eta}(N+\xi-\eta))
 f^-_1(s)\alpha^{\lambda} w(x)d\xi dx dm_2(s,y,\zeta)\\
&=:& I_{9,1}.
\end{eqnarray*}
By a similar argument as in the proof of Proposition 2 in \cite{KN16} and using Lemma \ref{lem-4}, we obtain
\begin{eqnarray*}
|I_5|&\leq& C\mathbb{E}\int^t_0\int_{D^{\lambda}_x}\int_{D_y}\int^N_{-N}\int^N_{-N}
\Big(\psi_{\eta}(-\eta+N-\xi)+\psi_{\eta}(-\eta+N+\xi)\Big)\\
&& \times\rho^{\lambda}_{\gamma} (y-x)\psi_{\delta}(\xi-\zeta) w(x)(1+|\xi|^2)d\nu^1_{s,x}(\xi)d\zeta dy dx ds\\
&\leq& C\int_{\mathbb{R}}\Big(\psi_{\eta}(-\eta+N-\xi)+\psi_{\eta}(-\eta+N+\xi)\Big)(1+|\xi|^2)\mathbb{E}\int^T_0\int_Dd\nu^1_{s,x}(\xi) dx ds\\
&\leq& C\int_{\mathbb{R}}\Big(\psi_{\eta}(-\eta+N-\xi)+\psi_{\eta}(-\eta+N+\xi)\Big)(1+|\xi|^2)d\mu_{\nu^1}(\xi)\\ &\rightarrow& C(1+N^2)(\mu'_{\nu^1}(N)+\mu'_{\nu^1}(-N)),
\end{eqnarray*}
as $\eta\rightarrow 0+$ , where $\mu_{\nu^1}$ is defined by (\ref{eqq-20}).
Similarly,
\begin{eqnarray*}
|I_6|&\leq&\mathbb{E}\int_{[0,t]\times D^{\lambda}_x\times (-N,N)}\int_{D_y}\int^N_{-N}\Psi_{\eta}(\zeta)\bar{f}^+_2(s)\Big(\psi_{\eta}(-\eta+N-\xi)+\psi_{\eta}(-\eta+N+\xi)\Big)\alpha^{\lambda} w(x)d\zeta dydm_1(s,x,\xi)\\
&\leq& C\mathbb{E}\int_{[0,t]\times D^{\lambda}_x\times (-N,N)}\int_{D_y}\int^N_{-N}(\psi_{\eta}(-\eta+N-\xi)+\psi_{\eta}(-\eta+N+\xi))\rho^{\lambda}_{\gamma} (y-x)\psi_{\delta}(\xi-\zeta)d\zeta dydm_1(s,x,\xi)\\
&\leq& C\int_{\mathbb{R}}(\psi_{\eta}(-\eta+N-\xi)+\psi_{\eta}(-\eta+N+\xi))\mathbb{E}\int^T_0\int_Ddm_1(s,x,\xi)\\
&=& C\int_{\mathbb{R}}(\psi_{\eta}(-\eta+N-\xi)+\psi_{\eta}(-\eta+N+\xi))d\mu_{m_1}(\xi)\\
&\rightarrow& C(\mu'_{m_1}(N)+\mu'_{m_1}(-N)),
\end{eqnarray*}
as $\eta\rightarrow 0+$ , where $\mu_{m_1}$ is defined by (\ref{a-2}).
 By the similar arguments as above, all the  terms  $I_{10}, I_{11}, I_{2,2}, I_{7,2}, I_{4,1}, I_{9,1}$ containing $\psi_{\eta}$, can be estimated from above as $\eta\rightarrow 0+$ by
\begin{eqnarray}\label{eqq-16}
I_N:=C(\mu'_{m_1}(\pm N)+\mu'_{m_2}(\pm N)+(1+N^2)(\mu'_{\nu^1}(\pm N)+\mu'_{\nu^2}(\pm N))).
\end{eqnarray}
Due to Lemma \ref{lem-4}, we see that  $\underset{N\rightarrow \infty}{{\rm{\limsup}}}\ I_N=0$.

Moreover,
 \begin{eqnarray*}
&&I_{12}+ I_{2,1}+I_{7,1}\\
&=&\frac{1}{2}\mathbb{E}\int^t_0\int_{D^{\lambda}_x}\int_{D_y}\int^N_{-N}\int^N_{-N}\lambda(x)w(x)\rho^{\lambda}_{\gamma}(y-x)\psi_{\delta}(\xi-\zeta)
\sum_{k\geq 1}|g_k(x,\xi)-g_k(y,\zeta)|^2d\nu^{1}_{s,x}\otimes \nu^{2}_{s,y}(\xi,\zeta)dy dxds.
\end{eqnarray*}
Combining all the previous estimates and letting $\eta\downarrow0$ in (\ref{eqq-14}), we get
\begin{eqnarray*}
 &&\mathbb{E} \int_{D^{\lambda}_x}\int_{D_y}\int^N_{-N}\int^N_{-N}f^+_1(t)\bar{f}^+_2(t)\alpha^{\lambda}w(x)d\xi d\zeta dy dx\\
 &\leq& \mathbb{E} \int_{D^{\lambda}_x}\int_{D_y}\int^N_{-N}\int^N_{-N}f_{1,0}\bar{f}_{2,0}\alpha^{\lambda}w(x)d\xi d\zeta dy dx \\
 && +\mathbb{E}\int^t_0\int_{\partial D^{\lambda}_x}\int_{D_y}\int^N_{-N}\int^N_{-N}\tilde{f}^{(\lambda)}_1(s)\bar{f}^+_2(s)(-a(\xi)\cdot\mathbf{ n})\alpha^{\lambda}w(x)d\xi d\zeta dy d\sigma(x)ds\\
 &&+\mathbb{E}\int^t_0\int_{D^{\lambda}_x}\int_{D_y}\int^N_{-N}\int^N_{-N} f_1(s)\bar{f}_2(s) (a(\xi)-a(\zeta))\cdot \nabla_x\rho^{\lambda}_{\gamma}(y-x)\lambda(x)w(x)\psi_{\delta}(\xi-\zeta)d\xi d\zeta dxdyds\\
&& +\mathbb{E}\int^t_0\int_{D^{\lambda}_x}\int_{D_y}\int^N_{-N}\int^N_{-N} f_1(s)\bar{f}_2(s) a(\xi)\cdot \nabla_x\lambda(x) \rho^{\lambda}_{\gamma}(y-x)w(x)\psi_{\delta}(\xi-\zeta)d\xi d\zeta  dy dx ds\\
&& -\sum^d_{j=1}\mathbb{E}\int^t_0\int_{D^{\lambda}_x}\int_{D_y}\int^N_{-N}\int^N_{-N} f_1(s)\bar{f}_2(s) a_j(\xi) \rho^{\lambda}_{\gamma}(y-x)\lambda(x)\psi_{\delta}(\xi-\zeta)d\xi d\zeta dy dx ds\\
&& + \frac{1}{2}\mathbb{E}\int^t_0\int_{D^{\lambda}_x}\int_{D_y}\int^N_{-N}\int^N_{-N}\lambda(x)w(x)\rho^{\lambda}_{\gamma}(y-x)\psi_{\delta}(\xi-\zeta)
\sum_{k\geq 1}|g_k(x,\xi)-g_k(y,\zeta)|^2d\nu^{1}_{s,x}\otimes \nu^{2}_{s,y}(\xi,\zeta)dy dxds\\
&& +I_N.
\end{eqnarray*}

A similar bound can be obtained  for
$\mathbb{E} \int_{D^{\lambda}_x}\int_{D_y}\int^N_{-N}\int^N_{-N}\bar{f}^+_1(t)f^+_2(t)\alpha^{\lambda}w(x)d\xi d\zeta dy dx.
$
Adding the above two bounds, we get the desired result (\ref{e-4}) for $f^+_i$.
To obtain the result for  $f^-_i$, we simply take $t_n\uparrow t$, write (\ref{e-4}) for $f^+_i(t_n)$ and let $n\rightarrow \infty$.

\end{proof}
The following  is the so called ``super $L^1_{w;x}-$contration principle'' mentioned in the introduction.
\begin{thm}\label{thm-1}
The kinetic solutions $u(t;\vartheta)$, $u(t;\tilde{\vartheta})$ of $\mathcal{E}(A,\Phi,\vartheta)$ and $\mathcal{E}(A,\Phi,\tilde{\vartheta})$ satisfy
  \begin{eqnarray}\label{e-16-1}
  \mathbb{E}\|u(t;\vartheta)-u(t;\tilde{\vartheta})\|_{L^1_{w;x}}\leq \mathbb{E}\|\vartheta-\tilde{\vartheta}\|_{L^1_{w;x}}-\sum^d_{j=1}\mathbb{E}\int^t_0\int_{ D}|A_j(u(s;\vartheta))-A_j(u(s;\tilde{\vartheta}))| dxds.
  \end{eqnarray}
\end{thm}
\begin{proof}
For any $t\geq 0$, $N\in \mathbb{D}$ and any element $\lambda$ of the partition of unity $\{\lambda_i\}$ on $\bar{D}$, define the error term
\begin{eqnarray}\notag
&&\mathcal{E}^{N,\lambda}_t(\gamma,\delta)\\ \notag
&:=&\mathbb{E}\int_{D^{\lambda}_x}\int_{D_y}\int^N_{-N}\int^N_{-N}(f^{\pm}_1(t)\bar{f}^{\pm}_2(t)+\bar{f}^{\pm}_1(t)f^{\pm}_2(t))\rho^{\lambda}_{\gamma}(y-x)\psi_{\delta}(\xi-\zeta)\lambda(x)w(x)d\xi d\zeta dy dx\\
\label{qq-3}
&&-\mathbb{E}\int_{ D^{\lambda}_x}\int^N_{-N}(f^{\pm}_1(x,t,\xi)\bar{f}^{\pm}_2(x,t,\xi)+\bar{f}^{\pm}_1(x,t,\xi)f^{\pm}_2(x,t,\xi))\lambda(x)w(x)d\xi dx.
\end{eqnarray}
Clearly, it can be written as
\begin{eqnarray*}
\mathcal{E}^{N,\lambda}_t(\gamma,\delta)=: H^{N,\lambda}_1(t)+H^{N,\lambda}_2(t),
\end{eqnarray*}
where
\begin{eqnarray}\notag
H^{N,\lambda}_1(t)&=& \mathbb{E}\int_{D^{\lambda}_x}\int_{D_y}\int^N_{-N}\int^N_{-N}(f^{\pm}_1(t)\bar{f}^{\pm}_2(t)+\bar{f}^{\pm}_1(t)f^{\pm}_2(t))\rho^{\lambda}_{\gamma}(y-x)\psi_{\delta}(\xi-\zeta)\lambda(x)w(x)d\xi d\zeta dy dx\\ \notag
&&-\mathbb{E}\int_{D^{\lambda}_x}\int_{D_y}\int^N_{-N}(f^{\pm}_1(t,x,\xi)\bar{f}^{\pm}_2(t,y,\xi)+\bar{f}^{\pm}_1(t,x,\xi){f}^{\pm}_2(t,y,\xi))\rho^{\lambda}_{\gamma}(y-x)\lambda(x)w(x)d\xi dy dx,\\ \notag
H^{N,\lambda}_2(t)&=& \mathbb{E}\int_{D^{\lambda}_x}\int_{D_y}\int^N_{-N}(f^{\pm}_1(t,x,\xi)\bar{f}^{\pm}_2(t,y,\xi)+\bar{f}^{\pm}_1(t,x,\xi){f}^{\pm}_2(t,y,\xi))\rho^{\lambda}_{\gamma}(y-x)\lambda(x)w(x)d\xi dy dx\\ \notag
&& -\mathbb{E}\int_{ D^{\lambda}_x}\int^N_{-N}(f^{\pm}_1(t,x,\xi)\bar{f}^{\pm}_2(t,x,\xi)+\bar{f}^{\pm}_1(t,x,\xi)f^{\pm}_2(t,x,\xi))\lambda(x)w(x)d\xi dx.
\end{eqnarray}
We start with the estimate of $H^{N,\lambda}_1(t)$. Notice that
\begin{eqnarray*}
  &&\int_{D^{\lambda}_x}\int_{D_y}\int^{N}_{-N}\rho^{\lambda}_{\gamma}(y-x)f^{\pm}_1(t,x,\xi)\bar{f}^{\pm}_2(t,y,\xi)\lambda(x)w(x)d\xi dy dx\\
  &=& \int_{D^{\lambda}_x}\int_{D_y}\int^{N}_{-N}\int^{N}_{-N}\rho^{\lambda}_{\gamma}(y-x)\psi_{\delta}(\xi-\zeta)f^{\pm}_1(t,x,\xi)\bar{f}^{\pm}_2(t,y,\xi)\lambda(x)w(x)d\xi d\zeta dy dx\\
  && +\int_{D^{\lambda}_x}\int_{D_y}\int^{N}_{-N}\rho^{\lambda}_{\gamma}(y-x)f^{\pm}_1(t,x,\xi)\bar{f}^{\pm}_2(t,y,\xi)\lambda(x)w(x)\Big(1-\int^{N}_{-N}\psi_{\delta}(\xi-\zeta)d\zeta\Big)d\xi dy dx,
\end{eqnarray*}
which implies
\begin{eqnarray*}\notag
&&\Big|\int_{D^{\lambda}_x}\int_{D_y}\int^{N}_{-N}\rho^{\lambda}_{\gamma}(y-x)
f^{\pm}_1(t,x,\xi)\bar{f}^{\pm}_2(t,y,\xi)\lambda(x)w(x)d\xi dy dx\\ \notag
&&\ -\int_{D^{\lambda}_x}\int_{D_y}\int^{N}_{-N}\int^{N}_{-N}
f^{\pm}_1(t,x,\xi)\bar{f}^{\pm}_2(t,y,\zeta)\rho^{\lambda}_{\gamma}(y-x)\psi_{\delta}(\xi-\zeta)\lambda(x)w(x)d\xi d\zeta dy dx\Big|\\ \notag
&\leq&\Big|\int_{D^{\lambda}_x}\int_{D_y}\rho^{\lambda}_{\gamma}(y-x)\lambda(x)w(x)
\int^N_{-N}I_{u^{\pm}(t,x;\vartheta)>\xi}\int^N_{-N}\psi_{\delta}(\xi-\zeta)(I_{u^{\pm}(t,y;\tilde{\vartheta})\leq\xi}-I_{u^{\pm}(t,y;\tilde{\vartheta})\leq\zeta})d\zeta d\xi dy dx\Big|\\
&& +C\int^{N}_{-N}\Big(1-\int^{N}_{-N}\psi_{\delta}(\xi-\zeta)d\zeta\Big)d\xi\\ \notag
&=:& K^{N,\lambda}_1(\gamma,\delta)+\Upsilon^{N}(\delta).
\end{eqnarray*}

Applying the dominated convergence theorem, we see that
 $\lim_{\delta\rightarrow 0}\Upsilon^{N}(\delta)=0$.
Moreover, by the fact that $\int^{\delta}_0\psi_{\delta}(\zeta')d\zeta'=\int^{0}_{-\delta}\psi_{\delta}(\zeta')d\zeta'=\frac{1}{2}$,  we deduce that
\begin{eqnarray*}\notag
&&K^{N,\lambda}_1(\gamma,\delta)\\ \notag
&\leq&\int_{D^{\lambda}_x}\int_{D_y}\rho^{\lambda}_{\gamma}(y-x)\lambda(x)w(x)\int^N_{-N}I_{u^{\pm}(t,x;\vartheta)>\xi}\int^{\xi}_{(\xi-\delta)\vee (-N)}\psi_{\delta}(\xi-\zeta)I_{\zeta<u^{\pm}(t,y;\tilde{\vartheta})<\xi} d\zeta d\xi dy dx\\ \notag
&&+\int_{D^{\lambda}_x}\int_{D_y}\rho^{\lambda}_{\gamma}(y-x)\lambda(x)w(x)\int^N_{-N}I_{u^{\pm}(t,x;\vartheta)>\xi}\int^{(\xi+\delta)\wedge N}_{\xi}\psi_{\delta}(\xi-\zeta)I_{\xi<u^{\pm}(t,y;\tilde{\vartheta})<\zeta} d\zeta d\xi dy dx\\ \notag
&\leq&\int_{D^{\lambda}_x}\int_{D_y}\rho^{\lambda}_{\gamma}(y-x)\lambda(x)w(x)\int^{N\wedge u^{\pm}(t,x;\vartheta)\wedge (u^{\pm}(t,y;\tilde{\vartheta})+\delta)}_{u^{\pm}(t,y;\tilde{\vartheta})}\int^{\xi}_{(\xi-\delta)\vee (-N)}\psi_{\delta}(\xi-\zeta) d\zeta d\xi dy dx\\ \notag
&& +\int_{D^{\lambda}_x}\int_{D_y}\rho^{\lambda}_{\gamma}(y-x)\lambda(x)w(x)\int^{ u^{\pm}(t,x;\vartheta)\wedge u^{\pm}(t,y;\tilde{\vartheta})}_{-N\vee(u^{\pm}(t,y;\tilde{\vartheta})-\delta)}\int^{(\xi+\delta)\wedge N}_{\xi}\psi_{\delta}(\xi-\zeta) d\zeta d\xi dy dx\\ \notag
&\leq&\int_{D^{\lambda}_x}\int_{D_y}\rho^{\lambda}_{\gamma}(y-x)\lambda(x)w(x)\int^{N\wedge u^{\pm}(t,x;\vartheta)\wedge (u^{\pm}(t,y;\tilde{\vartheta})+\delta)}_{u^{\pm}(t,y;\tilde{\vartheta})}\int^{\delta\wedge (\xi+N)}_{0}\psi_{\delta}(r) dr d\xi dy dx\\ \notag
&& +\int_{D^{\lambda}_x}\int_{D_y}\rho^{\lambda}_{\gamma}(y-x)\lambda(x)w(x)\int^{ u^{\pm}(t,x;\vartheta)\wedge u^{\pm}(t,y;\tilde{\vartheta})}_{-N\vee(u^{\pm}(t,y;\tilde{\vartheta})-\delta)}\int^{0}_{(-\delta)\vee (\xi-N)}\psi_{\delta}(r) drd\xi dy dx\\ \notag
&\leq& \delta\int_{D^{\lambda}_x}\int_{D_y}\rho^{\lambda}_{\gamma}(y-x)\lambda(x)w(x)dy dx\\
&\leq&  C\delta\int_{D^{\lambda}_x}\lambda(x)dx, \quad a.s..
\end{eqnarray*}
Hence, we get
\begin{eqnarray}\notag
&&\Big|\int_{D^{\lambda}_x}\int_{D_y}\int^{N}_{-N}\rho^{\lambda}_{\gamma}(y-x)f^{\pm}_1(t,x,\xi)\bar{f}^{\pm}_2(t,y,\xi)\lambda(x)w(x)d\xi dy dx\\ \notag
&&\ -\int_{D^{\lambda}_x}\int_{D_y}\int^{N}_{-N}\int^{N}_{-N}f^{\pm}_1(t,x,\xi)\bar{f}^{\pm}_2(t,y,\zeta)\rho^{\lambda}_{\gamma}(y-x)\psi_{\delta}(\xi-\zeta)\lambda(x)w(x)d\xi d\zeta dydx\Big|\\
 \label{qeq-14}
&\leq& C\delta\int_{D^{\lambda}_x}\lambda(x)dx+\Upsilon^{N}(\delta), \quad a.s..
\end{eqnarray}
Similarly, it follows that
\begin{eqnarray}\notag
&&\Big|\int_{D^{\lambda}_x}\int_{D_y}\int^{N}_{-N}\rho^{\lambda}_{\gamma}(y-x)\bar{f}^{\pm}_1(t,x,\xi)f^{\pm}_2(t,y,\xi)\lambda(x)w(x)d\xi dy dx\\ \notag
&&\ -\int_{D^{\lambda}_x}\int_{D_y}\int^{N}_{-N}\int^{N}_{-N}\bar{f}^{\pm}_1(t,x,\xi)f^{\pm}_2(t,y,\zeta)\rho^{\lambda}_{\gamma}(y-x)\psi_{\delta}(\xi-\zeta)\lambda(x)w(x)d\xi d\zeta dydx\Big|\\
\label{qeq-14-1}
&\leq & C\delta\int_{D^{\lambda}_x}\lambda(x)dx+\Upsilon^N(\delta), \quad a.s..
\end{eqnarray}
Based on (\ref{qeq-14}) and (\ref{qeq-14-1}), by using the dominated convergence theorem, we obtain
\begin{eqnarray}\notag
  \sum^M_{i=0}|H^{N,\lambda_i}_1(t)|&\leq& C\delta  \sum^M_{i=0}\int_{D^{\lambda_i}_x}\lambda_i(x)dx+2M\Upsilon^{N}(\delta)\\
  \notag
  &=& C\delta  \int_{D}\sum^M_{i=0}\lambda_i(x)dx+2M\Upsilon^{N}(\delta)\\
  \label{rrr-2}
  &\leq& C\delta+2M\Upsilon^{N}(\delta).
\end{eqnarray}
Moreover, by utilizing $\rho^{\lambda}_{\gamma}(y-x)=0$ on $D^{\lambda}_x\times D^c$, it follows that
\begin{eqnarray*}
&&\Big|\int_{D^{\lambda}_x}\int_{D_y}\int^{N}_{-N}\rho^{\lambda}_{\gamma}(y-x)\lambda(x)w(x)f^{\pm}_1(t,x,\xi)\bar{f}^{\pm}_2(t,y,\xi)d\xi dydx\\
&& -\int_{D^{\lambda}_x}\int^N_{-N}f^{\pm}_1(t,x,\xi)\bar{f}^{\pm}_2(t,x,\xi)\lambda(x)w(x)d\xi dx\Big|\\ \notag
&=&\Big|\int_{D^{\lambda}_x}\int_{D}\int^N_{-N}\rho^{\lambda}_{\gamma}(y-x)\lambda(x)w(x)f^{\pm}_1(t,x,\xi)(\bar{f}^{\pm}_2(t,y,\xi)-\bar{f}^{\pm}_2(t,x,\xi))d\xi dydx\Big|\\
&=&\Big|\int_{D^{\lambda}_x}\int_{\{z_i\in (-\gamma, \gamma), z_d\in (\gamma L_{\lambda},\gamma L_{\lambda}+2\gamma)\}}\int^N_{-N}\rho^{\lambda}_{\gamma}(z)\lambda(x)w(x)f^{\pm}_1(t,x,\xi)(\bar{f}^{\pm}_2(t,x+z,\xi)-\bar{f}^{\pm}_2(t,x,\xi))d\xi dzdx\Big|\\
&\leq&\sup_{\{z_i\in (-\gamma, \gamma), z_d\in (\gamma L_{\lambda},\gamma L_{\lambda}+2\gamma)\}}\int_{D^{\lambda}_x}\lambda(x)w(x)\int_{\mathbb{R}}f^{\pm}_1(t,x,\xi)|\bar{f}^{\pm}_2(t,x+z,\xi)-\bar{f}^{\pm}_2(t,x,\xi)|d\xi dx\\ \notag
&\leq& C \sup_{\{z_i\in (-\gamma, \gamma), z_d\in (\gamma L,\gamma L+2\gamma)\}}\int_{D^{\lambda}_x}\lambda(x)\int_{\mathbb{R}}|-f^{\pm}_2(t,x+z,\xi)+I_{0>\xi}-I_{0>\xi}+f^{\pm}_2(t,x,\xi)|d\xi dx\\
&\leq& C \sup_{\{z_i\in (-\gamma, \gamma), z_d\in (\gamma L,\gamma L+2\gamma)\}}\int_{D^{\lambda}_x}\lambda(x)\int_{\mathbb{R}}|\Lambda_{f^{\pm}_2}(t,x+z,\xi)-\Lambda_{f^{\pm}_2}(t,x,\xi)|d\xi dx,
\end{eqnarray*}
where we have used the boundedness of $w$. Hence,
\begin{eqnarray*}
&&\sum^M_{i=0}\Big|\int_{D^{\lambda_i}_x}\int_{D_y}\int^{N}_{-N}\rho^{\lambda_i}_{\gamma}
(y-x)\lambda_i(x)w(x)f^{\pm}_1(t,x,\xi)\bar{f}^{\pm}_2(t,y,\xi)d\xi dydx\\
&& -\int_{D^{\lambda_i}_x}\int^N_{-N}f^{\pm}_1(t,x,\xi)\bar{f}^{\pm}_2(t,x,\xi)\lambda_i(x)w(x)d\xi dx\Big|\\ \notag
&\leq& C \sup_{\{z_i\in (-\gamma, \gamma), z_d\in (\gamma L,\gamma L+2\gamma)\}}\sum^M_{i=0}\int_{D^{\lambda_i}_x}\lambda_i(x)\int_{\mathbb{R}}|\Lambda_{f^{\pm}_2}(t,x+z,\xi)-\Lambda_{f^{\pm}_2}(t,x,\xi)|d\xi dx\\
&=&C \sup_{\{z_i\in (-\gamma, \gamma), z_d\in (\gamma L,\gamma L+2\gamma)\}}\int_{D}\sum^M_{i=0}\lambda_i(x)\int_{\mathbb{R}}|\Lambda_{f^{\pm}_2}(t,x+z,\xi)-\Lambda_{f^{\pm}_2}(t,x,\xi)|d\xi dx\\
&=&C \sup_{\{z_i\in (-\gamma, \gamma), z_d\in (\gamma L,\gamma L+2\gamma)\}}\int_{D}\int_{\mathbb{R}}|\Lambda_{f^{\pm}_2}(t,x+z,\xi)-\Lambda_{f^{\pm}_2}(t,x,\xi)|d\xi dx, \quad a.s..
\end{eqnarray*}
The integrability of $\Lambda_{f^{\pm}_2}$ on $D\times \mathbb{R}$ implies that
\begin{eqnarray*}
&&\lim_{\gamma\rightarrow 0}\sum^M_{i=0}\Big|\int_{D^{\lambda_i}_x}\int_{D_y}\int^{N}_{-N}\rho^{\lambda_i}_{\gamma}(y-x)\lambda_i(x)w(x)f^{\pm}_1(t,x,\xi)\bar{f}^{\pm}_2(t,y,\xi)d\xi dydx\\ \notag
&&\quad\quad\quad -\int_{D^{\lambda_i}_x}\int^N_{-N}f^{\pm}_1(t,x,\xi)\bar{f}^{\pm}_2(t,x,\xi)\lambda_i(x)w(x)d\xi dx\Big|= 0, \quad a.s..
\end{eqnarray*}
Consequently, by the  dominated convergence theorem, we obtain
\begin{eqnarray*}\notag
&&\lim_{\gamma\rightarrow 0}\sum^M_{i=0}\mathbb{E}\Big|\int_{D^{\lambda_i}_x}\int_{D_y}\int^{N}_{-N}\rho^{\lambda_i}_{\gamma}(y-x)\lambda_i(x)w(x)f^{\pm}_1(t,x,\xi)\bar{f}^{\pm}_2(t,y,\xi)d\xi dydx\\
&&\quad\quad\quad -\int_{D^{\lambda_i}_x}\int^N_{-N}f^{\pm}_1(t,x,\xi)\bar{f}^{\pm}_2(t,x,\xi)\lambda_i(x)w(x)d\xi dx\Big|= 0.
\end{eqnarray*}
Similarly, we have
\begin{eqnarray*}\notag
&&\lim_{\gamma\rightarrow 0}\sum^M_{i=0}\mathbb{E}\Big|\int_{D^{\lambda_i}_x}\int_{D_y}\int^{N}_{-N}\rho^{\lambda_i}_{\gamma}(y-x)\lambda_i(x)w(x)\bar{f}^{\pm}_1(t,x,\xi)f^{\pm}_2(t,y,\xi)d\xi dy dx\\
&& \quad\quad\quad-\int_{D^{\lambda_i}_x}\int^N_{-N}\bar{f}^{\pm}_1(t,x,\xi)f^{\pm}_2(t,x,\xi)\lambda_i(x)w(x)d\xi dx\Big|= 0.
\end{eqnarray*}
Thus, we can conclude that
\begin{eqnarray}\label{eqq-7}
\lim_{\gamma\rightarrow 0}\sum^M_{i=0}|H^{N,\lambda_i}_2|(t)=0.
\end{eqnarray}
Combining  (\ref{rrr-2}) with (\ref{eqq-7}), we have for any $t\in [0,T]$
\begin{eqnarray}\label{qq-4}
\lim_{\gamma, \delta\rightarrow 0}\sum^M_{i=0}|\mathcal{E}^{N,\lambda_i}_t(\gamma,\delta)|=0.
\end{eqnarray}
In particular,
\begin{eqnarray}\label{qq-5}
\lim_{\gamma, \delta\rightarrow 0}\sum^M_{i=0}|\mathcal{E}^{N,\lambda_i}_0(\gamma,\delta)|=0.
\end{eqnarray}
Using the dominated convergence theorem, it follows from (\ref{qq-4}) that
\begin{eqnarray}\label{eqq-8}
\lim_{\gamma, \delta\rightarrow 0}\sum^M_{i=0}\int^T_0|\mathcal{E}^{N,\lambda_i}_t(\gamma,\delta)|dt=0.
\end{eqnarray}
For any $t\in [0,T]$ and $N>0$, define the error term
\begin{eqnarray}\notag
&&r^{N,\lambda}_t(\gamma,\delta)\\ \notag
&:=& \mathbb{E}\int^t_0\int_{\partial D^{\lambda}_x}\int_{D_y}\int^N_{-N}\int^N_{-N}(\tilde{f}^{(\lambda)}_1(s)\bar{f}_2(s)+\tilde{\bar{f}}^{(\lambda)}_1(s)f_2(s))(-a(\xi)\cdot\mathbf{ n})\rho^{\lambda}_{\gamma}(y-x)\\ \notag
&& \times \psi_{\delta}(\xi-\zeta)\lambda(x)w(x)d\xi d\zeta dy d\sigma(x)ds\\
\label{eqq-1}
&& - \mathbb{E}\int^t_0\int_{\partial D^{\lambda}_x}\int^N_{-N}(\tilde{f}^{(\lambda)}_1(s)\tilde{\bar{f}}^{(\lambda)}_2(s)+\tilde{\bar{f}}^{(\lambda)}_1(s)\tilde{f}^{(\lambda)}_2(s))(-a(\xi)\cdot\mathbf{ n})\lambda(x)w(x)d\xi d\sigma(x)ds.
\end{eqnarray}
According to Proposition \ref{prp-7}, there exists a subsequence, still denoted by $\{\gamma\}_{\gamma>0}$, such that
 $ \bar{f}_2\ast \rho^{\lambda}_{\gamma}\rightarrow \tilde{\bar{f}}^{(\lambda)}_2$ and $f_2\ast \rho^{\lambda}_{\gamma}\rightarrow \tilde{f}^{(\lambda)}_2$ in the weak star topology in $L^{\infty}(\Omega\times\Sigma^{\lambda}\times \mathbb{R})$, as $\gamma\rightarrow 0$.
 Using similar methods as the estimates of $H^{N,\lambda}_1$, we deduce that for each $N>0, t>0$,
\begin{eqnarray}\label{eqq-2}
  \lim_{\gamma, \delta\rightarrow 0}\sum^M_{i=0}r^{N,\lambda_i}_t(\gamma,\delta)=0.
\end{eqnarray}
Applying the  dominated convergence theorem again, it follows that
\begin{eqnarray}\label{eqq-9}
  \lim_{\gamma, \delta\rightarrow 0}\sum^M_{i=0}\int^T_0r^{N,\lambda_i}_t(\gamma,\delta)dt=0.
\end{eqnarray}
The above estimates imply
\begin{eqnarray}\notag
 &&\mathbb{E}\int_{ D^{\lambda}}\int^N_{-N}(f^{\pm}_1(t,x,\xi)\bar{f}^{\pm}_2(t,x,\xi)+\bar{f}^{\pm}_1(t,x,\xi)f^{\pm}_2(t,x,\xi))\lambda(x)w(x)d\xi dx\\ \notag
 &\leq& \mathbb{E} \int_{D^{\lambda}}\int^N_{-N}(f_{1,0}\bar{f}_{2,0}+\bar{f}_{1,0}f_{2,0})\lambda(x)w(x)d\xi  dx \\ \notag
 && +\mathbb{E}\int^t_0\int_{\partial D^{\lambda}}\int^N_{-N}(\tilde{f}^{(\lambda)}_1(s)\tilde{\bar{f}}^{(\lambda)}_2(s)+\tilde{\bar{f}}^{(\lambda)}_1(s)\tilde{f}^{(\lambda)}_2(s))
 (-a(\xi)\cdot\mathbf{ n})\lambda(x)w(x)d\xi d\sigma(x)ds\\
 \label{eqq-3}
&& +\sum_{i=2}^5J_i+r^{N,\lambda}_t(\gamma,\delta)+\mathcal{E}^{N,\lambda}_t(\gamma,\delta)+\mathcal{E}^{N,\lambda}_0(\gamma,\delta)+I_N,
\end{eqnarray}
where $J_i, i=2,\dots,5$ and $I_N$ were defined in the statement of Proposition 4.1, $I_N$ was defined in (4.5).
  Noting that $a(\xi)\cdot\mathbf{ n}\tilde{\bar{f}}^{(\lambda)}_2=a(\xi)\cdot\mathbf{ n}\tilde{\bar{f}}_2$ a.e. on $[0,T)\times \partial D^{\lambda}\times (-N,N)$, it follows that
  \begin{eqnarray*}
   &&\sum^M_{i=0}\mathbb{E}\int^t_0\int_{\partial D^{\lambda_i}}\int^N_{-N}(\tilde{f}^{(\lambda_i)}_1(s)\tilde{\bar{f}}^{(\lambda_i)}_2(s)+\tilde{\bar{f}}^{(\lambda_i)}_1(s)\tilde{f}^{(\lambda_i)}_2(s))
 (-a(\xi)\cdot\mathbf{n})\lambda_i(x)w(x)d\xi d\sigma(x)ds \\
 &=&\sum^M_{i=0}\mathbb{E}\int^t_0\int_{\partial D^{\lambda_i}}\int^N_{-N}(\tilde{f}^{(\lambda_i)}_1(s)\tilde{\bar{f}}_2(s)+\tilde{\bar{f}}^{(\lambda_i)}_1(s)\tilde{f}_2(s))
 (-a(\xi)\cdot\mathbf{n})\lambda_i(x)w(x)d\xi d\sigma(x)ds \\
 &=& \mathbb{E}\int^t_0\int_{\partial D}\int^N_{-N}\sum^M_{i=0}\lambda_i(x)(\tilde{f}^{(\lambda_i)}_1(s)\tilde{\bar{f}}_2(s)+\tilde{\bar{f}}^{(\lambda_i)}_1(s)\tilde{f}_2(s))
 (-a(\xi)\cdot\mathbf{n})w(x)d\xi d\sigma(x)ds \\
 &=& \mathbb{E}\int^t_0\int_{\partial D}\int^N_{-N}(\tilde{f}_1(s)\tilde{\bar{f}}_2(s)+\tilde{\bar{f}}_1(s)\tilde{f}_2(s))
 (-a(\xi)\cdot\mathbf{n})w(x)d\xi d\sigma(x)ds,
  \end{eqnarray*}
  where we have used the facts that $\sum^M_{i=0}\lambda_i(x)\tilde{f}^{(\lambda_i)}_1=\tilde{f}_1$ and $\sum^M_{i=0}\lambda_i(x)\tilde{\bar{f}}^{(\lambda_i)}_1=\tilde{\bar{f}}_1$.

  Thus, summing (\ref{eqq-3}) over $i=0, \dots, M$, and using $\sum^{M}_{i=0}\lambda_i=1$, we get
  \begin{eqnarray}\notag
 &&\mathbb{E}\int_{ D}\int^N_{-N}(f^{\pm}_1(t,x,\xi)\bar{f}^{\pm}_2(t,x,\xi)+\bar{f}^{\pm}_1(t,x,\xi)f^{\pm}_2(t,x,\xi))w(x)d\xi dx\\ \notag
 &\leq& \mathbb{E} \int_{D}\int^N_{-N}(f_{1,0}\bar{f}_{2,0}+\bar{f}_{1,0}f_{2,0})w(x)d\xi dx\\ \notag
 && +\mathbb{E}\int^t_0\int_{\partial D}\int^N_{-N}(\tilde{f}_1(s)\tilde{\bar{f}}_2(s)+\tilde{\bar{f}}_1(s)\tilde{f}_2(s))
 (-a(\xi)\cdot\mathbf{n})w(x)d\xi d\sigma(x)ds\\
&& +\sum^M_{i=0}
\Big(J_2+J_3+J_4+J_5+r^{N,\lambda_i}_t(\gamma,\delta)
+\mathcal{E}^{N,\lambda_i}_t(\gamma,\delta)+\mathcal{E}^{N,\lambda_i}_0(\gamma,\delta)+I_N\Big).\label{eqq-3-1}
\end{eqnarray}
From the proof of Theorem 15 in \cite{D-V-1}, it is known that
\begin{eqnarray*}
\sum^M_{i=0}|J_2|\leq CM\delta\gamma^{-1},\quad \sum^M_{i=0}|J_5|\leq CM(\gamma^2\delta^{-1}+\delta).
\end{eqnarray*}
For the boundary term, according to (3.11) in \cite{KN16},
\begin{eqnarray*}
\int^N_{-N}(\tilde{f}_1(s)\tilde{\bar{f}}_2(s)+\tilde{\bar{f}}_1(s)\tilde{f}_2(s))
 (-a(\xi)\cdot\mathbf{n})d\xi
 \leq |a(0)|\int_{\mathbb{R}}(f_{1,b}\bar{f}_{2,b}+\bar{f}_{1,b} f_{2,b})d\xi.
\end{eqnarray*}
Thus, we deduce from (\ref{eqq-3-1}) that
 \begin{eqnarray*}\notag
 &&\mathbb{E}\int_{ D}\int^N_{-N}(f^{\pm}_1(t)\bar{f}^{\pm}_2(t)+\bar{f}^{\pm}_1(t)f^{\pm}_2(t))w(x)d\xi dx\\ \notag
 &\leq& \mathbb{E} \int_{D}\int^N_{-N}(f_{1,0}\bar{f}_{2,0}+\bar{f}_{1,0}f_{2,0})w(x)d\xi dx \\ \notag
 && +|a(0)|\mathbb{E}\int^t_0\int_{\partial D}\int_{\mathbb{R}}(f_{1,b}\bar{f}_{2,b}+\bar{f}_{1,b} f_{2,b})w(x)d\xi d\sigma(x)ds\\
&& +CM\delta\gamma^{-1}+CM(\gamma^2\delta^{-1}+\delta)
+\sum^M_{i=0}\Big(J_3+J_4+r^{N,\lambda_i}_t(\gamma,\delta)+\mathcal{E}^{N,\lambda_i}_t(\gamma,\delta)+\mathcal{E}^{N,\lambda_i}_0(\gamma,\delta)+I_N\Big).
\end{eqnarray*}
On the other hand, by  (3.9) in \cite{KN16}, we have
\begin{eqnarray}\label{eqq-34}
\lim_{\gamma,\delta\rightarrow 0}\sum^M_{i=0}J_3=\mathbb{E}\int^t_0\int_{D}\int^N_{-N} (f_1(s)\bar{f}_2(s)+\bar{f}_1(s)f_2(s)) a(\xi)\cdot \nabla_x\Big(\sum^M_{i=0}\lambda_i(x)\Big) d\xi dx ds=0.
\end{eqnarray}
Applying the similar method as in the proof of (\ref{qq-4}), and utilizing (\ref{eqq-31}) and $\max_{\xi\in[-N,N]}|a(\xi)|= M_N$, we obtain
\begin{eqnarray}\label{eqq-35}
\lim_{\gamma,\delta\rightarrow 0}\sum^M_{i=0}J_4=-\sum^d_{j=1}\mathbb{E}\int^t_0\int_{D}\int^N_{-N} (f_1(s)\bar{f}_2(s)+\bar{f}_1(s)f_2(s)) a_j(\xi) d\xi dx ds.
\end{eqnarray}
Now, taking $\delta=\gamma^{\frac{4}{3}}$ and letting $\gamma\rightarrow 0$, we deduce from (\ref{qq-4}), (\ref{qq-5}), (\ref{eqq-2}), (\ref{eqq-34}) and (\ref{eqq-35}) that
\begin{eqnarray}\notag
 &&\mathbb{E}\int_{ D}\int^N_{-N}(f^{\pm}_1(t)\bar{f}^{\pm}_2(t)+\bar{f}^{\pm}_1(t)f^{\pm}_2(t))w(x)d\xi dx\\ \notag
 &\leq& \mathbb{E} \int_{D}\int^N_{-N}(f_{1,0}\bar{f}_{2,0}+\bar{f}_{1,0}f_{2,0})w(x)d\xi dx \\ \notag
 && -\sum^d_{j=1}\mathbb{E}\int^t_0\int_{D}\int^N_{-N} (f_1(s)\bar{f}_2(s)+\bar{f}_1(s)f_2(s)) a_j(\xi) d\xi dx ds\\
 \label{a-7}
 && +|a(0)|\mathbb{E}\int^t_0\int_{\partial D}\int_{\mathbb{R}}(f_{1,b}\bar{f}_{2,b}+\bar{f}_{1,b} f_{2,b})w(x)d\xi d\sigma(x)ds
+\sum^M_{i=0}I_N.
\end{eqnarray}
Recall that $f_1(t,x,\xi)=I_{u(t,x;\vartheta)>\xi}$,  $f_2(t,x,\xi)=I_{u(t,x;\tilde{\vartheta})>\xi}$. Denote by $u_1=u_1(s)=u(s;\vartheta), u_2=u_2(s)=u(s;\tilde{\vartheta})$. Next, we prove that for any $1\leq j\leq d$,
\begin{eqnarray}\label{a-6}
  &&\lim_{N\rightarrow \infty}\mathbb{E}\int^t_0\int_{D}\int^N_{-N} (f_1(s)\bar{f}_2(s)+\bar{f}_1(s){f}_2(s)) a_j(\xi) d\xi dx ds\nonumber\\
  &=&\lim_{N\rightarrow \infty}\mathbb{E}\int^t_0\int_{D}\int^N_{-N} (I_{u_1>\xi}\bar{I}_{u_2>\xi} +\bar{I}_{u_1>\xi}I_{u_2>\xi}) a_j(\xi) d\xi dx ds\nonumber\\
  &=&\mathbb{E}\int^t_0\int_{D}|A_j(u_1)-A_j(u_2)|dxds.
\end{eqnarray}
As $A_j$ is increasing, we have
\begin{eqnarray*}
\int^N_{-N} I_{u_1>\xi}\bar{I}_{u_2>\xi}a_j(\xi) d\xi&=&(A_j(u_1\wedge N)-A_j(u_2\vee (-N)))^+, \\
\int^N_{-N}\bar{I}_{u_1>\xi}I_{u_2>\xi} a_j(\xi) d\xi&=&(A_j(u_1\vee (-N))-A_j(u_2\wedge N))^-.
\end{eqnarray*}
On the other hand, it follows that
\begin{eqnarray*}
|(A_j(u_1\wedge N)-A_j(u_2\vee (-N)))^+|&\leq & C(|u_1|^{q_0+1}+|u_1|^{q_0+1}) , \\
|(A_j(u_1\vee (-N))-A_j(u_2\wedge N))^-|&\leq& C(|u_1|^{q_0+1}+|u_1|^{q_0+1}).
\end{eqnarray*}
Now, (\ref{a-6}) follows from the dominated convergence theorem.

Letting $N\rightarrow \infty$ in (\ref{a-7}), using (\ref{eqq-32}) and (\ref{a-6}), and  the dominated convergence theorem, we obtain
\begin{eqnarray}\notag
 &&\mathbb{E}\int_{ D}\int_{\mathbb{R}}(f^{\pm}_1(t)\bar{f}^{\pm}_2(t)+\bar{f}^{\pm}_1(t)f^{\pm}_2(t))w(x)d\xi dx\\ \notag
 &\leq& \mathbb{E} \int_{D}\int_{\mathbb{R}}(f_{1,0}\bar{f}_{2,0}+\bar{f}_{1,0}f_{2,0})w(x)d\xi dx \\ \notag
 && -\sum^d_{j=1}\mathbb{E}\int^t_0\int_{ D}|A_j(u(s;\vartheta))-A_j(u(s;\tilde{\vartheta}))|dx ds\\
 \label{eqq-21}
 && +|a(0)|\mathbb{E}\int^t_0\int_{\partial D}\int_{\mathbb{R}}(f_{1,b}\bar{f}_{2,b}+\bar{f}_{1,b} f_{2,b})w(x)d\xi d\sigma(x)ds.
\end{eqnarray}
Since $f_1(t,x,\xi)=I_{u(t,x;\vartheta)>\xi}$,  $f_2(t,x,\xi)=I_{u(t,x;\tilde{\vartheta})>\xi}$, $f_{i,b}=I_{0>\xi}$, $i=1,2$, $f_{1,0}=I_{\vartheta>\xi}$, $f_{2,0}=I_{\tilde{\vartheta}>\xi} $, using
\begin{eqnarray}\label{equ-1}
\int_{\mathbb{R}}I_{u>\xi}\overline{I_{v>\xi}}d\xi=(u-v)^+,
\quad
\int_{\mathbb{R}}\overline{I_{u>\xi}}I_{v>\xi}d\xi=(u-v)^-,
\end{eqnarray}
we deduce from (\ref{eqq-21}) that
\begin{eqnarray*}
&&\mathbb{E}\int_{ D}|u(t;\vartheta)-u(t;\tilde{\vartheta})|w(x)dx\\
&\leq& \mathbb{E}\int_{ D}|\vartheta-\tilde{\vartheta}|w(x) dx-\sum^d_{j=1}\mathbb{E}\int^t_0\int_{ D}|A_j(u(s;\vartheta))-A_j(u(s;\tilde{\vartheta}))|dx ds.
\end{eqnarray*}



\end{proof}

\begin{proof}[\textbf{Proof of Theorem \ref{thm-8}}]
As a consequence of Theorem 4.2, we have
\begin{eqnarray*}
  \mathbb{E}\|u(t;\vartheta)-u(t;\tilde{\vartheta})\|_{L^1_{w;x}}\leq \mathbb{E}\|\vartheta-\tilde{\vartheta}\|_{L^1_{w;x}}.
  \end{eqnarray*}

\end{proof}
\

At the end of this section, we mention that with the help of Lemma \ref{lem-5}, along the same arguments as in the proof of Proposition \ref{prp-1} and Theorem \ref{thm-1}, we also can prove the following result.
\begin{lemma}\label{lem-2}
 Let $u(t;\vartheta), u(t;\tilde{\vartheta})$ be kinetic solutions of $\mathcal{E}(A,\Phi,\vartheta)$ and $ \mathcal{E}(A,\Phi,\tilde{\vartheta})$ on $[0,T]$, respectively. Under Hypotheses (H1)-(H3), for almost every $0\leq s<t\leq T$, we have
  \begin{eqnarray}\label{e-19}
  \mathbb{E}\|u(t;\vartheta)-u(t;\tilde{\vartheta})\|_{L^1_{w;x}}\leq \mathbb{E}\|u(s;\vartheta)-u(s;\tilde{\vartheta})\|_{L^1_{w;x}}-\sum^d_{j=1}\mathbb{E}\int^t_s\int_{ D}|A_j(u(r;\vartheta))-A_j(u(r;\tilde{\vartheta}))| dxdr.
  \end{eqnarray}

\end{lemma}

\section{Proof of the continuity extension in the weighted space}

In this section, we will prove that the kinetic solution admits a continuous extension in the time variable. To this end,
for $n\geq 1$, we consider the following approximating equation:
\begin{eqnarray}\label{ee-2-1}
\left\{
  \begin{array}{ll}
  du^{n}-\frac{1}{n} \Delta u^{n}dt+div(A(u^{n}))dt=\Phi(u^{n}) dW(t) \quad \text{in}\ \ \Omega\times D\times (0,T),\\
u^{n}(\cdot,0)=\vartheta\in L^{q}(\Omega\times D) \quad \text{in}\ \  \Omega\times D,\\
u^{n}=0\quad \text{on}\  \ \Omega\times\Sigma.
  \end{array}
\right.
\end{eqnarray}
According to \cite{GR00}, for any $n\geq 1$, (\ref{ee-2-1}) admits a unique continuous solution $u^{n}\in L^2(\Omega;C([0,T];H)\cap L^2([0,T];H^1( D)))$ satisfying that for all $p\in [2,q]$,
\begin{eqnarray}\label{rr-5-1}
 \mathbb{E}\sup_{t\in [0,T]}\|u^{n}(t)\|^p_{L^p_x}+\frac{1}{n}\mathbb{E}\int^T_0\|\nabla u^{n} \|^2_{L^2_x}dt\leq C(1+\mathbb{E}\|\vartheta\|^p_{L^p_x})\leq C(1+\mathbb{E}\|\vartheta\|^{q}_{L^{q}_x}),
\end{eqnarray}
where $C$ is independent of $n$.

On the other hand, employing the techniques in \cite{KN16} and \cite{DHV}, we can derive the following kinetic formulation satisfied by $f=I_{u^n>\xi}$. Firstly, for any $\phi\in C_b(\mathbb{R})$, a chain rule formula holds true:
\begin{eqnarray}\label{P-24}
\partial_{x_i} \int^{u^n}_0\phi(r)dr=\phi(u^n)\partial_{x_i} u^n,\quad  \text{in}\  \mathcal{D}'(D)\ \ a.e.\ (\omega, t).
\end{eqnarray}
Moreover, there exists a kinetic measure $m^{n}$ and for any $N>0$, there exist nonnegative functions $\bar{m}^{\pm}_{N,n} \in L^1(\Omega\times \Sigma\times (-N,N))$ such that $\{\bar{m}^{\pm}_{N,n}(t)\}$ are predictable,  $\underset{\xi\uparrow N}{{\rm{\lim}}}\ \bar{m}^{+}_{N,n}(t,x,\xi)=\underset{\xi\downarrow -N}{{\rm{\lim}}}\ \bar{m}^{-}_{N,n}(t,x,\xi)=0$, and for all $\varphi\in C^{\infty}_c( [0,T)\times \bar{D}\times (-N,N))$, $f=I_{u^n>\xi}$ satisfies
\begin{eqnarray}\notag
&&\int^T_0\langle f(t), \partial_t \varphi(t)\rangle dt+\langle f_0, \varphi(0)\rangle +\int^T_0\langle f(t), a(\xi)\cdot \nabla \varphi (t)+\frac{1}{n} \Delta \varphi (t)\rangle dt+M_N\int_{\Sigma\times \mathbb{R}}f_b\varphi d\xi d\sigma(x)dt \\ \notag
&=& -\sum_{k\geq 1}\int^T_0\int_{ D}\int_{\mathbb{R}} g_k(x,\xi)\varphi (t,x,\xi)d\nu_{t,x}(\xi)dxd\beta_k(t) \\ \notag
&& -\frac{1}{2}\int^T_0\int_{ D}\int_{\mathbb{R}}\partial_{\xi}\varphi (t,x,\xi)G^2(x,\xi)d\nu_{t,x}(\xi)dxdt+q^{n}(\partial_{\xi} \varphi)\\
\label{ee-7-1}
&& +m^{n}(\partial_{\xi} \varphi)+\int_{\Sigma\times \mathbb{R}}\partial_{\xi} \varphi\bar{m}^+_{N,n} d\xi d\sigma(x)dt,\  a.s.,
\end{eqnarray}
and $\bar{f}:=1-f$ satisfies
\begin{eqnarray}\notag
&&\int^T_0\langle \bar{f}(t), \partial_t \varphi(t)\rangle dt+\langle \bar{f}_0, \varphi(0)\rangle +\int^T_0\langle \bar{f}(t), a(\xi)\cdot \nabla \varphi (t)+\frac{1}{n} \Delta \varphi (t)\rangle dt+M_N\int_{\Sigma\times \mathbb{R}}\bar{f}_b\varphi d\xi d\sigma(x)dt \\ \notag
&=& \sum_{k\geq 1}\int^T_0\int_{ D}\int_{\mathbb{R}} g_k(x,\xi)\varphi (t,x,\xi)d\nu_{t,x}(\xi)dxd\beta_k(t) \\ \notag
&& +\frac{1}{2}\int^T_0\int_{ D}\int_{\mathbb{R}}\partial_{\xi}\varphi (t,x,\xi)G^2(x,\xi)d\nu_{t,x}(\xi)dxdt+q^{n}(\partial_{\xi} \varphi)\\
\label{eqq-17}
&& -m^{n}(\partial_{\xi} \varphi)-\int_{\Sigma\times \mathbb{R}}\partial_{\xi} \varphi\bar{m}^-_{N,n} d\xi d\sigma(x)dt,\  a.s.,
\end{eqnarray}
where $f_0=I_{\vartheta>\xi}$, $f_b=I_{0>\xi}$, $\nu=-\partial_{\xi}f=\partial_{\xi}\bar{f}=\delta_{u^n=\xi}$ and $q^n: \Omega\rightarrow M^+_0(D\times[0,T]\times \mathbb{R})$ is defined as follows:
for any $\phi\in C_b(D\times[0,T]\times \mathbb{R})$
\[
q^n(\phi)=\frac{1}{n}\int^T_0\int_{D}\int_{\mathbb{R}}\phi(t,x,\xi)|\nabla u^{n}|^2d\delta_{u^n=\xi}dxdt.
\]
For simplicity, we write $q^n=\frac{1}{n}|\nabla u^n|^2\delta_{u^n=\xi}$.
For any $\vartheta_1, \vartheta_2\in L^{q}_{\omega}L^{q}_{x}$, denote by $u^n(t;\vartheta_1), u^n(t;\vartheta_2)$ the solutions of (\ref{ee-2-1}), respectively. Applying the doubling variables method again, similar to the proof of Proposition 4.1 we can show the following comparison theorem associated to $u^n(t;\vartheta_1)$ and $u^n(t;\vartheta_2)$.
\begin{lemma}\label{lem-3}
For any $t\in (0,T)$, $\gamma,\delta>0$, $N\in \mathbb{D}$, and any element $\lambda$ of the partition of unity $\{\lambda_i\}_{i=0,1,\dots,M}$ on $\bar{D}$, the functions $f_1(t):=f_1(t,x,\xi)=I_{u^{n}(t,x;\vartheta_1)>\xi}$ and $f_2(t):=f_2(t,y,\zeta)=I_{u^{n}(t,y;\vartheta_2)>\zeta}$ with data $f_{1,0}=I_{\vartheta_1>\xi}, f_{2,0}=I_{\vartheta_2>\zeta}, f_{i,b}=I_{0>\xi}, i=1,2,$ satisfy
\begin{eqnarray}\notag
 &&\mathbb{E} \int_{D^{\lambda}_x}\int_{D_y}\int^N_{-N}\int^N_{-N}(f^{\pm}_1(t)\bar{f}^{\pm}_2(t)+\bar{f}^{\pm}_1(t)f^{\pm}_2(t))\rho^{\lambda}_{\gamma}(y-x)\psi_{\delta}(\xi-\zeta)\lambda(x)w(x)d\xi d\zeta dy dx\\ \notag
 &\leq& \mathbb{E} \int_{D^{\lambda}_x}\int^N_{-N}(f_{1,0}\bar{f}_{2,0}+\bar{f}_{1,0}f_{2,0})\lambda(x)w(x)d\xi dx +\mathcal{E}^{N,\lambda}_0(\gamma,\delta)\\
\notag
&& +|a(0)|\mathbb{E} \int^{t}_0\int_{\partial D^{\lambda}_{x}}\int_{\mathbb{R}}(f_{1,b}\bar{f}_{2,b}+\bar{f}_{1,b}f_{2,b})w(x)d\xi d\sigma(x)ds\\ \notag
&& -2 n^{-1}\sum^d_{j=1} \mathbb{E}\int^t_0\int_{D^{\lambda}_x}\int^N_{-N} (f_1(s)\bar{f}_2(s)+\bar{f}_1(s)f_2(s)) \partial_{x_j} \lambda(x) d\xi dxds\\ \notag
&& +n^{-1} \mathbb{E}\int^t_0\int_{D^{\lambda}_x}\int^N_{-N} (f_1(s)\bar{f}_2(s)+\bar{f}_1(s)f_2(s))  w(x)\Delta_x \lambda(x) d\xi dxds\\
 \label{e-22}
&&+r^{N,\lambda}_t(\gamma,\delta)+\tilde{r}^{N,\lambda}_t(\gamma,\delta)+C\delta \gamma^{-1}+C(\gamma^2\delta^{-1}+\delta)+\tilde{L}^{N,\lambda}_1+\tilde{L}^{N,\lambda}_2+I_N,
  \end{eqnarray}
  where
  \begin{eqnarray*}
\tilde{L}^{N,\lambda}_1&=&\mathbb{E}\int^t_0\int_{D^{\lambda}_x}\int_{D_y}\int^N_{-N}\int^N_{-N} (f_1(s)\bar{f}_2(s)+\bar{f}_1(s)f_2(s)) a(\xi)\cdot \nabla_x\lambda(x) \rho^{\lambda}_{\gamma}(y-x)w(x)\psi_{\delta}(\xi-\zeta)d\xi d\zeta dy dx ds,\\
\tilde{L}^{N,\lambda}_2&=&-\sum^d_{j=1}\mathbb{E}\int^t_0\int_{D^{\lambda}_x}\int_{D_y}\int^N_{-N}\int^N_{-N} (f_1(s)\bar{f}_2(s)+\bar{f}_1(s)f_2(s)) a_j(\xi) \rho^{\lambda}_{\gamma}(y-x)\lambda(x)\psi_{\delta}(\xi-\zeta)d\xi d\zeta dy dx ds,\\
  I_N&=&C(\mu'_{q^n_1}(\pm N)+\mu'_{m^n_1}(\pm N)+\mu'_{q^n_2}(\pm N)+\mu'_{m^n_2}(\pm N)+(1+N^2)\mu'_{\nu^1}(\pm N)+(1+N^2)\mu'_{\nu^2}(\pm N)).
\end{eqnarray*}
Here, $\tilde{r}^{N,\lambda}_t(\gamma,\delta)$ is defined by (\ref{s-3}), the Young measures $\nu_1=\delta_{u^{n}(x;\vartheta_1)=\xi}$, $\nu_2=\delta_{u^{n}(y;\vartheta_2)=\zeta}$ and $m^n_1$, $q^{n}_1=\frac{1}{n}|\nabla_x u^{n}(x;\vartheta_1)|^2\delta_{u^{n}(x;\vartheta_1)=\xi}$, $m^n_2$, $q^{n}_2=\frac{1}{n}|\nabla_y u^{n}(y;\vartheta_2)|^2\delta_{u^{n}(y;\vartheta_2)=\zeta}$ are two pairs kinetic measures corresponding to $u^n(t,x;\vartheta_1)$ and $u^n(t,y;\vartheta_2)$, respectively. In addition, $\mu'_{q^n_i}, i=1,2$ are defined by the same way as $m^n_i$ in (\ref{a-2}) satisfying Lemma \ref{lem-4}, which implies that $\underset{N\rightarrow \infty}{{\rm{\limsup}}}\ I_N=0$.

\end{lemma}
\begin{proof}
Similar to Proposition \ref{prp-7}, the equations (\ref{ee-7-1}) and (\ref{eqq-17}) satisfied by $f$ and $\bar{f}$ can be strengthened to be only weak in $(x,\xi)$.
Compared with the proof of Proposition \ref{prp-1}, we only need to handle the additional terms $J^{\sharp}_1, J^{\sharp}_2$ generated by $n^{-1} \Delta_x u^{n}(t,x;\vartheta_1)$ and $n^{-1} \Delta_y u^{n}(t,y;\vartheta_2)$, respectively.

By integration by parts formula, we have
\begin{eqnarray*}
  J^{\sharp}_1&=& n^{-1} \mathbb{E}\int^t_0\int_{D^{\lambda}_x}\int_{D_y}\int^N_{-N}\int^N_{-N}\Psi_{\eta}(\xi,\zeta) f_1(s)\bar{f}_2(s) \Delta_x (\alpha^{\lambda}w(x))d\xi d\zeta dx dy ds\\
  && +n^{-1} \mathbb{E}\int^t_0\int_{D^{\lambda}_x}\int_{D_y}\int^N_{-N}\int^N_{-N}\Psi_{\eta}(\xi,\zeta) f_1(s)\bar{f}_2(s) \Delta_y \alpha^{\lambda}w(x)d\xi d\zeta dx dy ds\\
  && -\mathbb{E}\int^t_0\int_{D^{\lambda}_x}\int_{D_y}\int^N_{-N}\int^N_{-N}\Psi_{\eta}(\xi,\zeta)\alpha^{\lambda}w(x)d\nu^{1,-}_{s,x}(\xi)dxdq^n_2(s,y,\zeta)\\
 && -\mathbb{E}\int^t_0\int_{D^{\lambda}_x}\int_{D_y}\int^N_{-N}\int^N_{-N}\Psi_{\eta}(\xi,\zeta)\alpha^{\lambda}w(x)d\nu^{2,+}_{s,y}(\zeta)dydq^n_1(s,x,\xi)\\
  &=:&  \sum_{i=1}^4I_i,
\end{eqnarray*}
and
\begin{eqnarray*}
  J^{\sharp}_2&=& n^{-1} \mathbb{E}\int^t_0\int_{D^{\lambda}_x}\int_{D_y}\int^N_{-N}\int^N_{-N}\Psi_{\eta}(\xi,\zeta) \bar{f}_1(s)f_2(s) \Delta_x (\alpha^{\lambda}w(x))d\xi d\zeta dy dx ds\\
  && +n^{-1} \mathbb{E}\int^t_0\int_{D^{\lambda}_x}\int_{D_y}\int^N_{-N}\int^N_{-N}\Psi_{\eta}(\xi,\zeta) \bar{f}_1(s)f_2(s) \Delta_y \alpha^{\lambda}w(x)d\xi d\zeta dy dx ds\\
  && -\mathbb{E}\int^t_0\int_{D^{\lambda}_x}\int_{D_y}\int^N_{-N}\int^N_{-N}\Psi_{\eta}(\xi,\zeta)\alpha^{\lambda}w(x)d\nu^{1,-}_{s,x}(\xi)dxdq^n_2(s,y,\zeta)\\
 && -\mathbb{E}\int^t_0\int_{D^{\lambda}_x}\int_{D_y}\int^N_{-N}\int^N_{-N}\Psi_{\eta}(\xi,\zeta)\alpha^{\lambda}w(x)d\nu^{2,+}_{s,y}(\zeta)dydq^n_1(s,x,\xi).
\end{eqnarray*}
%
Clearly, by the definition of $q^n_1$ and $q^n_2$, we have
\begin{eqnarray}\notag
&&I_3+I_4\\ \notag
&=& -n^{-1} \mathbb{E}\int^t_0\int_{D^{\lambda}_x}\int_{D_y}\Psi_{\eta}(u^n(s,x;\vartheta_1),u^n(s,y;\vartheta_2))
\lambda(x)w(x)\rho^{\lambda}_{\gamma}(y-x)\psi_{\delta}(u^n(s,x;\vartheta_1)-u^n(s,y;\vartheta_2))\\ \notag
&&\quad \times|\nabla_x u^n(s,x;\vartheta_1)|^2dydxds\\
\notag
&& -n^{-1} \mathbb{E}\int^t_0\int_{D^{\lambda}_x}\int_{D_y}\Psi_{\eta}(u^n(s,x;\vartheta_1),u^n(s,y;\vartheta_2))
\lambda(x)w(x)\rho^{\lambda}_{\gamma}(y-x)\psi_{\delta}(u^n(s,x;\vartheta_1)-u^n(s,y;\vartheta_2))\\
\label{ee-20}
&& \quad \times|\nabla_y u^n(s,y;\vartheta_2)|^2dydxds.
\end{eqnarray}
As $\Delta_x \alpha=\Delta_y\alpha$, it follows that
\begin{eqnarray*}
I_1+I_2&=& 2 n^{-1} \mathbb{E}\int^t_0\int_{D^{\lambda}_x}\int_{D_y}\int^N_{-N}\int^N_{-N}\Psi_{\eta}(\xi,\zeta) f_1\bar{f}_2(\Delta_x \alpha)\lambda(x)w(x)d\xi d\zeta dy dx ds\\
&& +2n^{-1} \mathbb{E}\int^t_0\int_{D^{\lambda}_x}\int_{D_y}\int^N_{-N}\int^N_{-N}\Psi_{\eta}(\xi,\zeta) f_1\bar{f}_2 \nabla_x \alpha  \cdot \nabla_x (\lambda(x)w(x)) d\xi d\zeta dy dxds\\
&& +n^{-1} \mathbb{E}\int^t_0\int_{D^{\lambda}_x}\int_{D_y}\int^N_{-N}\int^N_{-N}\Psi_{\eta}(\xi,\zeta) f_1\bar{f}_2 \alpha \Delta_x (\lambda(x)w(x)) d\xi d\zeta dy dxds\\
&=:& K_1+K_2+K_3.
\end{eqnarray*}
By integration by parts formula, we get
\begin{eqnarray*}
K_1&=& -2n^{-1} \mathbb{E}\int^t_0\int_{D^{\lambda}_x}\int_{D_y}\int^N_{-N}\int^N_{-N}\Psi_{\eta}(\xi,\zeta) \nabla_x(f_1\lambda(x)w(x))\cdot(\nabla_y\bar{f}_2) \alpha d\xi d\zeta dy dxds\\
&=& -2n^{-1} \mathbb{E}\int^t_0\int_{D^{\lambda}_x}\int_{D_y}\int^N_{-N}\int^N_{-N}\Psi_{\eta}(\xi,\zeta) (\nabla_xf_1)\cdot(\nabla_y\bar{f}_2) \alpha\lambda(x)w(x) d\xi d\zeta dy dxds\\
&& -2n^{-1} \mathbb{E}\int^t_0\int_{D^{\lambda}_x}\int_{D_y}\int^N_{-N}\int^N_{-N}\Psi_{\eta}(\xi,\zeta) f_1\nabla_x(\lambda(x)w(x))\cdot(\nabla_y\bar{f}_2) \alpha d\xi d\zeta dy dxds\\
&=:& K_{1,1}+K_{1,2}.
\end{eqnarray*}
Hence
\begin{eqnarray}\label{ee-21}
  J^{\sharp}_1=K_{1,1}+K_{1,2}+K_2+K_3+I_3+I_4.
\end{eqnarray}
Based on (\ref{P-24}), using the same method as the proof of Theorem 3.3 in \cite{DHV} (the estimates of $J_2$), we get
\begin{eqnarray*}
  &&K_{1,1}\\
  &=&2n^{-1} \mathbb{E}\int^t_0\int_{D^{\lambda}_x}\int_{D_y}\Psi_{\eta}(u^n(s,x;\vartheta_1),u^n(s,y;\vartheta_2))
  \lambda(x)w(x)
  \rho^{\lambda}_{\gamma}(y-x)\psi_{\delta}(u^n(s,x;\vartheta_1)-u^n(s,y;\vartheta_2))\\
  && \times\nabla_x u^n(s,x;\vartheta_1)\cdot\nabla_y u^n(s,y;\vartheta_2)dy dxds.
\end{eqnarray*}
This together with (\ref{ee-20}) imply
\begin{eqnarray*}
&&K_{1,1}+I_3+I_4\\
&=& -n^{-1} \mathbb{E}\int^t_0\int_{D^{\lambda}_x}\int_{D_y}\Psi_{\eta}(u^n(s,x;\vartheta_1),u^n(s,y;\vartheta_2))
  \lambda(x)w(x)
  \rho^{\lambda}_{\gamma}(y-x)\psi_{\delta}(u^n(s,x;\vartheta_1)-u^n(s,y;\vartheta_2))\\
  && \times|\nabla_x u^n(s,x;\vartheta_1)-\nabla_y u^n(s,y;\vartheta_2)|^2dydxds\\
  &\leq& 0.
\end{eqnarray*}
Moreover,
\begin{eqnarray*}
  K_2&=&-2n^{-1} \mathbb{E}\int^t_0\int_{D^{\lambda}_x}\int_{D_y}\int^N_{-N}\int^N_{-N}\Psi_{\eta}(\xi,\zeta) f_1\bar{f}_2 \nabla_y \alpha  \cdot \nabla_x (\lambda(x)w(x)) d\xi d\zeta dy dxds\\
  &=&2n^{-1} \mathbb{E}\int^t_0\int_{D^{\lambda}_x}\int_{D_y}\int^N_{-N}\int^N_{-N}\Psi_{\eta}(\xi,\zeta)
  f_1\nabla_x(\lambda(x)w(x))\cdot(\nabla_y\bar{f}_2) \alpha d\xi d\zeta dy dxds\\
  &=&-K_{1,2}.
\end{eqnarray*}
By the definition of $w$, we have
\begin{eqnarray*}
K_3&=&-2  n^{-1} \sum^{d}_{j=1} \mathbb{E}\int^t_0\int_{D^{\lambda}_x}\int_{D_y}\int^N_{-N}\int^N_{-N}\Psi_{\eta}(\xi,\zeta) f_1\bar{f}_2 \rho^{\lambda}_{\gamma}(y-x)\psi_{\delta}(\xi-\zeta) \partial_{x_j} \lambda(x) d\xi d\zeta dydxds\\
&& +n^{-1} \mathbb{E}\int^t_0\int_{D^{\lambda}_x}\int_{D_y}\int^N_{-N}\int^N_{-N}\Psi_{\eta}(\xi,\zeta) f_1\bar{f}_2 \rho^{\lambda}_{\gamma}(y-x)\psi_{\delta}(\xi-\zeta) w(x)\Delta_x \lambda(x) d\xi d\zeta dydxds.
\end{eqnarray*}
Based on the above, in view of (\ref{ee-21}), we conclude that $ J^{\sharp}_1\leq K_3$.
A similar estimates also holds for $J^{\sharp}_2$.

Define
\begin{eqnarray}\notag
  \tilde{r}^{N,\lambda}_t(\gamma,\delta)&:=&-2  n^{-1} \sum^{d}_{j=1} \mathbb{E}\int^t_0\int_{D^{\lambda}_x}\int_{D_y}\int^N_{-N}\int^N_{-N} (f_1\bar{f}_2+\bar{f}_1f_2) \rho^{\lambda}_{\gamma}(y-x)\psi_{\delta}(\xi-\zeta) \partial_{x_j} \lambda(x) d\xi d\zeta dy dxds\\ \notag
  && +2n^{-1} \sum^{d}_{j=1} \mathbb{E}\int^t_0\int_{D^{\lambda}}\int^N_{-N} (f_1\bar{f}_2+\bar{f}_1f_2)  \partial_{x_j} \lambda(x) d\xi dxds\\ \notag
&& +n^{-1} \mathbb{E}\int^t_0\int_{D^{\lambda}_x}\int_{D_y}\int^N_{-N}\int^N_{-N} (f_1\bar{f}_2+\bar{f}_1f_2) \rho^{\lambda}_{\gamma}(y-x)\psi_{\delta}(\xi-\zeta) w(x)\Delta_x \lambda(x) d\xi d\zeta dy dxds\\
\label{s-3}
&& -n^{-1} \mathbb{E}\int^t_0\int_{D^{\lambda}}\int^N_{-N} (f_1\bar{f}_2+\bar{f}_1f_2) w(x)\Delta_x \lambda(x) d\xi dxds.
\end{eqnarray}
Arguing similarly as in the proof of (\ref{qq-4}) and (\ref{eqq-8}), it follows that
\begin{eqnarray}\label{s-4}
\lim_{\gamma, \delta\rightarrow 0}\sum^M_{i=0}|\tilde{r}^{N,\lambda_i}_t(\gamma,\delta)|=0, \quad \lim_{\gamma, \delta\rightarrow 0}\sum^M_{i=0}\int^T_0|\tilde{r}^{N,\lambda_i}_t(\gamma,\delta)|dt=0.
\end{eqnarray}
Finally, proceeding as Theorem \ref{thm-1},  we get the desired result by taking $\tilde{L}^{N,\lambda}_1:=L_3$ and $\tilde{L}^{N,\lambda}_2:=L_4$.

\end{proof}

The following result states that  $u^n$ converges to $u$ in the
space $L^1(\Omega\times [0,T];L^1_{w;x})$. 
\begin{prp}\label{prp-12}
   \begin{eqnarray}\label{rr-1}
     \lim_{n\rightarrow \infty}  \mathbb{E}\int^T_0\|u(t;\vartheta)-u^{n}(t;\vartheta)\|_{L^1_{w;x}}dt=0.
   \end{eqnarray}
\end{prp}
\begin{proof}
 Let $f_1(t):=f_1(t,x;\xi)=I_{u(t,x;\vartheta)>\xi}$  and $f_2(t):=f_2(t,y;\zeta)=I_{u^{n}(t,y;\vartheta)>\zeta}$ with the corresponding data $f_{1,0}=I_{\vartheta>\xi}$, $f_{2,0}=I_{\vartheta>\zeta}$, $f_{1,b}=I_{0>\xi}$, $f_{2,b}=I_{0>\zeta}$.
  The corresponding kinetic measures are denoted by $m$ and $(m^{n},q^n)$.

Using the same method as in the proof of Proposition \ref{prp-1}, we only need to deal with
the additional terms generated by the term $n^{-1} \Delta_y u^{n}(t,y;\vartheta)$:
\begin{eqnarray*}
  J^*&=&n^{-1} \mathbb{E}\int^t_0\int_{D^{\lambda}_x}\int_{D_y}\int^N_{-N}\int^N_{-N}\Psi_{\eta}(\xi,\zeta)(f_1\bar{f}_2+\bar{f}_1f_2)\Delta_y \alpha^{\lambda} w(x)d\xi d\zeta dydxds\\
  && +\mathbb{E}\int^t_0\int_{D^{\lambda}_x}\int_{D_y}\int^N_{-N}\int^N_{-N}\Psi_{\eta}(\xi,\zeta)f^-_1w(x)\partial_{\zeta}\alpha^{\lambda} d\xi dx dq^{n}(s,y,\zeta)\\
  && -\mathbb{E}\int^t_0\int_{D^{\lambda}_x}\int_{D_y}\int^N_{-N}\int^N_{-N}\Psi_{\eta}(\xi,\zeta)\bar{f}^+_1w(x)\partial_{\zeta}\alpha^{\lambda}  d\xi dx dq^{n}(s,y,\zeta)\\
  &=:&J^*_1+ J^*_2+J^*_3,
\end{eqnarray*}
where $q^n=\frac{1}{n}|\nabla_y u^{n}|^2\delta_{u^{n}=\zeta}$.

 Similar to the estimates of $I_4$ and $I_9$ in Proposition \ref{prp-1}, we can show that $J^*_2+J^*_3\leq C\mu'_{q^n}(\pm N)$ with $\underset{N\rightarrow \infty}{{\rm{\limsup}}}\ \mu'_{q^n}(\pm N)=0$.
Define
\[
\Upsilon(\xi,\zeta)=\int^{\infty}_{\zeta}\int^{\xi}_{-\infty}\Psi_{\eta}(\xi',\zeta')
\psi_{\delta}(\xi'-\zeta')d\xi'd\zeta'.
\]
We have $\Upsilon(\xi,\zeta)\leq C(|\xi|+|\zeta|+\delta)$.

Then, by the boundedness of $w$, it follows that
\begin{eqnarray*}
&&\frac{1}{n}\mathbb{E}\int^t_0\int_{D^{\lambda}_x}\int_{D_y}\int^N_{-N}\int^N_{-N}\Psi_{\eta}(\xi,\zeta)f_1\bar{f}_2\Delta_y \alpha^{\lambda} w(x)d\xi d\zeta dydxds\\
&=&-\frac{1}{n}\mathbb{E}\int^t_0\int_{D^{\lambda}_x}\int_{D_y}\int^N_{-N}\int^N_{-N}f_1\bar{f}_2\Delta_y \rho^{\lambda}_{\gamma}(y-x) \partial_{\zeta}\partial_{\xi}\Upsilon(\xi,\zeta)\lambda(x)w(x)d\xi d\zeta dydxds\\
 &=& \frac{1}{n}\mathbb{E}\int^t_0\int_{D^{\lambda}_x}\int_{D_y}\Delta_y \rho^{\lambda}_{\gamma}(y-x)\lambda(x)w(x)\int^N_{-N}\int^N_{-N} \Upsilon(\xi,\zeta)d\nu^1_{s,x}\otimes \nu^2_{s,y}(\xi,\zeta)dydxds\\
 &\leq & C\frac{1}{n}\mathbb{E}\int^t_0\int_{D^{\lambda}_x}\int_{D_y} |\Delta_y \rho^{\lambda}_{\gamma}(y-x)|\lambda(x)w(x)\int_{\mathbb{R}^2}(|\xi|+|\zeta|+\delta)d\nu^1_{s,x}\otimes \nu^2_{s,y}(\xi,\zeta)dydxds\\
 &\leq & Cn^{-1}\gamma^{-2}+Cn^{-1}\gamma^{-2}\delta\int^t_0\int_{D^{\lambda}_x}\lambda(x)dxds.
\end{eqnarray*}
Moreover, we have the same upper bound for $\frac{1}{n}\mathbb{E}\int^t_0\int_{D^{\lambda}_x}\int_{D_y}\int^N_{-N}\int^N_{-N}\Psi_{\eta}(\xi,\zeta)\bar{f}_1f_2\Delta_y \alpha^{\lambda} w(x)d\xi d\zeta dydxds$, which implies
\[
J^*\leq Cn^{-1}\gamma^{-2}+Cn^{-1}\gamma^{-2}\delta\int^t_0\int_{D^{\lambda}_x}\lambda(x)dxds+C\mu'_{q^n}(\pm N).
\]
Due to Proposition 4.1, we obtain that for any $0\leq t\leq T$,
\begin{eqnarray}\notag
 &&\mathbb{E} \int_{D^{\lambda}_x}\int_{D_y}\int^N_{-N}\int^N_{-N}(f^{\pm}_1(t)\bar{f}^{\pm}_2(t)+\bar{f}^{\pm}_1(t)f^{\pm}_2(t))\rho^{\lambda}_{\gamma}(y-x)\psi_{\delta}(\xi-\zeta)\lambda(x)w(x)d\xi d\zeta  dydx\\ \notag
 &\leq& \mathbb{E} \int_{D^{\lambda}_x}\int_{D_y}\int^N_{-N}\int^N_{-N}(f_{1,0}\bar{f}_{2,0}+\bar{f}_{1,0}f_{2,0})\rho^{\lambda}_{\gamma}(y-x)\psi_{\delta}(\xi-\zeta)\lambda(x)w(x)d\xi d\zeta  dydx\\
 \label{eqq-6}
&& +\sum_{i=1}^5\tilde{J}_i+I_N+Cn^{-1}\gamma^{-2}+Cn^{-1}\gamma^{-2}\delta\int^t_0\int_{D^{\lambda}_x}\lambda(x)dxds,
  \end{eqnarray}
  where $\tilde{J}_i$ are the corresponding terms to $J_i$, $i=1,\dots, 5$ of Proposition 4.1 with $f_1=I_{u(t;\vartheta)>\xi}$ and $f_2=I_{u^{n}(t;\vartheta)>\zeta}$ and
  \[
  I_N=C(\mu'_{m}(\pm N)+\mu'_{m^n}(\pm N)+\mu'_{q^n}(\pm N)+(1+N^2)\mu'_{\nu^1}(\pm N)+(1+N^2)\mu'_{\nu^2}(\pm N))
  \]
satisfying $\underset{N\rightarrow \infty}{{\rm{\limsup}}}\ I_N=0$.
Notice that $\tilde{J}_4\leq 0$, then by the same method as for the proof of Theorem \ref{thm-1} and integrating $t$ from $0$ to $T$, we get
%
%
 \begin{eqnarray*}\notag
 &&\mathbb{E}\int^T_0\int_{ D}\int^N_{-N}(f^{\pm}_1(t)\bar{f}^{\pm}_2(t)+\bar{f}^{\pm}_1(t)f^{\pm}_2(t))w(x)d\xi dxdt\\ \notag
 &\leq& T\mathbb{E} \int_{D}\int^N_{-N}(f_{1,0}\bar{f}_{2,0}+\bar{f}_{1,0}f_{2,0})w(x)d\xi dx\\
 && +|a(0)|\mathbb{E}\int^T_0\int^t_0\int_{\partial D}\int_{\mathbb{R}}(f_{1,b}\bar{f}_{2,b}+\bar{f}_{1,b}f_{2,b})w(x)d\xi d\sigma(x)dsdt+CMT\delta\gamma^{-1}\\
 \notag
&&+CTM(\gamma^2\delta^{-1}+\delta)+TMI_N+CMTn^{-1}\gamma^{-2}+CT^2n^{-1}\gamma^{-2}\delta+\sum^M_{i=0}\mathcal{E}^{N,\lambda_i}_0(\gamma,\delta)T
\\
&& +\sum^M_{i=0}\Big(\int^T_0r^{N,\lambda_i}_t(\gamma,\delta)dt
+\int^T_0\mathcal{E}^{N,\lambda_i}_t(\gamma,\delta)dt\Big),
\end{eqnarray*}
where error terms $\mathcal{E}^{N,\lambda}_t(\gamma,\delta)$ and $r^{N,\lambda}_t(\gamma,\delta)$ are defined by (\ref{qq-3}) and (\ref{eqq-1}), respectively.

Noting that $\limsup_{N\rightarrow \infty}I_N=0$, and by the dominated convergence theorem, we have
\begin{eqnarray*}
\lim_{N\rightarrow \infty}\mathbb{E} \int_{D}\int^N_{-N}(f_{1,0}\bar{f}_{2,0}+\bar{f}_{1,0}f_{2,0})w(x)d\xi dx=\mathbb{E} \int_{D}\int_{\mathbb{R}}(f_{1,0}\bar{f}_{2,0}+\bar{f}_{1,0}f_{2,0})w(x)d\xi dx.
\end{eqnarray*}
Moreover, employing a similar method as in the proof of (\ref{a-6}) with $f_1=I_{u(t;\vartheta)>\xi}$ and $f_2=I_{u^{n}(t;\vartheta)>\zeta}$, by utilizing (\ref{rr-5-1}), we have
\begin{eqnarray*}
&&\lim_{N\rightarrow \infty}\mathbb{E}\int^T_0\int_{ D}\int^N_{-N}(f^{\pm}_1(t)\bar{f}^{\pm}_2(t)+\bar{f}^{\pm}_1(t)f^{\pm}_2(t))w(x)d\xi dxdt\\
&=&\mathbb{E}\int^T_0\int_{ D}\int_{\mathbb{R}}(f^{\pm}_1(t)\bar{f}^{\pm}_2(t)+\bar{f}^{\pm}_1(t)f^{\pm}_2(t))w(x)d\xi dxdt, \quad \text{uniformly\ on}\ n.
\end{eqnarray*}
Hence, for any $\iota>0$, there exists a big enough constant $N_0$ independent of $\gamma,\delta,n$ such that
\begin{eqnarray*}\notag
 &&\mathbb{E}\int^T_0\int_{ D}\int_{\mathbb{R}}(f^{\pm}_1(t)\bar{f}^{\pm}_2(t)+\bar{f}^{\pm}_1(t)f^{\pm}_2(t))w(x)d\xi dxdt\\ \notag
 &\leq& \mathbb{E}\int^T_0\int_{ D}\int^{N_0}_{-N_0}(f^{\pm}_1(t)\bar{f}^{\pm}_2(t)+\bar{f}^{\pm}_1(t)f^{\pm}_2(t))w(x)d\xi dxdt+\iota\\ \notag
 &\leq& T\mathbb{E} \int_{D}\int_{\mathbb{R}}(f_{1,0}\bar{f}_{2,0}+\bar{f}_{1,0}f_{2,0})w(x)d\xi dx +(T+1)\iota\\
  && +|a(0)|\mathbb{E}\int^T_0\int^t_0\int_{\partial D}\int_{\mathbb{R}}(f_{1,b}\bar{f}_{2,b}+\bar{f}_{1,b}f_{2,b})w(x)d\xi d\sigma(x)dsdt\\
&& +CTM(\gamma^2\delta^{-1}+\delta)+CMTn^{-1}\gamma^{-2}+CT^2n^{-1}\gamma^{-2}\delta+\sum^M_{i=0}\mathcal{E}^{N_0,\lambda_i}_0(\gamma,\delta)T
\\
&& +\sum^M_{i=0}\Big(\int^T_0r^{N_0,\lambda_i}_t(\gamma,\delta)dt
+\int^T_0\mathcal{E}^{N_0,\lambda_i}_t(\gamma,\delta)dt\Big).
\end{eqnarray*}
Taking $\delta=\gamma^{\frac{4}{3}}$, $\gamma=n^{-\frac{1}{3}}$ and letting $n\rightarrow \infty$ (in this case, $\gamma,\delta\rightarrow 0$), we deduce from  (\ref{qq-5})-(\ref{eqq-8}) and (\ref{eqq-9})
 that
\begin{eqnarray*}
&&\lim_{n\rightarrow \infty}\mathbb{E}\int^T_0\int_{ D}\int_{\mathbb{R}}(f^{\pm}_1(t)\bar{f}^{\pm}_2(t)+\bar{f}^{\pm}_1(t)f^{\pm}_2(t))w(x)d\xi dxdt\\ \notag
 &\leq& T\mathbb{E} \int_{D}\int_{\mathbb{R}}(f_{1,0}\bar{f}_{2,0}+\bar{f}_{1,0}f_{2,0})w(x)d\xi dx +(T+1)\iota\\
  && +|a(0)|\mathbb{E}\int^T_0\int^t_0\int_{\partial D}\int_{\mathbb{R}}(f_{1,b}\bar{f}_{2,b}+\bar{f}_{1,b}f_{2,b})w(x)d\xi d\sigma(x)dsdt.
\end{eqnarray*}
Since $\iota$ is arbitrary, it follows that
\begin{eqnarray*}
&&\lim_{n\rightarrow \infty}\mathbb{E}\int^T_0\int_{ D}\int_{\mathbb{R}}(f^{\pm}_1(t)\bar{f}^{\pm}_2(t)+\bar{f}^{\pm}_1(t)f^{\pm}_2(t))w(x)d\xi dxdt\\ \notag
 &\leq& T\mathbb{E} \int_{D}\int_{\mathbb{R}}(f_{1,0}\bar{f}_{2,0}+\bar{f}_{1,0}f_{2,0})w(x)d\xi dx\\
  && +|a(0)|\mathbb{E}\int^T_0\int^t_0\int_{\partial D}\int_{\mathbb{R}}(f_{1,b}\bar{f}_{2,b}+\bar{f}_{1,b}f_{2,b})w(x)d\xi d\sigma(x)dsdt.
\end{eqnarray*}
Note that $f_1=I_{u(t;\vartheta)>\xi}$ and $f_2=I_{u^{n}(t;\vartheta)>\zeta}$ with the corresponding data $f_{1,0}=I_{\vartheta>\xi}, f_{2,0}=I_{\vartheta>\zeta}$ and $f^b_1=I_{0>\xi}, f^b_2=I_{0>\zeta}$. Applying (\ref{equ-1}), we get
\begin{eqnarray*}
\lim_{n\rightarrow \infty}\mathbb{E}\int^T_0\int_{ D}|u(t;\vartheta)-u^n(t;\vartheta)|w(x)dxdt= 0.
\end{eqnarray*}

\end{proof}

Now, we are in a position to give the proof of Theorem \ref{thm-10}.
\begin{proof}[\textbf{Proof of Theorem \ref{thm-10}}]
 From Proposition \ref{prp-12}, we know that there exists a subset $\mathcal{I}\subset [0,T]$ with $|\mathcal{I}|=T$  and a (non-relabelled) subsequence such that
\begin{eqnarray}\label{rr-8}
   \lim_{n\rightarrow \infty}\mathbb{E}\|u(t;\vartheta)-u^{n}(t;\vartheta)\|_{L^1_{w;x}}=0, \quad {\rm{for \ every\ }}t\in \mathcal{I}.
\end{eqnarray}
In the following, we will prove that for any $\tau>0$, there exists $\varsigma>0$ such that
\begin{eqnarray}\label{rr-9}
   \mathbb{E}\|u^{n}(t;\vartheta)-u^{n}(t';\vartheta)\|_{L^1_{w;x}}<\tau,
\end{eqnarray}
for every $n\geq 1$ and $t>t'\in \mathcal{I}$ with $|t-t'|<\varsigma$. For simplicity, we write $u^{n}(t,x;\vartheta)=u^{n}(t,x)$ and $u^{n}(t',x;\vartheta)=u^{n}(t',x)$.

Noting that $\rho^{\lambda}_{\gamma}(y-x)=0$ on $D^{\lambda}_x\times D^c$, we have
\begin{eqnarray*}\notag
   &&\mathbb{E}\int_{ D}|u^{n}(t,x)-u^{n}(t',x)|w(x)dx\\ \notag
   &=& \sum^M_{i=0} \mathbb{E}\int_{ D^{\lambda_i}_x}|u^{n}(t,x)-u^{n}(t',x)|w(x)\lambda_i(x)dx\\ \notag
   &=&  \sum^M_{i=0} \mathbb{E}\int_{ D^{\lambda_i}_x}\int_{D_y}|u^{n}(t,x)-u^{n}(t',x)|w(x)\lambda_i(x)\rho^{\lambda_i}_{\gamma}(y-x)dydx\\ \notag
   &\leq& \sum^M_{i=0} \mathbb{E}\int_{ D^{\lambda_i}_x}\int_{D_y}|u^{n}(t,x)-u^{n}(t',y)|\lambda_i(x)w(x)\rho^{\lambda_i}_{\gamma}(y-x)dy dx\\ \notag
   && +\sum^M_{i=0} \mathbb{E}\int_{ D^{\lambda_i}_x}\int_{D_y}|u^{n}(t',x)-u^{n}(t',y)|\lambda_i(x)w(x)\rho^{\lambda_i}_{\gamma}(y-x)dy dx\\
    &=:& K^n_1(\gamma)+K^n_2(\gamma).
\end{eqnarray*}
Letting $f_1(t'):=f_1(t',x,\xi)=I_{u^{n}(t',x)>\xi}, f_2(t'):=f_2(t',y,\zeta)=I_{u^{n}(t',y)>\zeta}$ with $f_{1,0}=I_{\vartheta>\xi}, f_{2,0}=I_{\vartheta>\zeta}$, it follows that
\begin{eqnarray*}\notag
  K^n_2(\gamma) =\sum^M_{i=0} \mathbb{E}\int_{ D^{\lambda_i}_x}\int_{D_y}\int_{\mathbb{R}}(f_1(t',x,\xi)\bar{f}_2(t',y,\xi)+\bar{f}_1(t',x,\xi)f_2(t',y,\xi))
  \lambda_i(x)w(x)\rho^{\lambda_i}_{\gamma}(y-x)d\xi dydx.
  \end{eqnarray*}
  Define
  \begin{eqnarray*}\notag
 \tilde{K}^{n,N}_2(\gamma):=\sum^M_{i=0} \mathbb{E}\int_{ D^{\lambda_i}_x}\int_{D_y}\int^N_{-N}(f_1(t',x,\xi)\bar{f}_2(t',y,\xi)+\bar{f}_1(t',x,\xi)f_2(t',y,\xi))
  \lambda_i(x)w(x)\rho^{\lambda_i}_{\gamma}(y-x)d\xi dydx,
   \end{eqnarray*}
 then, employing a similar method as in the proof of (\ref{a-6}) with $f_1(t'), f_2(t')$, and by utilizing (\ref{rr-5-1}), it follows that
  \begin{eqnarray}\label{rrr-7}
  K^n_2(\gamma) = \lim_{N\rightarrow \infty}\tilde{K}^{n,N}_2(\gamma),\quad \text{uniformly \ on} \ n.
   \end{eqnarray}
Applying (\ref{rrr-2}), we get for any $N$,
   \begin{eqnarray*}
\tilde{K}^{n,N}_2(\gamma)
  &\leq&
 C\delta+2M\Upsilon^{N}(\delta)\\
 && +\sum^M_{i=0} \mathbb{E}\int_{ D^{\lambda_i}_x}\int_{D_y}\int^N_{-N}\int^N_{-N}(f_1(t')\bar{f}_2(t')+\bar{f}_1(t')f_2(t'))\lambda_i(x)w(x)
 \rho^{\lambda_i}_{\gamma}(y-x)\psi_{\delta}(\xi-\zeta)d\xi d\zeta dydx\\
 &=:& C\delta+2M\Upsilon^{N}(\delta)+J^{n,N}(\gamma,\delta),
 \end{eqnarray*}
 where $\lim_{\delta\rightarrow 0}\Upsilon^{N}(\delta)=0$.

Applying Lemma \ref{lem-3} with $\vartheta_1=\vartheta_2=\vartheta$, and noticing $\tilde{L}^{N}_2\leq 0$, we get
\begin{eqnarray}\notag
J^{n,N}(\gamma, \delta)
&\leq& \sum^M_{i=0} \Big[\mathbb{E}\int_{ D^{\lambda_i}_x}\int^{N}_{-N}(f_{1,0}\bar{f}_{2,0}+\bar{f}_{1,0}f_{2,0})\lambda_i(x)w(x)d\xi dx+\mathcal{E}^{N,\lambda_i}_0(\gamma,\delta)\\ \notag
&&+|a(0)| \int^{t'}_0\int_{\partial D^{\lambda_i}_x}\int_{\mathbb{R}}(f_{1,b}\bar{f}_{2,b}+\bar{f}_{1,b}f_{2,b})w(x)d\xi d\sigma(x)ds\\ \notag
&& -2 n^{-1}\sum^d_{j=1} \mathbb{E}\int^{t'}_0\int_{D^{\lambda_i}_x}\int^{N}_{-N} (f_1(s)\bar{f}_2(s)+\bar{f}_1(s)f_2(s)) \partial_{x_j} \lambda_i(x) d\xi dxds\\ \notag
&& +n^{-1} \mathbb{E}\int^{t'}_0\int_{D^{\lambda_i}_x}\int^{N}_{-N} (f_1(s)\bar{f}_2(s)+\bar{f}_1(s)f_2(s))  w(x)\Delta_x \lambda_i(x) d\xi dxds\\ \notag
&& +r^{N,\lambda_i}_{t'}(\gamma, \delta)+\tilde{r}^{N,\lambda_i}_{t'}(\gamma, \delta)+C\delta \gamma^{-1}+C(\gamma^2\delta^{-1}+\delta)+\tilde{L}^{N,\lambda_i}_1+I_{N}\Big]\\ \notag
&\leq& \mathbb{E}\int_{ D}\int^{N}_{-N}(f_{1,0}\bar{f}_{2,0}+\bar{f}_{1,0}f_{2,0})w(x)d\xi dx+\sum^M_{i=0}\mathcal{E}^{N,\lambda_i}_0(\gamma,\delta)
+\sum^M_{i=0}r^{N,\lambda_i}_{t'}(\gamma, \delta)\\
\label{eqq-33}
&&+\sum^M_{i=0}\tilde{r}^{N,\lambda_i}_{t'}(\gamma, \delta)+CM\delta \gamma^{-1}+CM(\gamma^2\delta^{-1}+\delta)+\sum^M_{i=0}(\tilde{L}^{N,\lambda_i}_1+I_{N}),
\end{eqnarray}
where we have used the facts:
\begin{eqnarray*}
  &&-2 n^{-1}\sum^M_{i=0}\sum^d_{j=1} \mathbb{E}\int^{t'}_0\int_{D^{\lambda_i}_x}\int^{N}_{-N} (f_1(s)\bar{f}_2(s)+\bar{f}_1(s)f_2(s)) \partial_{x_j} \lambda_i(x) d\xi dxds\\
  &=& -2 n^{-1}\sum^d_{j=1} \mathbb{E}\int^{t'}_0\int_{D^{\lambda_i}_x}\int^{N}_{-N} (f_1(s)\bar{f}_2(s)+\bar{f}_1(s)f_2(s)) \partial_{x_j} \Big(\sum^M_{i=0}\lambda_i(x)\Big) d\xi dxds=0,
\end{eqnarray*}
and
\begin{eqnarray*}
  &&n^{-1} \sum^M_{i=0}\mathbb{E}\int^{t'}_0\int_{D^{\lambda_i}_x}\int^{N}_{-N} (f_1(s)\bar{f}_2(s)+\bar{f}_1(s)f_2(s))  w(x)\Delta_x \lambda_i(x) d\xi dxds\\
  &=&n^{-1} \mathbb{E}\int^{t'}_0\int_{D^{\lambda_i}_x}\int^{N}_{-N} (f_1(s)\bar{f}_2(s)+\bar{f}_1(s)f_2(s))  w(x)\Delta_x \Big(\sum^M_{i=0}\lambda_i(x)\Big) d\xi dxds=0.
\end{eqnarray*}

Notice that $\lim_{N\rightarrow \infty}\sum^M_{i=0}I_{N}=0$ and by the Lebesgue dominated convergence theorem,
\[
\lim_{N\rightarrow \infty}\mathbb{E}\int_{ D}\int^N_{-N}(f_{1,0}\bar{f}_{2,0}+\bar{f}_{1,0}f_{2,0})w(x)d\xi dx=\mathbb{E}\int_{ D}\int_{\mathbb{R}}(f_{1,0}\bar{f}_{2,0}+\bar{f}_{1,0}f_{2,0})w(x)d\xi dx=0.
\]
Then, by (\ref{rrr-7}), we know that for any $\iota>0$, there exists a big enough constant $N_0$ independent of $\gamma,\delta,n$ such that
\begin{eqnarray*}
|K^n_2(\gamma)-\tilde{K}^{n,N_0}_2(\gamma)|+\mathbb{E}\int_{ D}\int^{N_0}_{-N_0}(f_{1,0}\bar{f}_{2,0}+\bar{f}_{1,0}f_{2,0})w(x)d\xi dx+\sum^M_{i=0}I_{N_0}<\iota.
\end{eqnarray*}
In view of (\ref{eqq-33}), we have
\begin{eqnarray}\notag
K^n_2(\gamma)
&\leq& \iota+\tilde{K}^{n,N_0}_2(\gamma)-\mathbb{E}\int_{ D}\int^{N_0}_{-N_0}(f_{1,0}\bar{f}_{2,0}+\bar{f}_{1,0}f_{2,0})w(x)d\xi dx-\sum^M_{i=0}I_{N_0}\\ \notag
&\leq& \iota+C\delta+2M\Upsilon^{N_0}(\delta)+J^{n,N_0}(\gamma,\delta)-\mathbb{E}\int_{ D}\int^{N_0}_{-N_0}(f_{1,0}\bar{f}_{2,0}+\bar{f}_{1,0}f_{2,0})w(x)d\xi dx-\sum^M_{i=0}I_{N_0}\\ \notag
&\leq& \iota+C\delta+2M\Upsilon^{N_0}(\delta)
 +\sum^M_{i=0}\mathcal{E}^{N_0,\lambda_i}_0(\gamma,\delta)
+\sum^M_{i=0}r^{N_0,\lambda_i}_{t'}(\gamma, \delta)+\sum^M_{i=0}\tilde{r}^{N_0,\lambda_i}_{t'}(\gamma, \delta)\\
\label{s-5}
&&+CM\delta \gamma^{-1}+CM(\gamma^2\delta^{-1}+\delta)+\sum^M_{i=0}\tilde{L}^{N_0,\lambda_i}_1.
\end{eqnarray}
From (\ref{qq-5}), (\ref{eqq-2}), (\ref{eqq-34}) and (\ref{s-4}), we have
\begin{eqnarray*}
\lim_{\gamma,\delta\rightarrow 0}\sum^M_{i=0}\mathcal{E}^{N_0,\lambda_i}_0(\gamma,\delta)=0, \quad \lim_{\gamma,\delta\rightarrow 0}\sum^M_{i=0}r^{N_0,\lambda_i}_{t'}(\gamma, \delta)=0,\\
\lim_{\gamma,\delta\rightarrow 0}\sum^M_{i=0}\tilde{L}^{N_0,\lambda_i}_1=0, \quad \lim_{\gamma,\delta\rightarrow 0}\sum^M_{i=0}\tilde{r}^{N_0,\lambda_i}_{t'}(\gamma, \delta)=0.
\end{eqnarray*}
Taking $\delta=\gamma^{\frac{4}{3}}$, which are independent of $n$, and letting $\gamma\rightarrow 0$ in (\ref{s-5}), we get
\begin{eqnarray*}
\lim_{\gamma,\delta\rightarrow 0}K^n_2(\gamma)\leq \iota.
\end{eqnarray*}
Since $\iota$ is arbitrary,  we deduce that
\begin{eqnarray}\label{s-6}
\lim_{\gamma\rightarrow 0} K^n_2(\gamma)=0,\quad \text{uniformly\ on}\  n.
\end{eqnarray}


Now, we focus on the estimates of $K^n_1(\gamma)$.
%
Let $\eta_{\varepsilon}$ be a symmetric approximation of $|\cdot|$ given by
\[
\eta_{\varepsilon}(0)=\eta'_{\varepsilon}(0)=0,\ \eta''_{\varepsilon}(r)=\varepsilon^{-1}\tilde{\eta}(\varepsilon^{-1}|r|)
\]
for some non-negative $\tilde{\eta}\in C^{\infty}(\mathbb{R})$ which is bounded by 2, supported in $(0,1)$ and integrates to $1$. The following properties of $\eta_{\varepsilon}$ hold:
\begin{eqnarray}\label{rr-6}
  |\eta_{\varepsilon}(r)-|r||\lesssim \varepsilon, \ {\rm{supp}}\ \eta''_{\varepsilon}\subset [-\varepsilon, \varepsilon],\  |\eta''_{\varepsilon}(r)|\leq 2\varepsilon^{-1},\ \eta'_{\varepsilon}(r)\in [0,\frac{1}{2}].
\end{eqnarray}
This implies that
\begin{eqnarray}\notag
K^n_{1}(\gamma) &\lesssim &\varepsilon \|w\|_{L^{1}_x}+\sum^M_{i=0}\mathbb{E}\int_{D^{\lambda_i}_x}\int_{D_y}\eta_{\varepsilon}(u^{n}(t,x)-u^{n}(t',y))
\lambda_i(x)w(x)\rho^{\lambda_i}_{\gamma}(y-x)dydx\\ \notag
&=:& \varepsilon \|w\|_{L^{1}_x}+L.
\end{eqnarray}
 Under Hypotheses (H1)-(H3), we may apply the generalized It\^{o} formula from Proposition A.1 in \cite{DHV} to deduce that
\begin{eqnarray*}
L&=& \sum^M_{i=0}\mathbb{E}\int_{D_y}\int_{D^{\lambda_i}_x}
\eta_{\varepsilon}(u^{n}(t',x)-u^{n}(t',y))\lambda_i(x)w(x)\rho^{\lambda_i}_{\gamma}(y-x)dxdy\\
  && -\frac{1}{n}\sum^M_{i=0}\mathbb{E}\int_{D_y}\int^{t}_{t'}\int_{D^{\lambda_i}_x}
\eta''_{\varepsilon}(u^{n}(s,x)-u^{n}(t',y))|\nabla u^{n}(s,x)|^2 \lambda_i(x)w(x)\rho^{\lambda_i}_{\gamma}(y-x)dxds dy\\
  && -\frac{1}{n}\sum^M_{i=0}\mathbb{E}\int_{D_y}\int^{t}_{t'}\int_{D^{\lambda_i}_x}
\eta'_{\varepsilon}(u^{n}(s,x)-u^{n}(t',y))\nabla u^{n}(s,x)\cdot \nabla(\lambda_i(x)w(x)\rho^{\lambda_i}_{\gamma}(y-x))dxdsdy\\
  && -\sum^M_{i=0}\mathbb{E}\int_{D_y}\int^{t}_{t'}\int_{D^{\lambda_i}_x}
\eta'_{\varepsilon}(u^{n}(s,x)-u^{n}(t',y))a(u^{n}(s,x))\cdot \nabla u^{n}(s,x) \lambda_i(x)w(x)\rho^{\lambda_i}_{\gamma}(y-x)dxdsdy\\
  && +\frac{1}{2}\sum^M_{i=0}\mathbb{E}\int_{D_y}\int^{t}_{t'}\int_{D^{\lambda_i}_x}
\eta''_{\varepsilon}(u^{n}(s,x)-u^{n}(t',y))
  G^2(x,u^{n}(s,x))\lambda_i(x)w(x)\rho^{\lambda_i}_{\gamma}(y-x)dxdsdy\\
  &=:&\sum_{i=1}^5I_i.
\end{eqnarray*}
By (\ref{rr-6}), we have
\begin{eqnarray*}
 I_1&\lesssim& \varepsilon\|w\|_{L^{1}_x}+\sum^M_{i=0}\mathbb{E}\int_{D^{\lambda_i}_x}
  \int_{D_y}|u^{n}(t',x)-u^{n}(t',y)|\lambda_i(x)w(x)\rho^{\lambda_i}_{\gamma}(y-x)dydx\\
  &\leq&\varepsilon\|w\|_{L^{1}_x}+K^n_2(\gamma).
\end{eqnarray*}
Clearly, $I_2\leq 0$.
Using (\ref{rr-5-1}), we deduce that
\begin{eqnarray*}
I_3&\leq& \frac{\gamma^{-1}}{n}\|w\|_{W^{1,\infty}_x}(t-t')^{\frac{1}{2}}\mathbb{E}\Big[\int^t_{t'}\|\nabla u^{n}\|^2_{L^2_x}ds\Big]^{\frac{1}{2}}\\
&\leq& \gamma^{-1}\|w\|_{W^{1,\infty}_x}(t-t')^{\frac{1}{2}}(1+\mathbb{E}\|\vartheta\|^2_{L^2_x}).
\end{eqnarray*}
Letting
\[
H(\xi)=\int^{\xi}_0\eta'_{\varepsilon}(\zeta-u^{n}(t',y))a(\zeta)d\zeta,
\]
by integration by parts formula and (\ref{rr-5-1}), we have
\begin{eqnarray*}
  I_{4}
  &=&-\sum^M_{i=0}\mathbb{E}\int^{t}_{t'}\int_{D^{\lambda_i}_x}\int_{ D_y}(\nabla\cdot H(u^{n}(s,x)))\lambda_i(x)w(x)\rho^{\lambda_i}_{\gamma}(y-x)dy dxds\\
  &=& \sum^M_{i=0}\mathbb{E}\int^{t}_{t'}\int_{D^{\lambda_i}_x}\int_{ D_y}H(u^{n}(s,x))\cdot\nabla(\lambda_i(x)w(x)\rho^{\lambda_i}_{\gamma}(y-x))dy dxds\\
  &\lesssim& \gamma^{-1}\varepsilon^{-1}\|w\|_{W^{1,\infty}_x} \mathbb{E}\int^{t}_{t'}\int_{ D_x}(1+|u^{n}(s,x)|^{q_0+1})dxds\\
  &\lesssim& \gamma^{-1}\varepsilon^{-1}\|w\|_{W^{1,\infty}_x}(t-t')\Big[1+\mathbb{E}\sup_{t\in [0,T]}\|u^{n}(t)\|^{q_0+1}_{L^{q_0+1}_x}\Big]\\
  &\lesssim& \gamma^{-1}\varepsilon^{-1}\|w\|_{W^{1,\infty}_x}(t-t')\Big[1+\mathbb{E}\|\vartheta\|^{q}_{L^{q}_x}\Big].
\end{eqnarray*}
By Hypothesis (H3), it follows that
\begin{eqnarray*}
  I_5&\lesssim& \varepsilon^{-1}\gamma^{-1}\|w\|_{L^{\infty}_x}(t-t')\mathbb{E}\sup_{t\in [0,T]}(1+\|u^{n}(t)\|^2_{L^2_x})\\
  &\lesssim& \varepsilon^{-1}\gamma^{-1}(t-t')(1+\mathbb{E}\|\vartheta\|^2_{L^2_x}).
\end{eqnarray*}
Combining all the above estimates, we get
\begin{eqnarray*}
  K^n_{1}(\gamma)&\leq& 2\varepsilon\|w\|_{L^{\infty}_x}+K^n_2(\gamma)\\
  && +\gamma^{-1}\|w\|_{W^{1,\infty}_x}(t-t')^{\frac{1}{2}}(1+\mathbb{E}\|\vartheta\|^2_{L^2_x})
  +\gamma^{-1}\varepsilon^{-1}\|w\|_{W^{1,\infty}_x}(t-t')\Big[1+\mathbb{E}\|\vartheta\|^{q}_{L^{q}_x}\Big]
   \\
  &&
 +\varepsilon^{-1}\gamma^{-1}(t-t')\mathbb{E}(1+\|\vartheta\|^2_{L^2_x}).
\end{eqnarray*}
Thus, we conclude that
\begin{eqnarray*}\notag
&&\mathbb{E}\int_{D}|u^{n}(t,x)-u^{n}(t',x)|w(x)dx \\ \notag
 &\leq & 2\varepsilon\|w\|_{L^{\infty}_x} +\gamma^{-1}\|w\|_{W^{1,\infty}_x}(t-t')^{\frac{1}{2}}(1+\mathbb{E}\|\vartheta\|^2_{L^2_x})
  +\gamma^{-1}\varepsilon^{-1}\|w\|_{W^{1,\infty}_x}(t-t')\Big[1+\mathbb{E}\|\vartheta\|^{q}_{L^{q}_x}\Big]
   \\
  &&
 +\varepsilon^{-1}\gamma^{-1}(t-t')\mathbb{E}(1+\|\vartheta\|^2_{L^2_x})+2K^n_2(\gamma).
\end{eqnarray*}
Due to (\ref{s-6}), for any $\tau>0$, there exists a small positive constant $\gamma_0$, independent of $n$, such that
$K^n_2(\gamma_0)<\frac{\tau}{4}$.
For such $\tau$, we can choose small positive constants $\varepsilon_0$ and $\varsigma$ independent of $n$ such that for any $t,t'\in \mathcal{I}$ with $|t-t'|<\varsigma$ such that
\begin{eqnarray*}
&& 2\varepsilon_0\|w\|_{L^{\infty}_x} +\gamma^{-1}_0\|w\|_{W^{1,\infty}_x}\varsigma^{\frac{1}{2}}(1+\mathbb{E}\|\vartheta\|^2_{L^2_x})
  +\gamma^{-1}_0\varepsilon^{-1}_0\|w\|_{W^{1,\infty}_x}\varsigma\Big[1+\mathbb{E}\|\vartheta\|^{q}_{L^{q}_x}\Big]
   \\
  &&
 +\varepsilon^{-1}_0\gamma^{-1}_0\varsigma\mathbb{E}(1+\|\vartheta\|^2_{L^2_x})
 <\frac{\tau}{2}.
\end{eqnarray*}
Thus, we have proved that for any $\tau>0$, there exists $\varsigma>0$, independent of $n$, such that for every $t,t'\in \mathcal{I}$ with $|t-t'|<\varsigma$, it holds that
\begin{eqnarray*}
   \mathbb{E}\|u^{n}(t;\vartheta)-u^{n}(t';\vartheta)\|_{L^1_{w;x}}<\tau,
\end{eqnarray*}
which is the desired result (\ref{rr-9}). Taking $n\rightarrow \infty$ on (\ref{rr-9}) implies
\begin{eqnarray*}
 \mathbb{E}\|u(t;\vartheta)-u(t';\vartheta)\|_{L^1_{w;x}} \leq  \tau.
\end{eqnarray*}
As a result, $u: \mathcal{I}\rightarrow L^1_{\omega}L^1_{w;x}$ is uniformly continuous, hence it has a unique continuous extension on $[0,T]$.

\end{proof}

\begin{proof}[\textbf{Proof of Proposition \ref{prpo-1}}]
 For any $\vartheta\in L^{1}_{w;x}$, there exists a sequence $\{\vartheta_n\}_{n\geq 1}\subset L^{q}_{x} $ such that $\vartheta_n\rightarrow \vartheta$ in $L^{1}_{w;x}$. From Theorem \ref{thm-10}, we have $u(\cdot;\vartheta_n)\in C([0,T];L^1_{\omega}L^1_{w;x})$ for each $n\geq 1$. Furthermore, we deduce from Theorem \ref{thm-8} that for any $m>n\geq 1$,
 \[
 \sup_{t\in [0,T]}\mathbb{E}\|u(t;\vartheta_m)-u(t;\vartheta_n)\|_{L^1_{w;x}}\leq \mathbb{E}\|\vartheta_m-\vartheta_n\|_{L^1_{w;x}}.
 \]
Thus, $\{u(\cdot; \vartheta_n)\}_{n\geq 1}$ is a Cauchy sequence in $C([0,T];L^{1}_{\omega}L^{1}_{w;x})$, which yields that the limit $v(\cdot;\vartheta):=\lim_{n\rightarrow \infty}u(\cdot;\vartheta_n)$ exists in $C([0,T];L^{1}_{\omega}L^{1}_{w;x})$. Moreover,  using Theorem \ref{thm-8}, we see that the limit $v(\cdot;\vartheta)$ is independent of the choices of $\{\vartheta_n\}_{n\geq 1}$. Clearly, for $\vartheta\in L^{q}_{\omega}L^{q}_x$, we have $v(\cdot; \vartheta)=u(\cdot;\vartheta)$ in $C([0,T];L^{1}_{\omega}L^{1}_{w;x})$. Thus, $v$ is the unique continuous extension of $u$ on $C([0,T];L^{1}_{\omega}L^{1}_{w;x})$. Finally, (\ref{e-43}) follows easily by construction.

\end{proof}

\section{Ergodicity}\label{sec-1}
In this section, we will prove the main result Theorem \ref{thm-3}.
First, we obtain a polynomial decay for the difference of kinetic solutions of (\ref{P-19})-(\ref{P-19-2}) with different initial conditions.
\begin{prp}\label{prp-6}
Assume Hypotheses (H1)-(H3) are in force. Let $u(t;\vartheta)$ and $u(t;\tilde{\vartheta})$ be the kinetic solutions of (\ref{P-19})-(\ref{P-19-2}) with initial datums $\vartheta,\tilde{\vartheta}\in L^{q}_{\omega}L^{q}_{x}$.
Then for all $t\geq 0$,
  \begin{eqnarray}\label{e-18}
   \sup_{\vartheta,\tilde{\vartheta}\in L^{q}_{\omega}L^{q}_{x}} \mathbb{E}\|u(t;\vartheta)-u(t;\tilde{\vartheta})\|_{L^1_{w;x}}\leq C_{q_0}\|w\|^{q^*}_{L^{q^*}_x}t^{-\frac{1}{q_0}},
  \end{eqnarray}
 where the constant $C_{q_0}$ depends only on $q_0$ and $q^*=\frac{q_0+1}{q_0}$.
\end{prp}

\begin{proof}
By (\ref{e-19}) and using Hypothesis (H1), we deduce that
   \begin{eqnarray*}
  \mathbb{E}\|u(t;\vartheta)-u(t;\tilde{\vartheta})\|_{L^1_{w;x}}\leq \mathbb{E}\|u(s;\vartheta)-u(s;\tilde{\vartheta})\|_{L^1_{w;x}}-C_{q_0}\mathbb{E}\int^t_s\|u(r;\vartheta)-u(r;\tilde{\vartheta})\|^{q_0+1}_{L^{q_0+1}_{x}}dr,
  \end{eqnarray*}
  where the constant $C_{q_0}$ depends only on $q_0$.

  By H\"{o}lder inequality,
  \begin{eqnarray*}
  \mathbb{E}\|u(r;\vartheta)-u(r;\tilde{\vartheta})\|_{L^1_{w;x}}\leq \Big(\mathbb{E}\|u(r;\vartheta)-u(r;\tilde{\vartheta})\|^{q_0+1}_{L^{q_0+1}_{x}}\Big)^{\frac{1}{q_0+1}}\|w\|_{L^{q^*}_x},
  \end{eqnarray*}
  where $q^*=\frac{q_0+1}{q_0}$.
 Thus, we deduce that for any $0\leq s<t<T$,
    \begin{eqnarray}\notag
  &&\mathbb{E}\|u(t;\vartheta)-u(t;\tilde{\vartheta})\|_{L^1_{w;x}}\\  \label{e-22-1}
  &\leq & \mathbb{E}\|u(s;\vartheta)-u(s;\tilde{\vartheta})\|_{L^1_{w;x}}-C_{q_0}\|w\|^{-(q_0+1)}_{L^{q^*}_x}\int^t_s\Big(\mathbb{E}\|u(r;\vartheta)-u(r;\tilde{\vartheta})\|_{L^1_{w;x}}\Big)^{q_0+1}dr.
  \end{eqnarray}
  Set $f(t):=\mathbb{E}\|u(t;\vartheta)-u(t;\tilde{\vartheta})\|_{L^1_{w;x}}$. From Theorem \ref{thm-10}, we know that $f(t)$ is continuous on $[0,T]$. (\ref{e-22-1}) yields
  \begin{eqnarray*}
    f(t)-f(s)\leq -C_{q_0}\|w\|^{-(q_0+1)}_{L^{q^*}_x}\int^t_s f^{q_0+1}(r)dr,
  \end{eqnarray*}
  for all $0\leq s<t<T$.
Hence, by a comparison argument as in the proof of Theorem 3.8 in \cite{DGT20}, we obtain
  \begin{eqnarray*}
    \mathbb{E}\|u(t;\vartheta)-u(t;\tilde{\vartheta})\|_{L^1_{w;x}}\leq C_{q_0}\|w\|^{q^*}_{L^{q^*}_x}t^{-\frac{1}{q_0}},
  \end{eqnarray*}
 which is the desired result.
\end{proof}

As a consequence, we have the following
\begin{cor}\label{cor-5}
  For the unique continuous extension $v$ of $u$ given by Proposition \ref{prpo-1}, it holds that for any $t>0$,
  \begin{eqnarray}\label{z-3}
   \sup_{\vartheta,\tilde{\vartheta}\in L^{1}_{w;x}} \mathbb{E}\|v(t;\vartheta)-v(t;\tilde{\vartheta})\|_{L^1_{w;x}}\leq C_{q_0}\|w\|^{q^*}_{L^{q^*}_x}t^{-\frac{1}{q_0}},
  \end{eqnarray}
 with the constant $C_{q_0}$ depends only on $q_0$.
\end{cor}

For technical reasons,  we extend the time horizon to $-\infty$. We need the following notations. For the data $ \vartheta\in L^{q}_{\omega}L^{q}_x$ and $s>-\infty$, we denote by $u_s(\cdot;\vartheta)$ the kinetic solution of
\begin{eqnarray}\label{e-23}
\left\{
  \begin{array}{ll}
    \partial_t u_s(t,x;\vartheta)=div A(u_s(t,x;\vartheta))+\sum_{k\geq1}\sigma^k(x,u_s(t,x;\vartheta))\dot{\beta}_k(t)\quad {\rm{in}} \ \Omega\times D\times(0,T),\\
   u_s(s,x;\vartheta)=\vartheta \quad {\rm{in}} \ \Omega\times D,\\
   u_s(s,\cdot;\vartheta)=0 \quad {\rm{on}} \ \Omega\times \Sigma,
  \end{array}
\right.
\end{eqnarray}
for $t\geq s$, where we have extended $\beta_k(t)$ for $t<0$ by gluing at $t=0$ an independent Brownian motion evolving backwards in time. According to this new notation,   $u_0(\cdot;\vartheta)=u(\cdot;\vartheta)$.
The global well-posedness of (\ref{e-23}) for the case $s\neq 0$ can be obtained analogously as the case $s=0$.
In addition, for the mapping
  \[
  L^{q}_{\omega}L^{q}_x\ni \vartheta\mapsto u_s(\cdot; \vartheta)\in C([s,\infty);L^{1}_{\omega}L^{1}_{w;x}),
  \]
we use $v_s(\cdot;\cdot)$ to denote its unique continuous extension from $L^{1}_{w;x}$ to $C([s,\infty);L^{1}_{\omega}L^{1}_{w;x})$ as stated in Theorem 3.2.

The following result is the cocycle property of the dynamic generated by (\ref{e-23}).
\begin{prp}\label{prp-5}
For every $\vartheta\in L^{q}_x$ and $-\infty<s_1\leq s_2\leq t< T$, it holds true that $u_{s_1}(t;\vartheta)=u_{s_2}(t;u_{s_1}(s_2;\vartheta))$ in $L^1_{\omega}L^1_{w;x}$.
\end{prp}

\begin{proof}
Fix $s>0$, without loss of generality, we only need to prove the case $(s_1,s_2)=(0,s)$ for some $0<s<T$. Let $\vartheta\in L^{q}_x$ and $u(t;\vartheta)$ be a kinetic solution of (\ref{e-23}) on $[0,T)$. From Lemma \ref{lem-5}, we know that
for all $\varphi\in C^{\infty}_c(\mathbb{R}^d\times \mathbb{R})$, the kinetic solution $u(t;\vartheta)$ satisfies
 \begin{eqnarray}\notag
    &&-\int_D\int^N_{-N}\Psi_{\eta}f^+(t)\varphi d\xi dx+\int^t_s\int_D\int^N_{-N}\Psi_{\eta}f a(\xi)\cdot \nabla \varphi d\xi dxdr\\ \notag
          && +\int_D\int^N_{-N}\Psi_{\eta}f^+_s\varphi d\xi dx+\int^t_s\int_{\partial D}\int^N_{-N}\Psi_{\eta}(-a(\xi)\cdot \mathbf{n}) \tilde{f}\varphi d\xi d\sigma dr\\ \notag
          &=& -\sum_{k\geq 1}\int^t_s\int_D\int^N_{-N}\Psi_{\eta}(\xi)g_k(x,\xi)\varphi d\nu_{x,r}(\xi)dxd\beta_k(r)\\ \notag
          &&-\frac{1}{2}\int^t_s\int_D\int^N_{-N}\Psi_{\eta}(\xi)\partial_{\xi}\varphi G^2(x,\xi)d\nu_{x,r}(\xi)dxdr+\int_{(s,t]\times D\times (-N,N)}\Psi_{\eta}(\xi)\partial_{\xi}\varphi dm\\ \notag
          &&
          +\frac{1}{2}\int^t_s\int_D\int^N_{-N}\Big(\psi_{\eta}(N-\xi-\eta)-\psi_{\eta}(N+\xi-\eta)\Big)G^2(x,\xi)\varphi d\nu_{r,x}(\xi)dxdr\\ \label{eqq-30}
          && -\int_{(s,t]\times D\times (-N,N)}\Big(\psi_{\eta}(N-\xi-\eta)-\psi_{\eta}(N+\xi-\eta)\Big)\varphi dm \quad a.s.
 \end{eqnarray}
on $[s,T)$. Moreover, we claim that $u(s;\vartheta)\in L^{q}_{\omega}L^{q}_{x}$. Indeed, by
(\ref{eqq-31}), we have
\begin{eqnarray*}
\underset{0\leq t\leq T}{{\rm{ess\sup}}}\ \mathbb{E}\|u(t)\|^{q}_{L^{q}_x}\leq C_{q}.
\end{eqnarray*}
 By Theorem \ref{thm-10}, there exists a sequence $s_n\rightarrow s$ such that $u(s_n, x)\rightarrow u(s,x)$ for almost every $(\omega,x)$ and
\[
\sup_{n\geq 1}\mathbb{E}\|u(s_n)\|^{q}_{L^{q}_x}\leq C_{q}.
\]
Then, by Fatou's lemma, we get
\[
\mathbb{E}\|u(s)\|^{q}_{L^{q}_x}\leq \liminf_{n\rightarrow\infty}\mathbb{E}\|u(s_n)\|^{q}_{L^{q}_x}\leq\sup_{n\geq 1}\mathbb{E}\|u(s_n)\|^{q}_{L^{q}_x}\leq C_{q}.
\]
 Hence, $u(s;\vartheta)\in L^{q}_{\omega}L^{q}_{x}$. Now, we can apply the uniqueness of kinetic solutions to (\ref{eqq-30}) to conclude that $u(t;\vartheta)=u_s(t;u(s;\vartheta))$ in $L^1_{\omega}L^1_{w;x}$ for every $t\in [s,T)$.

\end{proof}

 Proposition \ref{prp-4} says that the mappings $P_t, t\geq 0$ define a Feller semigroup.

\begin{proof}[\textbf{Proof of Proposition \ref{prp-4}}]
After the preparations in Proposition 6.3, the proof now follows from standard arguments, see, e.g. Theorem 9.14 (or Theorem 9.8) in \cite{DPZ}. We omit the details.
\end{proof}

\vskip 0.4cm
Now, we are in a position to prove the polynomial mixing of $P_t$.
\begin{proof}[\textbf{Proof of Theorem \ref{thm-3}}]
For any $\vartheta \in L^{q}_{x}$, denote by $\eta_{s}(\vartheta)=u_s(0;\vartheta)$. By Proposition \ref{prp-5}, for $s_1\leq s_2\leq -1$, it follows that
  $\eta_{s_1}(\vartheta)=u_{s_2}(0;u_{s_1}(s_2;\vartheta))$ in $L^1_{\omega}L^1_{w;x}$. Hence,
\[
\eta_{s_2}(\vartheta)-\eta_{s_1}(\vartheta)=u_{s_2}(0;\vartheta)-u_{s_2}(0;u_{s_1}(s_2;\vartheta))
\]
   in $L^1_{\omega}L^1_{w;x}$. By Theorem \ref{thm-10} and Corollary \ref{cor-5}, we have
   \begin{eqnarray}
     \mathbb{E}\|\eta_{s_2}(\vartheta)-\eta_{s_1}(\vartheta)\|_{L^1_{w;x}}\lesssim \|w\|^{q^*}_{L^{q^*}_x}|s_2|^{-\frac{1}{q_0}},
   \end{eqnarray}
  which implies that $(\eta_s(\vartheta))_{s\leq -1}$ is a Cauchy sequence in $L^1_{\omega}L^1_{w;x}$. Hence, there exists a random variable $X(\vartheta)\in L^1_{\omega}L^1_{w;x}$ such that $\eta_s(\vartheta)\rightarrow X(\vartheta)$ in $L^1_{\omega}L^1_{w;x}$, as $s\rightarrow -\infty$.

  We claim that $X(\vartheta)$ is independent of the initial data $\vartheta$. Indeed, for any $\vartheta, \tilde{\vartheta}\in L^{q}_x$, by Proposition \ref{prp-6}, we have
   \begin{eqnarray}
     \mathbb{E}\|\eta_{s}(\vartheta)-\eta_{s}(\tilde{\vartheta})\|_{L^1_{w;x}}\lesssim \|w\|^{q^*}_{L^{q^*}_x}|s|^{-\frac{1}{q_0}}.
   \end{eqnarray}
  Then letting $s\rightarrow -\infty$, we have $X(\vartheta)=X(\tilde{\vartheta})$ in $L^1_{\omega}L^1_{w;x}$.

Let $X=X(0)$ and define $\mu=\mathbb{P}\circ X^{-1}\in \mathcal{M}_1(L^1_{w;x})$.
Next, we verify that $\mu$ is an invariant measure of $P_t$.
Denote by $P_{s,t}$ the semigroup associated to (\ref{e-23}) at time $t$, then $P_t=P_{0,t}$ for any $t\geq 0$. Keeping in mind that  $\eta_{-s}(0)\rightarrow X(0)$ in $L^1_{\omega}L^1_{w;x}$, when $s\rightarrow \infty$, we have
  \begin{eqnarray}\notag
    \int_{L^1_{w;x}}P_{0,t}F(\xi)\mu(d\xi)&=&
    \mathbb{E}P_{0,t}F(X(0))=\lim_{s\rightarrow \infty}\mathbb{E}P_{0,t}F(\eta_{-s}(0))\\ \notag
    &=&\lim_{s\rightarrow \infty}\mathbb{E}P_{0,t}F(u_{-s}(0;0))=\lim_{s\rightarrow \infty}\mathbb{E}P_{0,t}F(v_{-s}(0;0))\\ \notag
    &=&\lim_{s\rightarrow \infty}P_{-s,0}(P_{0,t}F)(0)=\lim_{s\rightarrow \infty}P_{-s,t}F(0)=\lim_{s\rightarrow \infty}P_{-(t+s),0}F(0)\\ \notag
    &=& \lim_{s\rightarrow \infty}\mathbb{E}F(v_{-(t+s)}(0;0))= \lim_{s\rightarrow \infty}\mathbb{E}F(u_{-(t+s)}(0;0))\\
     \label{e-46}
    &=& \lim_{s\rightarrow \infty}\mathbb{E}F(\eta_{-(t+s)}(0))=\mathbb{E}F(X(0))
= \int_{L^1_{w;x}}F(\xi)\mu(d\xi),
  \end{eqnarray}
  for every $F\in C_b(L^1_{w;x})$, here we have used the Feller property of $P_{0,t}$, $P_{s,r}P_{r,t}=P_{s,t}$ for any $s<r<t$, as well as $P_{s,t}=P_{s+r,t+r}$ for every $r\in \mathbb{R}$.
(\ref{e-46}) shows that  $\mu$ is an invariant measure of $P_t$ on $L^1_{w;x}$.

From Corollary \ref{cor-5}, we know that for any $F\in Lip(L^1_{w;x})$ and $\vartheta, \tilde{\vartheta}\in L^1_{w;x}$,
\begin{eqnarray}\notag
  |P_tF(\vartheta)-P_tF(\tilde{\vartheta})|&=&|\mathbb{E}F(v(t;\vartheta))-\mathbb{E}F(v(t;\tilde{\vartheta}))|\\ \notag
  &\leq& \|F\|_{Lip(L^1_{w;x})}\mathbb{E}\|v(t;\vartheta)-v(t;\tilde{\vartheta})\|_{L^1_{w;x}}\\
\label{z-4}
  &\lesssim &\|F\|_{Lip(L^1_{w;x})}\|w\|^{q^*}_{L^{q^*}_x}|t|^{-\frac{1}{q_0}},
\end{eqnarray}
which implies that any two invariant measures $\mu$ and $\tilde{\mu}$ on $L^1_{w;x}$ coincide.

 Finally, by utilizing (\ref{z-4}), it follows that for all $t>0$,
  \begin{eqnarray*}
  \Big|P_tF(\vartheta)-\int_{L^1_{w;x}}F(\xi)\mu(d\xi)\Big|&=&|P_tF(\vartheta)-\int_{L^1_{w;x}}P_tF(\xi)\mu(d\xi)|\\
  &=&\Big|P_tF(\vartheta)-\mathbb{E}P_t F(X(0))\Big|\\
  &\leq & \mathbb{E}|P_tF(\vartheta)-P_t F(X(0))|\\
  &\lesssim& \|F\|_{Lip(L^1_{w;x})}\|w\|^{q^*}_{L^{q^*}_x}|t|^{-\frac{1}{q_0}}.
\end{eqnarray*}
Taking the supremum over $\|F\|_{Lip(L^1_{w;x})}\leq 1$ and $\vartheta\in L^1_{w;x}$, we get the desired result.

\end{proof}

\vskip 0.2cm
\noindent{\bf  Acknowledgements}\quad  This work is partly supported by National Natural Science Foundation of China (No. 11671372, 11721101, 11801032, 11931004, 11971456), Key Laboratory of Random Complex Structures and Data Science, Academy of Mathematics and Systems Science, CAS (No. 2008DP173182) and Beijing Institute of Technology Research Fund Program for Young Scholars.

\def\refname{ References}

\end{document}